\documentclass[11pt]{article}
 \usepackage{amssymb}
\usepackage{tikz}
\usepackage{tikz-cd}
\usetikzlibrary{matrix, arrows,  decorations.markings,  patterns,  plotmarks, decorations.pathreplacing}
\usepackage{amsmath}
\usepackage{amsthm}
\usepackage[citestyle=alphabetic,bibstyle=alphabetic,backend=bibtex,maxbibnames=10,minbibnames=5]{biblatex}

\usepackage{enumerate}
\usepackage{subcaption}
\usepackage{verbatim}
\usepackage[left=1in,top=1in,right=1in]{geometry}
\usepackage{hyperref}
\usepackage{cleveref}

\newtheorem{theorem}{Theorem}[subsection]
\newtheorem*{theorem*}{Theorem}
\newtheorem{lemma}[theorem]{Lemma}
\newtheorem*{lemma*}{Lemma}
\newtheorem{prop}[theorem]{Proposition}
\newtheorem{corollary}[theorem]{Corollary}
\newtheorem{definition}[theorem]{Definition}
\newtheorem{notation}[theorem]{Notation}
\newtheorem{claim}[theorem]{Claim}
\newtheorem*{claim*}{Claim}

\newtheorem{remark}[theorem]{Remark}
\newtheorem{assumption}[theorem]{Assumption}

\crefname{claim}{claim}{claims}
\crefname{prop}{proposition}{propositions}
 \usepackage{amssymb}
\graphicspath{{./figures/}}
\newcommand{\ZZ}{\mathbb{Z}}
\newcommand{\CC}{\mathbb{C}}
\newcommand{\RR}{\mathbb{R}}
\newcommand{\NN}{\mathbb{N}}

\newcommand{\into}{\hookrightarrow}
\newcommand{\tensor}{\otimes}

\newcommand{\End}{\text{End}}

\newcommand{\mathcolorbox}[2]{\colorbox{#1}{$\displaystyle #2$}}
\DeclareMathOperator{\Crit}{Crit}
\DeclareMathOperator{\ind}{ind}
\DeclareMathOperator{\id}{id}

\DeclareMathOperator{\res}{res}
\DeclareMathOperator{\width}{width}
\DeclareMathOperator{\grad}{grad}
\DeclareMathOperator{\ord}{ord}

\DeclareMathOperator{\cyl}{cyl}
\newcommand{\val}{\nu}
\newcommand{\fil}{\ell}
\newcommand{\ver}{V}
\DeclareMathOperator{\syz}{val}
\DeclareMathOperator{\Flux}{Flux}

\newcommand{\ot}{\leftarrow}
\renewcommand{\Im}{\text{Im}}
\renewcommand{\Re}{\text{Re}}

\DeclareMathOperator{\Fuk}{Fuk}

\DeclareMathOperator{\Coh}{Coh}

\newcommand{\CF}{{CF^\bullet}}

\newcommand{\CM}{{CM^\bullet}}

\newcommand{\QH}{{QH_\bullet}}
\newcommand{\QC}{{QC_\bullet}}
\renewcommand{\jmath}{\iota} 

\begin{document}
\newcommand{\Addresses}{{
  \bigskip
  \footnotesize

  \noindent J.~Hicks, \textsc{School of Mathematics,  University of Edinburgh}\par\nopagebreak
  \noindent \textit{E-mail address}: \texttt{jeff.hicks@ed.ac.uk}

  \medskip

}}

\title{\normalsize \textbf{Wall-Crossing from Lagrangian Cobordisms}}
\author{\normalsize  Jeff Hicks}
\date{}

\tikzset{every picture/.style=thick}
\maketitle{}
\begin{abstract}    Biran and Cornea showed that monotone Lagrangian cobordisms give an equivalence of objects in the Fukaya category.
	However, there are currently no known non-trivial examples of monotone Lagrangian cobordisms with two ends. 
	We look at an extension of their theory to the pearly model of Lagrangian Floer cohomology and unobstructed Lagrangian cobordisms. 
	In particular, we examine the suspension cobordism of a Hamiltonian isotopy and the Haug mutation cobordism between mutant Lagrangian surfaces.
	In both cases, we show that these Lagrangian cobordisms can be unobstructed by bounding cochain and additionally induce an $A_\infty$ homomorphism between the Floer cohomology of the ends.
	This gives a first example of a two-ended Lagrangian cobordism giving a non-trivial equivalence of Lagrangian Floer cohomology. 
	
	A brief computation is also included which shows that the incorporation of bounding cochain from this equivalence accounts for the ``instanton-corrections'' considered by \cite{auroux2007mirror,pascaleff2017wall,rizell2018refined} for the wall-crossing formula between Chekanov and product tori in $(\CC^2)\setminus \{z_1z_2=1\}$. 
	
	We additionally prove some auxiliary results that may be of independent interest. These include
	a weakly filtered version of the Whitehead theorem for $A_\infty$ algebras and 
	an extension of Charest-Woodward's stabilizing divisor model of Lagrangian Floer cohomology to Lagrangian cobordisms. \end{abstract}
\section{Introduction}
	\label{sec:introduction}
	\subsection*{Wall-Crossing...}
The wall-crossing phenomenon for Lagrangian submanifolds is an observation that the count of holomorphic disks with boundary on a family of Lagrangian submanifolds need not be continuous over the family.
The count of these disks and how this count changes over families play an important role in describing the space of Lagrangian submanifolds up to Hamiltonian isotopy.

The moduli space of Lagrangian branes can be understood locally by constructing coordinate charts. 
A Lagrangian brane is a Lagrangian submanifold equipped with a unitary local system.
Nearby any Lagrangian, the \emph{flux of a Lagrangian isotopy} builds a local $(\CC^*)^k$ chart, where $k=\dim(H^1(L; \RR))$. 
An expectation which has been proven in good examples (such as \cite{auroux2007mirror,palmer2019invariance}) is that the \emph{open Gromov-Witten (OGW) potential}, which records an area-weighted count of the holomorphic disks with boundary on a given Lagrangian $L$, is a holomorphic map in these coordinates.
Both the symplectic area of these disks and the flux of an isotopy are complexified by the unitary local system on the Lagrangian submanifold.

Locally, the flux-charts and the OGW potential are holomorphic. 
However, neither of these can consistently provide coordinates or functions globally.
An inconsistency occurs when a holomorphic disk ``bubbles'' over a Lagrangian isotopy, causing a discontinuity in the OGW potential.
To consistently construct coordinates on the moduli space of Lagrangian submanifolds, one must incorporate ``instanton corrections'' to the count of holomorphic disks and flux computation. 
In \cite{kontsevich2001homological,auroux2007mirror} these corrections were phrased in terms of a wall-crossing formula. 
In this paradigm, the moduli space of Lagrangian submanifolds is divided into chambers of Lagrangians which do not bound Maslov index 0 disks.
Separating these chambers are ``walls'' consisting of the Lagrangian submanifolds which bound Maslov index 0 disks.
To transition from coordinates on one chamber to another, one computes a wall-crossing formula given by the count of Maslov index 0 disks, which appropriately modifies the OGW potential and flux computation.

A particular example of wall-crossing occurs when a monotone Lagrangian torus $L$ bounds a Lagrangian disk $D$. 
For such pairs there exists another Lagrangian, called the mutation $\mu_D(L)$, lying in a different chamber. 
\cite{pascaleff2017wall,palmer2019invariance,rizell2018refined} explicitly compute the wall-crossing formula between these two chambers. 
Notably, this computation gives a wall-crossing formula between two Lagrangians which are \emph{not} Hamiltonian isotopic. 

The geometric justification for these wall-crossing transformations comes from the homological mirror symmetry conjecture of \cite{kontsevich1994homological}. 
Associated to a symplectic space $X$ is the Fukaya category $\Fuk(X)$ whose objects are Lagrangian submanifolds. 
The conjecture predicts that to a Calabi-Yau manifold $X$, there exists a mirror Calabi-Yau manifold $\check X$ so that the symplectic geometry of $X$ as recorded by $\Fuk(X)$ is interchanged with the complex geometry of $\check X$ as recorded by $D^b\Coh(\check X)$. 
One way to recover the space $\check X$ is to study the moduli of points on $\check X$.
A useful perspective comes from the SYZ conjecture \cite{strominger1996mirror}, which presents mirror spaces as dual Lagrangian torus fibrations. 
From this viewpoint, a candidate mirror to the skyscraper sheaf of a point in $\check X$ is a Lagrangian torus fiber of the SYZ fibration $X\to Q$.

As Hamiltonian isotopic Lagrangian submanifolds give quasi-isomorphic objects of the Fukaya category, the moduli of objects in the Fukaya category is a good proxy for the Hamiltonian isotopy classes of Lagrangian submanifolds.  
This Fukaya-categorical interpretation has the advantage that the incorporation of ``instanton corrections'' naturally arises in the construction of the Fukaya category. 
The presence of holomorphic disks with boundary on a Lagrangian $L$ complicates the construction of the Fukaya category considerably.
To account for the presence of holomorphic disks, one must equip each Lagrangian submanifold with additional data. 
The objects of the Fukaya category are pairs $(L, b)$, where $b$ is a ``bounding cochain''  deforming $L$ as an object of the category.
This deformation encodes an algebraic cancellation of holomorphic disks with boundary on $L$.
Hamiltonian isotopy still produces an equivalence of such pairs, where disk bubbling is recorded by modifications of the bounding cochain.
In the examples of crossing a wall, Lagrangians $(L, 0)$ in one chamber are equivalent to $(L', b')$ in the second chamber. 
The perspective of \cite{fukaya2010lagrangian} is that the non-trivial deformation $b'$ should be considered as the correction which occurs in the wall-crossing formula.

\subsection*{... and Lagrangian Cobordisms}
In the previous discussion, we focused on Hamiltonian isotopy as the geometric equivalence relation on objects of the Fukaya category. 
Another such equivalence relation was exhibited by  \cite{biran2013fukayacategories}, who proved that Lagrangian submanifolds related by monotone Lagrangian cobordism are equivalent objects in the Fukaya category.
As every Hamiltonian isotopy gives an example of a Lagrangian cobordism, this indeed generalizes the previously considered equivalences.
However, the monotonicity condition precludes the existence of Maslov-index 0 disks, so examples considered by \cite{biran2013fukayacategories} do not realize the wall-crossing phenomenon.

A natural extension is to consider Lagrangian cobordisms equipped with bounding cochains.
Such a cobordism $(K, b)$ is predicted to yield equivalences in the Fukaya category between the ends $(L^+, b|_{L^+})$ and $(L^-, b|_{L^-})$. 
One interesting example of a Lagrangian cobordism is the \emph{mutation cobordism} constructed in \cite{haug2015lagrangian}, which relates mutant Lagrangians.
This is an example of a non-monotone Lagrangian cobordism, as mutant Lagrangians $L, \mu_D(L)$ are generally non-isomorphic as objects of the Fukaya category. 	\subsection{Results}
	The goal of this paper is to extend our understanding of wall-crossing between Hamiltonian isotopic Lagrangians to wall-crossing between cobordant Lagrangians.
We use the \emph{pearly Floer complex} $\CF(L)$ as a receptacle for counting holomorphic disks.
This is a deformation of the Morse complex $\CM(L)$ constructed by inserting holomorphic disks with boundary on $L$ into the flow lines of the Morse function.
$\CF(L)$ is a filtered $A_\infty$ algebra.
We first prove that the pearly Floer complex of a Lagrangian cobordism is compatible with the pearly Floer complex of the ends.
\begin{theorem}[Restatement of \cref{cor:projectionsareainfinity}]
    The projections of the Floer cohomology of a Lagrangian cobordism to that of its ends $\beta^\pm: \CF(K)\to \CF(L^\pm)$ are $A_\infty$ homomorphisms.
\end{theorem}
We employ these new Floer results techniques to extend  \cite{biran2013fukayacategories} to the non-monotone and cylindrical setting.
\begin{theorem}[Paraphrasing \cref{thm:cylindricityofcobordism}]
    Suppose that $K: L^-\rightsquigarrow L^+$ is a Lagrangian cobordism diffeomorphic to $L^+\times \RR$. 
    Then there is a homotopy equivalence of filtered $A_\infty$ algebras $\Theta_K: \CF(L^-)\to \CF(L^+)$. 
\end{theorem}
As every Hamiltonian isotopy of Lagrangians produces a cylindrical suspension Lagrangian cobordism realizing that isotopy, this generalizes invariance of Floer cohomology of  Hamiltonian isotopic Lagrangians to the unobstructed setting.
The method of proof differs from the Hamiltonian isotopy setting and uses that $\CF(K)$ is the deformation of the mapping cylinder $\CM(K)$.
In the setting where $L^-$ is tautologically unobstructed, we show that $K$ is unobstructed by a bounding cochain whose restriction to $L^-$ is trivial but whose restriction to $L^+$ may not be.

The first non-monotone and non-cylindrical example that we look at is the mutation cobordism. 
We similarly prove that this cobordism gives an equivalence. 
\begin{theorem}[Paraphrasing \cref{thm:wallcrossingcobordism}]
    Let $K_{\mu^\epsilon_D}: L\rightsquigarrow \mu^\epsilon_D(L)$ be a mutation cobordism satisfying \cref{cond:isolatedmutation}.
    Then there exists a deforming cochain $b\in \CF(K_{\mu^\epsilon_D})$ so that the $\CF_b(K_{\mu^\epsilon_D})$ is an $A_\infty$ mapping cocylinder, yielding a map of filtered $A_\infty$ algebras:
    \[\Theta_{\mu^\epsilon_D}: \CF_{b^-}(\mu_D^\epsilon( L))\to \CF_{b^+}(L).\]
\end{theorem}
The proof is more difficult than the cylindrical case, as $\CF(K)$ is \emph{not} the deformation of a mapping cylinder.
To show that $K$ is unobstructed, we must construct a bounding cochain for $K_{\mu^\epsilon_D}$ accounting for an interesting holomorphic disk with boundary on $K_{\mu^\epsilon_D}$. 
The main observation for mutation cobordisms is that portion of the pearly differential on $\CF(K)$ arising from this holomorphic disk exactly cancels out the non-cylindrical topology of $K$.

We verify in a computation included in the appendix that this theorem can be applied to the setting of monotone Chekanov and product tori in $(\CC^2)\setminus\{z_1z_2=1\}$.
In this example, the bounding cochain on the Chekanov and product tori inherited from the Lagrangian cobordism corrects the cohomology of these mutant Lagrangians. This correction matches the wall-crossing transformations constructed in \cite{auroux2009special}.
This gives (to our knowledge) the first example of a non-cylindrical two-ended Lagrangian cobordism yielding an equivalence in the Fukaya category.
This also gives a Fukaya category interpretation to the wall-crossing result of \cite{pascaleff2017wall}, and the invariance of immersed Floer cohomology in this setting (as computed by \cite{palmer2019invariance}). 
Our result can be considered as an application of the \cite{fukaya2010lagrangian,biran2013fukayacategories} viewpoint on wall-crossing to non-Hamiltonian isotopic Lagrangian submanifolds.

Finally, we include some auxiliary results: a weakly filtered version of the Whitehead theorem (\cref{thm:curvedhtt}), definition of domain and label dependent perturbation systems for the pearly model (\cref{subsec:prebrokentrees}), construction of stabilizing divisors for Lagrangian cobordisms (\cref{lem:weaklystabilizedcobordism,lem:estabilizedcobordism}), and compatibility of the pearly model of a Lagrangian cobordism with its ends (\cref{assum:pearlycompatibility}). 	\subsection{Structure of the paper}
	
\Cref{sec:background} discusses the pearly model of Floer cohomology in the setting of Lagrangian cobordisms. We review some existing versions of the pearly model in \cref{subsec:floerbackground}, as well as provide an outline extending the pearly model to the non-compact setting. We summarize the relevant results from the appendices which are required for the remainder of the paper.  The first is a way to construct perturbation datum for the purpose of computing specific structure constants for the pearly model. The second is the extension of the stabilizing divisor pearly model of \cite{charest2015floer} to the setting of Lagrangian cobordisms. The proofs of these results are delayed until \cref{subsec:prebrokentrees,app:cobordismstabilizingdivisor}.

In \cref{sec:continuation}, we look at a toy model where we recover invariance of $\CF(L)$ under Hamiltonian isotopy using non-monotone Lagrangian cobordisms.
The core of the proof is to show that the Lagrangian cobordism gives an example of an $A_\infty$ mapping cocylinder, from which a continuation map can be recovered by the application of \cref{thm:cylfrommap}.
The cobordism is our first example of a non-monotone cobordism giving an equivalence of objects in the Fukaya category, and the method of proof provides a road map to the more general case considered later.

\Cref{sec:cobordismandmutations} contains the main theorem of the paper. 
We first introduce the Lagrangian mutation cobordism from \cite{haug2015lagrangian} in \Cref{subsec:haugcobordism}.
We take some care to work out the parameterization of this cobordism explicitly in the setting of $\CC^2\setminus \{z_1z_2=1\}$.
This care is necessary to later characterize the kinds of holomorphic disks which can appear with boundary on this cobordism, and to compute explicitly the flux swept between the ends of the Lagrangian cobordism. 
\Cref{subsec:generalmutations} looks at the mutation cobordism constructed from a mutation pair -- a Lagrangian $L$ bounding a Lagrangian disk $D$. 
The remainder of the section is spent proving \cref{thm:wallcrossingcobordism}, which shows that this mutation cobordism $K_{\mu_D}$, when equipped with an appropriate bounding cochain, gives an equivalence in the Fukaya category. 
The idea of proof is very similar to \cref{sec:continuation}, in that we show that the mutation cobordism $K_{\mu_D}$ has Floer cohomology of a mapping cocylinder.
This is complicated by the fact that the underlying topology of $K_{\mu_D}$ is not that of a cylinder.  We outline the proof here. 
\begin{itemize}
    \item \Cref{subsub:Morsefunction}: we construct a  Morse function for $K_{\mu_D}$ and point out how this cobordism fails to be topological cylinder between its ends. 
    \item \Cref{subsub:exceptionalholomorphicdisk}: we find a holomorphic disk with boundary on $K_{\mu_D}$. 
    \item \Cref{subsub:removingcurvature}: we show that this holomorphic disk contributes to $m^0_{K_{\mu_D}}$ in such a way that its lowest order contributions can be canceled out by a deforming cochain $d_\epsilon$. 
    \item \Cref{subsub:invertingends}: we show that there exists a homotopy between $\CF_{d_\epsilon}(K_{\mu_D})$ and its left end $\CF_{d_\epsilon|_{L^-}}(L^-)$ which uses the component of the deformed differential arising from the holomorphic disk. 
    \item \Cref{subsub:applinghtt}: we present $K_{\mu_D}$ as a weakly filtered $A_\infty$ mapping cocylinder (\cref{def:mappingcylinder}) using the homotopy constructed above and conclude the existence of a continuation map. 
\end{itemize}

In \Cref{app:mutationexamples} we look at an application of \cref{thm:wallcrossingcobordism} to the Chekanov and product tori in   $\CC^2\setminus\{z_1z_2=1\}$.
In addition to showing that these objects have the same Floer cohomology (and conjecturally correspond to the same object of the Fukaya category), the bounding cochains inherited from the Lagrangian cobordism deform the local Flux coordinates on the moduli space of Chekanov and product tori.
Assuming convergence of sums taken with $\CC$ coefficients and a multiple cover formula for Maslov index 0 disks, these coordinate changes are shown to match the computations made in \cite{auroux2007mirror}.

We additionally include three appendices with auxiliary results. 
\Cref{app:homological} states some known facts for filtered $A_\infty$ algebras and morphism of $A_\infty$ algebras.
It also includes the statements of \cref{thm:curvedhtt} and \cref{thm:cylfrommap}.
The first theorem is an extension of the homotopy transfer principle from \cite{kadeishvili1980homology} or filtered $A_\infty$ Whitehead's theorem \cite{fukaya2010lagrangian} to the setting of weakly filtered $A_\infty$ homotopies.
The second is a proof of existence and characterization of mapping cocylinders for filtered $A_\infty$ algebras.

\Cref{subsec:prebrokentrees} extends the construction of \cite{charest2015floer} to use perturbations that depend on the domain and the labels of the domain. The main application, \cref{lemma:geometricflowlines}, shows that under strong conditions we can use a non-perturbed almost complex structure to compute some of the structure coefficients for the pearly model. 

\Cref{app:cobordismstabilizingdivisor} adapts the construction of \cite{charest2015floer} to Lagrangian cobordisms. The main results are \cref{lem:weaklystabilizedcobordism,lem:estabilizedcobordism} which constructs Donaldson divisors of large degree for Lagrangian cobordisms. The Donaldson divisors can be made weakly stabilizing. Furthermore, the divisors and weakly stabilized almost complex structures are split (across the product $X\times \CC$) outside of a compact set. 
This allows us to construct a pearly model for Lagrangian cobordisms (\cref{cor:pearlymodelexistence}) and obtain compatibility between the pearly model of a Lagrangian cobordism and that of its ends (\cref{cor:projectionsareainfinity}).

 	\subsection{Notation}
	\label{subsec:notation}
We introduce here several pieces of notation that are used throughout the paper. 
\begin{itemize}
    \item Let $A_1, \ldots, A_k$ be vector spaces. The component-wise projection is the map \[\pi^{k_1|\cdots|k_i}:(A_1\oplus \cdots\oplus A_i)^{\tensor k_1+\cdots + k_i}\to A_1^{\tensor k_1}\tensor \cdots \tensor A^{\tensor k_i}_i.\] 
    \item Let $A=A^-\oplus A^0\oplus A^+$ and let $m^k: A^{\tensor k}\to A$  be a multilinear map over field $\mathbb F$.
    When we want to restrict the domain and codomain, we'll write  
    \[m^{\ell_1\cdots \ell_k;\ell_0}:\bigotimes_{i=1}^k A^{\ell_i}\to A^{\ell_0}\]
    where  $\ell_i$ are chosen from the labels $\{+, -, 0\}$.
    In this notation, $m^{;\ell}:\mathbb F\to A^\ell$ denotes an element in $A^\ell$. 
    \label{item:multilinearnotation}
    \item Let $f: X^+\to X^-$ be a map of topological spaces. 
    The mapping cylinder is the topological space $\cyl(f)= (X^+\times I)\sqcup_{(x^+\times \{1\}\sim f(x^+))}X^-$. 
    This comes with inclusion maps $i_{X^\pm}: X^\pm\to \cyl(f)$.
    Additionally, there is a pushforward map $\pi: \cyl(f)\to X^-$. 
    We denote the maps on the cochains by 
    \begin{align*}
        \beta^\pm:=&(i_{X^\pm})^*:C^*(\cyl(f))\to C^*(X^\pm)\\
        \alpha^-:=&(\pi)^*:C^*(X^-)\to C^*(\cyl(f)).
    \end{align*}
    \item All symplectic manifolds considered are rational and aspherical. Unless specifically noted, all Lagrangian submanifolds considered are spin, graded, rational, and embedded.
\end{itemize} 	\subsection{Acknowledgements}
	First and foremost, I would like to thank my advisor Denis Auroux for his patience and explanations during my graduate studies at UC Berkeley.
I would also like to thank Paul Biran and Luis Haug, with whom I had many very productive conversations furthering my understanding of the anti-surgery cobordism. 
Additionally, I'd like to acknowledge Andrew Hanlon for the helpful comments on the introduction to this paper, as well as  Ailsa Keating, Nick Sheridan, and Chris Woodward for several valuable discussions.
This paper further benefited from the very thoughtful comments of an anonymous reviewer addressing expositional and mathematical gaps from an earlier version of this manuscript, especially with regards to adapting the stabilizing divisor approach of the pearly model to the setting of Lagrangian cobordisms.
Finally, I would also like to thank ETH Z\"urich for their hospitality in Fall 2018 during which a portion of this work was completed. 

This work was partially supported by:
NSF grants DMS-1406274 and DMS-1344991; 
Simons Foundation grant (\#\,385573, Simons Collaboration on Homological Mirror Symmetry); 
EPSRC Grant (EP/N03189X/1, Classification, Computation, and Construction: New Methods in Geometry); and
G:(EU-Grant)850713.

\section{Background}
	\label{sec:background}
	\subsection{Lagrangian cobordisms}
	\label{subsec:lagrangianbackground}
	
\label{subsec:lagrangiancobordism}
Lagrangian cobordisms are Lagrangian submanifolds in $X\times \CC$ which interpolate between Lagrangian submanifolds of $X$.
In addition to providing an interesting relation on Lagrangian submanifolds of $X$, cobordant Lagrangian submanifolds are frequently equivalent objects of the Fukaya category.

\begin{definition}[\cite{arnol1980lagrange}]
    Let $L^-, L^+$ be Lagrangian submanifolds of $X$. 
    A \emph{two-ended Lagrangian cobordism} between $L^-$ and  $L^+$ is a proper Lagrangian submanifold $K\subset X\times \CC$ which satisfies the following conditions:
    \begin{itemize}
        \item \emph{Fibered over ends:} There exist constants $t^-< t^+ $ such that 
        \[
            K\cap \{(x, z)\;:\; \Re(z)\geq t^+-\epsilon\}= L^+\times \RR_{\geq t^+-\epsilon}
        \]
        \[
            K\cap \{(x, z)\;:\; \Re(z)\leq t^-+\epsilon\}= L^-\times \RR_{\leq t^-+\epsilon}
        \]

        \item\emph{Compactness:}  There exists a constant $s>0$ so that the projection $\pi_{\jmath\RR}: K\to \jmath \RR \subset \CC$ is contained within an interval $[-\jmath s, \jmath s]$.
    \end{itemize}
    We denote such a cobordism $K:L^-\rightsquigarrow L^+$.
    \label{def:cobordism}
\end{definition}
We visualize a Lagrangian cobordism by drawing its projection to the $\CC$ parameter, as in \cref{fig:bottlenecks}. 
There is a generalization of this definition to cobordisms with multiple ends, which requires that the Lagrangian cobordism fibers over rays of fixed argument outside of a compact subset of $\CC$.
A specialization of the result of \cite{biran2013fukayacategories} proves that the equivalence relation of monotone cobordance descends to the Fukaya category.

\begin{theorem}
    Suppose that $K: L^-\rightsquigarrow L^+$ is a monotone Lagrangian cobordism.
    Then $L^-$ and $L^+$ are equivalent objects of the Fukaya category. 
\end{theorem}

In \cite{biran2013fukayacategories}, it is stated that this result is expected to hold in a more general setting where the Lagrangians are equipped with the additional data of orientations, local systems, bounding cochains, etc.
We look to understand this result in the context of pearly algebras of cobordant Lagrangians, where the Lagrangian cobordism may be non-monotone but unobstructed by bounding cochain.  

\begin{definition}
    \label{def:admissiblelagrangiancobordism}
    Let $(L^-, h^-)$ and $(L^+, h^+)$ be two Lagrangians equipped with choices of Morse function, spin structure, and grading.  
    An \emph{admissible Lagrangian cobordism} between these is the data of $(K, h)$ where
    \begin{itemize}
        \item
              $K$ is a Lagrangian cobordism between $L^-$ and $L^+$;
        \item
              $h:K\to \RR$ is a Morse function; and
        \item We have a spin structure and grading for $K$ compatible with those of the ends.
    \end{itemize}
    The Morse function $h$ is required to satisfy the following compatibility conditions:
    \begin{enumerate}
        \item The Morse flow restricted to the fibers above real coordinates $t^-$ and $t^+$ are determined by $h^\pm$,
    \begin{align*}
        \grad h|_{\pi_\RR^{-1}(t^-)} = (\grad h^-, 0) && \grad h|_{\pi_\RR^{-1}(t^+)}=(\grad h^+, 0) 
    \end{align*}
        \item The fibers above $t^\pm$ are perturbation of Morse-Bott maximums in the sense that at points $q\in K$,
        \begin{align*}
             dt(-\grad h)|_{\pi_\RR(q)<t^-}<0& &dt(-\grad h)|_{t^- <\pi_\RR(q) < t^-+\epsilon}>0\\
             dt(-\grad h)|_{\pi_\RR(q)>t^+}>0& &dt(-\grad h)|_{t^+-\epsilon <\pi_\RR(q) < t^+}<0.
        \end{align*}
\end{enumerate}
\end{definition}
One way to construct an admissible Morse function for a Lagrangian cobordism is to consider the function $\tilde h$ for $K$  as drawn in \cref{fig:bottlenecks} which is only dependent on $t$, and then take a Morsification which is constructed in the regions around $t^\pm$ using a $h^\pm$-perturbation. 
The function $\tilde h$ is Morse-Bott near the ends, and we call critical submanifolds $L^\pm\times\{t^\pm\}$ the bottlenecks of the cobordism.
\begin{figure}
    \centering
    \begin{tikzpicture}
\fill[gray!20]  (-1.5,3) rectangle (4.5,0.5);
    \draw[fill=gray!50] (-1.5,1.5) .. controls (-1,1.5) and (-0.5,1.5) .. (0,1.5) .. controls (0.5,1.5) and (0.5,1) .. (1,1) .. controls (1.5,1) and (1.5,1.5) .. (2,1.5) .. controls (2.5,1.5) and (3,1.5) .. (4.5,1.5) .. controls (3,1.5) and (2.5,1.5) .. (2,1.5) .. controls (1.5,1.5) and (1.5,2) .. (1,2) .. controls (0.5,2) and (0.5,1.5) .. (0,1.5) .. controls (-0.5,1.5) and (-1,1.5) .. (-1.5,1.5);

    \draw[->] (-1.5,3.5) -- (-1.5,5.5);
    \draw[->] (-1.5,3.5) -- (4.5,3.5);
    \draw (-1,5) .. controls (-0.75,5.25) and (-0.5,5.25) .. (0,5);
    \draw (1.6,4.25) .. controls (2.1,4) and (2.35,4) .. (2.6,4.25) .. controls (2.85,4.5) and (3.1,4.5) .. (3.35,4.25);
    \draw (0,5) -- (1.6,4.25);
    \draw (-1,5) -- (-1.5,4.5);
    \node at (-1.5,5.75) {$\tilde h(t)$};
    \node at (5.25,3.5) {$t=\pi_{i \mathbb R}$};
    \node at (1,1.5) {$\pi_\mathbb{C}(K)$};
    
    \draw[->, dotted] (-0.65,5) -- (-0.65,1.75);
    \draw[->,dotted] (2.975,4.25) -- (2.975,1.75);
    \node[right, above] at (-0.5,3.5) {$t^{-}$};
    
    \node[right, above] at (3,3.5) {$t^{+}$};
    
    \draw[dotted] (-0.65,-1) -- (-0.65,1.2);
    \draw[dotted] (3,-1) -- (3,1.2);
\node at (-0.5,-1.5) {$\text{CM}^\bullet(L^-, h^-)$};
\node at (1.7,-0.5) {$\text{CM}^\bullet(K, h)$};
\node at (3.1,-1.5) {$\text{CM}^\bullet(L^+, h^+)$};
\draw[->] (1,-0.75) -- (-0.2,-1.15);
\draw[->] (1.6,-0.75) -- (2.8,-1.15);
\node at (4,2.5) {$\mathbb C$};
\node[left] at (-1.5,2.25) {$\jmath s$};
\node[left] at (-1.5,0.75) {$-\jmath s$};
\draw[dotted] (-1.5,0.75) -- (4.5,0.75);
\draw[dotted] (-1.5,2.25) -- (4.5,2.25);
\end{tikzpicture}
     \caption{The profile of the Morse function for a cobordism. Bottlenecks are inserted on the ends of the cobordism.}
    \label{fig:bottlenecks}
\end{figure}
\begin{prop}
Let $K$ be a Lagrangian cobordism equipped with an admissible Morse function.
\begin{itemize}
    \item Even though $K$ is non-compact, the  moduli space of Morse flow lines between two critical points admits a compactification by broken flow lines and
    \item The Morse flow lines $\gamma$ of $\CM(K, h)$ for which $\pi_\RR(\gamma)=t^\pm$ are exactly the Morse flow lines of $CM(L^\pm, h^\pm)$.
\end{itemize}
    As a result, we obtain projections of the Morse cochain complexes to the ends:
\[
    \setlength\mathsurround{0pt}\begin{tikzcd}
        \; & \CM(K, h) \arrow{dl}{\beta^{-}} \arrow{dr}{\beta^+} \\
        \CM(L^-, h^-)& & \CM(L^+, h^+).
    \end{tikzcd}\setlength\mathsurround{.8pt}
\]
\label{prop:morseprojection}
\end{prop} 	
\subsection{Pearly model for Floer complex}
	\label{subsec:floerbackground}
	The receptacle we choose for the count of holomorphic disks with boundary on a compact Lagrangian $L$ is the \emph{pearly Floer cohomology} of a Lagrangian.
This is a filtered $A_\infty$  algebra $\CF(L, h)$ specified by a choice of a Lagrangian submanifold $L$, admissible Morse function $h: L\to \RR$, and perturbation data. 
The algebra structure is determined by taking a count of treed-disks, which are flow trees of the Morse function $h$ deformed with insertions of holomorphic disks with boundary on $L$ at the vertices.
Several versions of this algebra have been constructed with a variety of conditions imposed on the Lagrangian $L$ and the manifold $X$.
Some models of the pearly Floer cohomology include \cite{fukaya1997zero,biran2008lagrangian,lipreliminary,charest2015floer}.

In \cite{biran2008lagrangian}, a version of the pearly quantum homology for monotone Lagrangian submanifolds was constructed. 
In this model, the symplectic spaces studied are symplectically convex at infinity, and the Lagrangian submanifolds considered are compact and monotone. 
With these hypotheses, \cite{biran2008lagrangian} gave the definition of the pearly quantum cohomology $\QC(L, h)$, which is a deformation of Morse chain complex $CM_\bullet(L, h)$.
Furthermore, in \cite{biran2013lagrangian} they extend this theory to the setting where $K\subset X\times \CC$ is a monotone Lagrangian cobordism (\cref{def:cobordism}). 
Lagrangian cobordisms have pearly quantum homology with inclusion maps  $\QH(L^{\pm},h)\to \QH(K, h)$, where $L^\pm$ are the ends of the Lagrangian cobordism $K$. 

In \cite{charest2015floer}, the pearly model is constructed with relaxed requirements for the Lagrangian $L$. For compact symplectic spaces $X$ with rational symplectic form, \cite{charest2015floer} defines a pearly algebra $\CF(L, h)$.
The key piece of additional input into this construction is a stabilizing divisor which intersects every holomorphic disk with boundary on $L$.
These intersection points stabilize the domain of the holomorphic disks, allowing one to construct domain-dependent perturbations to the Cauchy Riemann equations. 
With this deformation, \cite{charest2015floer} cuts out a regularized moduli space of treed disks.

\begin{theorem}[\cite{charest2015floer}]
    Consider $(X, \omega)$ a rational symplectic manifold. Let $L\subset X$ be a spin and graded Lagrangian submanifold, $h: L\to \RR$ a Morse function. There exist a weakly-stabilizing divisor $D_X\subset X\setminus L$, and admissible domain-dependent perturbation datum $\mathcal P$ for $L$ defining a filtered $A_\infty$ algebra $\CF(L, h, \mathcal P)$. The filtered $A_\infty$ homotopy type of $\CF(L, h, \mathcal P)$ is independent of Morse function, weakly-stabilizing divisor, and admissible domain-dependent perturbation datum. 
\end{theorem}
For this reason, when the choice of perturbation $\mathcal P$ or Morse function $h$ is unimportant, we will write $\CF(L, h)$ or $\CF(L)$ instead.

In \cref{subsec:prebrokentrees,app:cobordismstabilizingdivisor} we make two extensions to the construction of \cite{charest2015floer}. We include short summaries of these two appendices here and highlight the results from those sections which are necessary for the main portion of this article. 

\subsubsection{Domain and Label dependent perturbations}
\label{subsub:domainandlabelsummary}
The first modification, constructed in \cref{subsec:prebrokentrees}, is to introduce perturbation datum for treed-pseudoholomorphic disks which depend not only on the domain, but also on the labelling of the domain. 
The labels are chosen from the critical points of the Morse function $h$. 

The advantage to using this set of perturbations comes from examples where a geometrically chosen almost complex structure and Morse function provides us with a computable set of regularly cut-out pseudoholomorphic flow trees between a specific set of Morse critical points. We'd like to keep these pseudoholomorphic flow trees, but it is unclear if our geometrically chosen almost complex structure is regularizing for pseudoholomorphic treed disks between other Morse critical points. By allowing our perturbation datum to depend additionally on the limiting labels of the pseudoholomorphic flow trees, we can keep our perturbation fixed on the areas where we know how to make our calculation, and modify it elsewhere.  See \cref{rem:abstractperturbations} for more detailed comments on where domain and label dependent perturbations fit in with other regularizations methods.

A specific example of where this works out is given in \cref{subsec:geometricflowlines}, the results of which we now state.
Let $J^0$ be an almost complex structure and let $h: L\to \RR$ be a Morse-Smale function. Suppose that there is a $J^0$-holomorphic disk $u_{ex}: (D^2, \partial D^2)\to (X, L)$ with the properties that :
\begin{itemize}
    \item The boundary of $u_{ex}$ is transverse to all upwards and downward flow spaces of critical points of $h$ and
    \item The disk $u_{ex}$ is regular and is the unique disk that intersects the stabilizing divisor at a single point.
\end{itemize}

Suppose that $\deg(x_1)=1, \deg(x_0)=2$, and that there are no Morse flow-lines from $x_1$ to $x_0$.
The hypothesis on the regularity of $u_{ex}$ and transversality of $\partial u_{ex}$ with the flow spaces shows that unbroken pseudoholomorphic flow-trees whose type belong to 
\[\mathbb I^{\leq 1}_{*; x_0}:=\left\{\underline \Gamma_P\;\middle|\;\ind(\underline \Gamma_P)\leq 0, \parbox{6cm}{Input label is empty or $x_1$, output is $x_0$, at most 1 interior marked point} \right\}\]
 are regularly cut for $\mathcal P^0$ ,the trivial perturbation determined by $J^0$ and $h$.

 The hypothesis on \cite[Theorem 4.19]{charest2015floer} is too restrictive to show that perturbation  system $\mathcal P^0_{\mathbb I^{\leq 1}_{*; x_0}}$ defined for types $\underline \Gamma_P\in \mathbb I^{\leq 1}_{*; x_0}$ can be coherently extended to a regular perturbation system for all combinatorial types. This is because $\mathbb I^{\leq 1}_{*; x_0}$ is not downward closed (informally: the coherence conditions for combinatorial types in $\mathbb I^{\leq 1}_{*; x_0}$ involve combinatorial types outside $\mathbb I^{\leq 1}_{*; x_0}$, see 
 \cref{def:downwardclosed}). For instance, the combinatorial type of Morse flow tree with inputs $y, x_1$ and output $x_0$ appears in the downward closure of $\mathbb I^{\leq 1}_{*; x_0}$, but not in $\mathbb I^{\leq 1}_{*; x_0}$ --- this tree could additionally break at further critical points, giving more broken types in the downward closure of our original set of combinatorial types. It is very unlikely that the trivial perturbation of Morse flow-lines will be regularizing for all of these combinatorial types, as Morse flow-trees are never regular for the trivial perturbation (unless the moduli space is empty!). So, we will usually not be able to find a $\mathcal P$ which is a coherent regular extension of $\mathcal P^0_{\mathbb I^{\leq 1}_{*; x_0}}$.

 However, we can take the following strategy: given $\mathcal P^0$ which regularly cuts out our moduli spaces $\mathbb I^{\leq 1}_{*; x_0}$ for some subset of types, we look for a regular coherent perturbation $\mathcal P$, \emph{not necessarily agreeing with $\mathcal P^0$}, but having the same moduli spaces for the desired set of combinatorial types.

 \begin{lemma}[Paraphrasing \cref{lemma:geometricflowlines}]
     There exists a full regular coherent perturbation system $\mathcal P$ so that for every combinatorial type in $\underline \Gamma_P\in \mathbb I^{\leq 1}_{*; x_0}$, the count of $\mathcal P$-perturbed pseudoholomorphic curves of type $\underline \Gamma_P$ agrees with the count of $\mathcal P^0$-perturbed pseudoholomorphic curves.
\end{lemma}

\subsubsection{Pearly Model for Lagrangian Cobordisms }

The second modification of the pearly model of \cite{charest2015floer}, constructed in \cref{app:cobordismstabilizingdivisor}, concerns the extension of the results of \cite{biran2013lagrangian} to \cite{charest2015floer}. We show that perturbation datum and stabilizing divisors may be chosen so that projection maps from the pearly Floer cohomology of a Lagrangian cobordism to the ends exist.
While we use \cite{charest2015floer} as a framework due to our familiarity with this particular version of the pearly model, we expect the techniques employed here should transfer in the most general sense to other models of Lagrangian Floer cohomology, or abstract regularization method used to construct the moduli spaces of pseudoholomorphic treed disks.

\begin{lemma}[Paraphrasing of \cref{lem:weaklystabilizedcobordism,lem:estabilizedcobordism}]
    Let $K:L^-\rightsquigarrow L^+$ be a Lagrangian cobordism, and $h: K\to \RR$ a compatible Morse function. 
    There exists a weakly stabilizing divisor $D\subset (X\times \CC)\setminus K$, and 
    $\mathcal P_{K}$ for $K$ satisfying the following properties:
    \begin{itemize}
        \item $\mathcal P_{K}$ is admissible , in that it is $E$-weakly stabilized by the divisor $D$, regular, and coherent.
        \item $\mathcal P_{K}$ consists of domain-dependent choices of almost complex structure for which the projection $\pi_\CC: X\times \CC\to \CC$ is holomorphic outside of a compact set. 
    \end{itemize}
\end{lemma}
This can be used to prove that the moduli space of pseudoholomorphic treed disks is compact (\cref{prop:compactness}) and compatible with the perturbation systems $\mathcal P_{L^\pm}$ for the ends $L^\pm$ of the Lagrangian cobordism (\cref{assum:pearlycompatibility}). Compactness allows us to define the pearly model for Lagrangian cobordisms (\cref{cor:pearlymodelexistence}), while compatibility proves that the projections
   \[
       \setlength\mathsurround{0pt}\begin{tikzcd}
           \; & \CF(K, h) \arrow{dl}{\beta^{-}} \arrow{dr}{\beta^+} \\
           \CF(L^-, h^-)& & \CF(L^+, h^+)
       \end{tikzcd}\setlength\mathsurround{.8pt}
   \]
are filtered $A_\infty$-homomorphisms (\cref{cor:projectionsareainfinity}). This is the extension of \cref{prop:morseprojection} to Lagrangian Floer theory.

\section{Toy model: continuation maps from cylindrical cobordisms}
	\label{sec:continuation}
	
	We now give an example where the results of \cite{biran2013fukayacategories} apply to the non-monotone setting.
While we now allow our Lagrangian cobordisms to possibly bound holomorphic disks, we require that the cobordism's topology be cylindrical.
A relevant set of examples are the suspension of a Hamiltonian isotopy.

\begin{definition}[Suspension of a Hamiltonian Isotopy]
	Let $H_t: X\times \RR_t\to \RR$ be a Hamiltonian  with support on $[t^-, t^+]$. 
		Let $\phi_t: L\times \RR_t$ be a Lagrangian isotopy induced by the Hamiltonian flow of $H_t$. The \emph{suspension of the Hamiltonian isotopy $\phi_t$}  is the Lagrangian cobordism $K_{H_t}$ given by the embedding
	\begin{align*}
		L\times \RR\into & X\times \CC             \\
		(x, t) \mapsto   & (\phi_t(x), t+\jmath H_t(\phi_t(x))).
	\end{align*}
	\label{def:suspensioncobordism}
\end{definition}
More generally, given a Lagrangian homotopy $\phi_t: L\times \RR_t\to X$ which is exact (in the sense that $\iota_{\frac{d\phi}{dt}}\omega|_{L_t}$ is an exact form for all $t$,), there exists a similar construction for the suspension of the exact homotopy.
In a similar vein to the equivalences constructed in \cite{biran2013fukayacategories}, the suspension Lagrangian cobordism gives a continuation map for the pearly algebra of $L$.

\begin{prop}
	Let $L^-, L^+\subset X$ be two Lagrangian submanifolds, with Morse functions $h^\pm: L^\pm\to \RR$. 
	Suppose that $K\subset X\times \CC$ is an embedded Lagrangian cobordism with ends on $L^-, L^+$.
	Furthermore, suppose that there is a diffeomorphism $K\cong L^-\times \RR$. 	Then there exists a filtered $A_\infty$ homomorphism $\Theta_K:\CF(L^-; h^-)\to CF(L^+; h^+)$.
	\label{thm:cylindricityofcobordism}
\end{prop}
\begin{proof}
	We pick a particular Morse function $h$ for $K$.
	Take a preliminary function 
	\begin{equation}
		\tilde h(t): \RR\to \RR
		\label{eqn:Morsebottfunction}
	\end{equation} which has maximums at $t^-$ and $t^+$ and a minimum at $t^+-\epsilon$.
		See \cref{fig:bottlenecks}.
	This defines a Morse-Bott function on $K$. 
	We Morsify this function by taking perturbations based on the functions $h^\pm$ as considered in \cref{prop:morseprojection}.
	By construction, we have an identification 
	\[\CF(K, h)=\CF(L^-, h^-)\oplus \CF(L^+, h^+)[1]\oplus \CF(L^+, h^+)\]
	as graded vector spaces. 
From \cref{assum:pearlycompatibility}, the pearly differential on $K_H$ decomposes as
	\begin{equation}
	m_{K}^1=	
	\begin{pmatrix}
	m^{-;-} & 0 & 0\\
	m^{-;0} & m^{0;0} & m^{+;0}\\
	0 & 0& m^{+;+}
		\end{pmatrix}.
		\label{eq:morsedifferentialcylinder}
	\end{equation}
	where $m^{\pm;\pm}$ matches $m^1_{L^\pm}$. 
	
	See \cref{item:multilinearnotation} for notation.
	We additionally have the projection $A_\infty$ homomorphisms $\beta^\pm: \CF(K, h)\to \CF(L^\pm, h^\pm)$.
	The map	$m^{+;0}$ is a deformation of the Morse differential on $K$ by terms of non-zero filtration.
	Since the Morse portion of the differential 
	\[\underline{m}^{+;0}:\CM(L^+, h^+)\to \CM(L^+, h^+)[1]\]
	is the identity, the Floer differential $m^{+;0}$ which arises as a deformation by higher energy terms remains an isomorphism. 
	Therefore  $\CF(L^-, h^-)\ot \CF(K, H)\to \CF(L^+ , h^+)$ is a filtered $A_\infty$ mapping cocylinder (\cref{def:mappingcylinder}).
	We obtain a map $\alpha: \CF(L^-, h^-)\to \CF(K, h)$  and a pushforward-pullback map  (\cref{thm:cylfrommap})
	\[\Theta_K:=\beta^+\circ \alpha: \CF(L^-, h^-)\to CF(L^+, h^+).\]
\end{proof}
\begin{remark}
	The hypothesis that $K$ is diffeomorphic to $L^-\times \RR$ can be weakened to $K$ is a $h$-cobordism, in which case the map $\underline{m}^{+;0}$ is a quasi-isomorphism. While \cref{def:mappingcylinder} would no longer apply, one can use \cite[Theorem 4.2.45]{fukaya2010lagrangian} to obtain the continuation map. 
\end{remark}
It immediately follows that there is a continuation map for the pearly model of Lagrangian Floer theory.
\begin{corollary}
	 Let $L^-, L^+$ be Hamiltonian isotopic Lagrangian submanifolds.
	There exists a pullback-pushforward map 
	\[
	\Theta_{H_t}:\CF(L^-,h^-)\to \CF(L^+, h^+)
	\]
	called the \emph{continuation map} of $H_t$.
	\label{cor:mappingcylindercochains}
\end{corollary}
\begin{proof}
	Consider the suspension cobordism $K_{H_t}$, and apply \cref{thm:cylindricityofcobordism}. 
\end{proof}
Additionally, the presence of such a map proves the existence of maps between the Maurer-Cartan space of these Lagrangians.
\begin{corollary}
	Let $L^-$ and $L^+$ be Hamiltonian isotopic Lagrangian submanifolds. 
	If $\CF(L^-)$ is unobstructed by bounding cochain $b^-$, then $L^+$ is unobstructed by the bounding cochain $b^+=(\Theta_{H_t})_*b^-$. 
	Furthermore, the continuation map can be made a map of uncurved $A_\infty$ algebras $(\Theta_{H_t})_\flat: \CF_{b^-}(L^-)\to \CF_{b^+}(L^+)$.
	\label{cor:uncurvedcontinuation} 
\end{corollary}
\begin{proof}
	The existence of an filtered $A_\infty$ homomorphism implies the existence of these maps (see \cref{claim:fflat}).
\end{proof}

\section{Cobordisms and mutations}
	\label{sec:cobordismandmutations}
	In this section, we look at cobordisms arising from mutation configurations.
	We prove that in certain scenarios such a cobordism gives an $A_\infty$ mapping cocylinder between the Floer cohomologies of its ends. 
	In \cref{subsec:haugcobordism}, we introduce a local model for mutation in $\CC^2\setminus\{z_1z_2=1\}$.
	This local description is used to construct a holomorphic disk with boundary on the Lagrangian cobordism. 
	In \cref{subsec:generalmutations} we use this disk to build a continuation map from a Lagrangian mutation cobordism.  
	\subsection{Review of the Haug cobordism}
	\label{subsec:haugcobordism}
	
This is a review of the mutation cobordism constructed in \cite{haug2015lagrangian} between the monotone Chekanov and product torus in $(\CC^2)\setminus \{z_1z_2=1\}$.
We assemble this cobordism from pieces built in Lefschetz fibrations. 
The fibration we consider is $W=z_1z_2: \CC^2\to \CC$, which has a single nodal fiber at the origin. 

We first give a description of Lagrangian surgery in Lefschetz fibrations.
Consider the paths
\begin{align*}
	\ell_{\uparrow/\downarrow} :[0, R]\to \CC &&	
	r\mapsto \pm \jmath\cdot r^2
\end{align*}
in the base of the Lefschetz fibration with ends on the critical value of the fibration.
We take lifts of these paths to give Lagrangian thimbles $L_\uparrow, L_\downarrow\subset \CC^2$.
These Lagrangians have the topology of $D^2$ and can be parameterized in coordinates by:
\begin{align*}
	L_\uparrow=\{(x+\jmath y, y+\jmath x)\;:\; x^2+y^2 \leq R^2   \} &&
	L_\downarrow= \{(x+\jmath y, -y-\jmath x)\;:\; x^2+y^2 \leq R^2\}.
\end{align*}
We now recall the construction of Lagrangian surgery and the associated suspension cobordism from \cite{biran2013lagrangian}.
We pick a \emph{surgery profile curve}, which is a map $(a(t)+\jmath b(t)): [-R, R] \to \CC$.
The functions $a(t)$ and $b(t)$ are non-strictly increasing, we require that $(a(t)+\jmath b(t))=t$ for $t<c$ and $(a(t)+\jmath b(t))=\jmath t$ for $t>c$.
This neck provides a local model for the Lagrangian surgery.
To insert this neck at the transverse intersection point of two Lagrangian submanifolds, we take a Darboux chart around the intersection point which identifies these Lagrangian submanifolds with the totally real and imaginary subspaces of $\CC^n$ and interpolate between these two subspaces using the surgery profile curve. 
For the example we are concerned with, the surgery $L_\uparrow\# L_\downarrow$ is explicitly described in coordinates by
\begin{equation}
	\{(( a(t)+\jmath b(t))(\hat x+\jmath\hat y), (a(t)+\jmath b(t)) (\hat y+\jmath\hat x))\;:\; \hat x^2+\hat y^2= 1\}.
	\label{eq:surgeryneck}
\end{equation}

With this construction, we take special care to the order of the summands in the connect sum.
Our convention is that the end of the neck which corresponds to $t\in [-R, -c]$ corresponds to the first summand of the Lagrangian connect sum.
The \emph{neck width} of a mutation measures the amount of flux swept out as one takes an isotopy from $L_\uparrow\#L_\downarrow$ to $L_\downarrow\# L_\uparrow$.
The projection of this cobordism under $W$ is
\[
	W(L_\uparrow\#L_\downarrow)= \jmath (a(t)+\jmath b(t))^2 (\hat x^2+\hat y^2)= \jmath (a(t)+\jmath b(t))^2
\]
whose image under $W$ is a curve in the complex plane:
\begin{align*}
	\ell_\uparrow\# \ell_\downarrow:[-R, R]\to \CC  &&
	t\mapsto (-2a(t)b(t) +\jmath (a(t)^2-b(t)^2 ))
\end{align*}
Notice that when $t\in [-R, -c]$ the coordinate $b(t)=0$ and the curve matches the positive imaginary axis.
Similarly, when $t\in [c, R]$ we have that $a(t)=0$ and the curve matches the negative imaginary axis.
The real component of $\ell_\uparrow \# \ell_\downarrow$ is always non-negative.
See \cref{fig:catseye,fig:cobordismpieces} for diagrams of the pieces described above. 
We can similarly construct a profile for $L_\downarrow\# L_\uparrow$ and corresponding path $\ell_\downarrow\# \ell_\uparrow$.

Given a Lagrangian surgery, \cite{biran2013fukayacategories} describes a \emph{trace cobordism}   $K_l: L_\uparrow\# L_\downarrow\rightsquigarrow L_\downarrow\sqcup L_\uparrow$ and $K_r:L_\downarrow\sqcup L_\uparrow\rightsquigarrow L_\downarrow\# L_\uparrow$.
There is a comparison between the neck width of the cobordism and the area of the projection of the cobordism to $\CC$, which is that the area of projection of the cobordism is at least the neck-width.
By concatenating these cobordisms together, we obtain a Lagrangian cobordism $K_r\circ K_l: L_\uparrow\# L_\downarrow\rightsquigarrow L_\uparrow\# L_\downarrow$.
The projection of $K_r\circ K_l$  to $\CC$ is the eye-shaped Lagrangian drawn in \cref{fig:catseye}. 
We additional draw the slices of the Lagrangian cobordism, $K|_t:=\pi_X(\pi^{-1}_\RR(t)\cap K)$. 
\begin{figure}
	\centering
	\begin{tikzpicture}

\fill[gray!10]  (-4,1) rectangle (6.5,-2.5);
\draw (-4,-0.5) -- (-0.5,-0.5);
\draw[ fill=gray!80] (-0.5,-0.5) .. controls (0,-0.5) and (0.5,0.5) .. (1,0.5) .. controls (1.5,0.5) and (2,-0.5) .. (2.5,-0.5);
\draw[fill=gray!80] (-0.5,-0.5) .. controls (0,-0.5) and (0.5,-1.5) .. (1,-1.5) .. controls (1.5,-1.5) and (2,-0.5) .. (2.5,-0.5);
\draw (6.5,-0.5) -- (2.5,-0.5);
\draw[red, fill=gray!10] (1,0.5) .. controls (0.5,0.5) and (0.2,-0.3) .. (0.2,-0.5) .. controls (0.2,-0.7) and (0.5,-1.5) .. (1,-1.5) .. controls (1.5,-1.5) and (1.8,-0.7) .. (1.8,-0.5) .. controls (1.8,-0.3) and (1.5,0.5) .. (1,0.5);

\begin{scope}[shift={(-0.5,1.5)}]
\draw  (-3,4) rectangle (-1,2);
\node[red] at (-2,3) {$\times$};
\draw (-2,4) .. controls (-2,3.5) and (-1.5,3.5) .. (-1.5,3) .. controls (-1.5,2.5) and (-2,2.5) .. (-2,2);

\node at (-1.25,2.25) {$X$};
\node at (-2,4.5) {$K|_{-c}$};
\end{scope}
\begin{scope}[shift={(3,1.5)}]
\draw  (-3,4) rectangle (-1,2);
\node[red] at (-2,3) {$\times$};
\draw (-2,4) .. controls (-2,3.5) and (-2,3.5) .. (-2,3) .. controls (-2,2.5) and (-2,2.5) .. (-2,2);

\node at (-1.25,2.25) {$X$};
\node at (-2,4.5) {$K|_0$};
\end{scope}
\begin{scope}[shift={(6.5,1.5)}]
\draw  (-3,4) rectangle (-1,2);
\node[red] at (-2,3) {$\times$};
\draw (-2,4) .. controls (-2,3.5) and (-2.5,3.5) .. (-2.5,3) .. controls (-2.5,2.5) and (-2,2.5) .. (-2,2);

\node at (-1.25,2.25) {$X$};
\node at (-2,4.5) {$K|_c$};
\end{scope}
\draw[->] (-2.5,2.5) -- (-2.5,1.5);
\draw[->] (1,2.5) -- (1,1.5);
\draw[->] (4.5,2.5) -- (4.5,1.5);
\node at (5.5,-2) {$\mathbb C$};
\node at (-3.5,0) {$K$};
\end{tikzpicture}
 	\caption{The projection of the  eye-shaped Lagrangian cobordism, $(K_r\circ K_l\cap U)\subset X\times \CC\to \CC$.
	The cycle $\pi^{-1}_{ U}(0,0)$ is highlighted in red.}
	\label{fig:catseye}
\end{figure}
\begin{remark}
	Another way of stating the relation between area of the projection and the neck width is to say that the \emph{shadow} of the cobordism $K_r\circ K_l$ as defined in \cite{cornea2019lagrangian} is an upper bound for the widths of the two surgeries. 
	\end{remark}

We now wish to make comparisons between the surgery of the thimbles $L^{\pm}$ , Chekanov and product tori, and the Whitney sphere.
In the base of the Lefschetz fibration, one can construct the Whitney sphere by taking a $c$-shaped path $\ell_c:[-S, S]\to \CC$ from the negative end of $\ell_\downarrow$ to the positive end of $\ell_\uparrow$ in a clockwise fashion so that the concatenation $\ell_\downarrow \cdot \ell_c \cdot \ell_\uparrow$ is a matching path from the single critical point of the fibration to itself.
Let $L_c:[-S, S]\times S^1\to \CC$ be the corresponding parallel transport of the vanishing cycle.
The union $L_c\cup L_\downarrow\cup L_\uparrow$ is the Whitney sphere $L_{S^2}$.
By attaching $L_c$ to $L_\downarrow\#L_\uparrow$, we will obtain the monotone Chekanov torus, denoted  $L^-$.
By attaching $L_c$ to $L_\uparrow\#L_\downarrow$, we obtain the monotone product torus, which will be denoted $L^+$.
See \cref{fig:cobordismpieces} for a picture of these Lagrangian submanifolds. 
\begin{figure}
	\centering
	\scalebox{.75}{\begin{tikzpicture}
	\begin{scope}[shift={(-7.5,-4.5)}]
		\fill[gray!20]  (2,2) rectangle (-2,-2);
		\node at (0.5,0) {$\times$};
		\draw[blue] (0.5,0) -- (0.5,1);
		\node[right, blue] at (0.5,0.5) {$\ell_\uparrow$};
		\draw[red] (0.5,-1) -- (0.5,0);
		\node[right, red] at (0.5,-0.5) {$\ell_\downarrow$};
		\node[left, purple] at (-0.5,0) {$\ell_c$};
		\draw[purple] (0.5,1) .. controls (0.5,1.5) and (-0.5,1.5) .. (-0.5,1) .. controls (-0.5,0.5) and (-0.5,-0.5) .. (-0.5,-1) .. controls (-0.5,-1.5) and (0.5,-1.5) .. (0.5,-1);

		\node at (1,-1.5) {\small  Whitney};
		\node[draw=black, fill=white, circle, scale=.	2] at (-0.25,0) {};
		\node at (1.5, 1.75) {$\mathbb C\setminus \{z_1z_2=1\}$};
\end{scope}

	\begin{scope}[shift={(2,-4.5)}]
		\fill[gray!20]   (2,2) rectangle (-2,-2);
		\node at (0.5,0) {$\times$};
		\node[right, teal] at (.5,1) {$\ell_\downarrow\#\ell_\uparrow$};
		\draw[teal] (0.5,1) .. controls (0.5,0.5) and (0,0.5) .. (0,0) .. controls (0,-0.5) and (0.5,-0.5) .. (0.5,-1);

		\node[right, brown] at (-1.5,0) {$\theta_H(\ell_c)$};
		\draw[brown] (0.5,-1) .. controls (0.5,-1.5) and (-0.5,-1.5) .. (-0.5,-1) .. controls (-0.5,-0.5) and (-1.5,-1) .. (-1.5,0) .. controls (-1.5,1) and (-0.5,0.5) .. (-0.5,1) .. controls (-0.5,1.5) and (0.5,1.5) .. (0.5,1);

		\node at (1,-1.5) {\small Chekanov};
		\node[draw=black, fill=white, circle, scale=.2] at (-0.25,0) {};
		\node at (1.5, 1.75) {$\mathbb C\setminus \{z_1z_2=1\}$};
	\end{scope}

	\begin{scope}[shift={(-2.5,-4.5)}]
		\fill[gray!20]  (2,2) rectangle (-2,-2);
		\node at (0.5,0) {$\times$};
		\node[right, orange] at (0.5,1) {$\ell_\uparrow\#\ell_\downarrow$};
		\draw[orange] (0.5,1) .. controls (0.5,0.5) and (1,0.5) .. (1,0) .. controls (1,-0.5) and (0.5,-0.5) .. (0.5,-1);
		\node[left, purple] at (-0.5,0) {$\ell_c$};
		\draw[purple] (0.5,1) .. controls (0.5,1.5) and (-0.5,1.5) .. (-0.5,1) .. controls (-0.5,0.5) and (-0.5,-0.5) .. (-0.5,-1) .. controls (-0.5,-1.5) and (0.5,-1.5) .. (0.5,-1);

		\node at (1,-1.5) {\small  Product};
		\node[draw=black, fill=white, circle, scale=.2] at (-0.25,0) {};
		\node at (1.5, 1.75) {$\mathbb C\setminus \{z_1z_2=1\}$};
	\end{scope}

\end{tikzpicture}
 }
	\caption{All of the paths in the Lefschetz fibration needed to assemble the tori and cobordism}
	\label{fig:cobordismpieces}
\end{figure}

Ideally at this point we would glue the cobordism $L_c\times \RR$ to the concatenation  $K_r\circ K_l$ to obtain a cobordism between the Chekanov and product torus. 
However the Lagrangian $L_c\times \RR$ does not glue to the cobordism  $K_r\circ K_l$, as the boundary of  $K_r\circ K_l$ corresponding to the ends of the thimbles $L^{\uparrow/\downarrow}$ move around the parameter space of the cobordism along curves dictated by the gluing profile.
We therefore modify $L_c\times \RR$ by a Hamiltonian isotopy so these pieces overlap and we can glue together. 

To achieve this, we need a more precise description of the construction of the Lagrangian surgery trace cobordism.
The surgery trace cobordism is obtained by performing the surgery in one dimension higher and stretching the neck of the surgery (see section 6.1 of \cite{biran2013lagrangian} ).
We will build our cobordism between Chekanov and product tori by performing surgery on an immersed Lagrangian submanifold. 
This immersed Lagrangian submanifold is the suspension of an exact homotopy $\phi_t:L_{S_2}\times \RR \to X$
which fixes the points of self-intersection.
However, the primitive $H_t: L_{S^2}\times \RR\to \RR$  for $\iota_{\frac{d\phi_t}{dt}}\omega |_{L_{S_2}}$ which is used to construct the suspension cobordism will take different values at the self-intersection points of the Whitney sphere, so that the Lagrangian suspension is embedded at all $t_0$ where $H_{t_0}:L\times \{t_0\}\to \RR$ is non-constant.
As we will later want to compute the flux of this exact homotopy, we construct this suspension of the exact homotopy using the pieces drawn in \cref{fig:catseyepaths}.
Take paths $\gamma_{\uparrow/\downarrow}(t): \RR\to \CC$ as drawn in \cref{fig:catseyepaths} and consider the cylindrical Lagrangian cobordisms $L_\downarrow\times \gamma_\downarrow$ and $L_\uparrow\times \gamma_\uparrow$. 
Choose a parameterization of these paths so that $\gamma_{\uparrow/\downarrow}(t) = t+\jmath f_{\uparrow/\downarrow}(t)$ (see \cref{fig:catseyepaths} (a)).
By taking a translation of these paths, we assume that $f_{\uparrow/\downarrow} (0)=0$ and $f^{\pm}(t_0)=0$.
\begin{figure}
	\centering
	\scalebox{.75}{\begin{tikzpicture}
	\begin{scope}[shift={(-0.5,9.5)}]
		\node[above] at (1,0) {$ L_\uparrow\times \gamma_\uparrow$};
		\node[below] at (1,-1) {$L_\downarrow\times \gamma_\downarrow$};
		\draw  (3.5,-2) rectangle (-1.5,1.5);
		\draw (-1,0.5) .. controls (-0.5,0) and (0,-0.5) .. (0.25,-0.75) .. controls (0.5,-1) and (1.5,-1) .. (1.75,-0.75) .. controls (2,-0.5) and (2.5,0) .. (3,0.5);
		\draw (-1,-1.5) .. controls (-0.5,-1) and (0,-0.5) .. (0.25,-0.25) .. controls (0.5,0) and (1.5,0) .. (1.75,-0.25) .. controls (2,-0.5) and (2.5,-1) .. (3,-1.5);
	\end{scope}

	\begin{scope}[shift={(-0.5,5)}]
		\node[above] at (1,-0.5) {$L_c\times \mathbb{R}$};
		\draw  (3.5,-2.5) rectangle (-1.5,1);
		\draw (-1,-1) -- (3,-1);
	\end{scope}

	\begin{scope}[shift={(5.5,5)}]
		\draw  (3.5,-2.5) rectangle (-1.5,1);
		\begin{scope}[shift={(0,-0.5)}]
			\clip (-1,0.5) .. controls (-0.5,0) and (0,-0.5) .. (0.25,-0.75) .. controls (0.5,-1) and (1.5,-1) .. (1.75,-0.75) .. controls (2,-0.5) and (2.5,0) .. (3,0.5);
			\clip (-1,-1.5) .. controls (-0.5,-1) and (0,-0.5) .. (0.25,-0.25) .. controls (0.5,0) and (1.5,0) .. (1.75,-0.25) .. controls (2,-0.5) and (2.5,-1) .. (3,-1.5);
			\fill[fill=gray!20]  (0,0) rectangle (2,-1);
		\end{scope}
		\fill[fill=gray!20] (-1,0) -- (0,-1) -- (-1,-2) -- cycle;
		\fill[fill=gray!20] (3,0) -- (3,-2) -- (2,-1) -- cycle;
		\draw[dashed] (-1,0) .. controls (-0.5,-0.5) and (0,-1) .. (0.25,-1.25) .. controls (0.5,-1.5) and (1.5,-1.5) .. (1.75,-1.25) .. controls (2,-1) and (2.5,-0.5) .. (3,0);
		\draw[dashed] (-1,-2) .. controls (-0.5,-1.5) and (0,-1) .. (0.25,-0.75) .. controls (0.5,-0.5) and (1.5,-0.5) .. (1.75,-0.75) .. controls (2,-1) and (2.5,-1.5) .. (3,-2);		\node[above] at (1,-0.5) {$ K_c$};
	\end{scope}

	\begin{scope}[shift={(5.5,9.5)}]
		\draw  (3.5,-2) rectangle (-1.5,1.5);
		\begin{scope}[]
			\clip (-1,0.5) .. controls (-0.5,0) and (0,-0.5) .. (0.25,-0.75) .. controls (0.5,-1) and (1.5,-1) .. (1.75,-0.75) .. controls (2,-0.5) and (2.5,0) .. (3,0.5);
			\clip (-1,-1.5) .. controls (-0.5,-1) and (0,-0.5) .. (0.25,-0.25) .. controls (0.5,0) and (1.5,0) .. (1.75,-0.25) .. controls (2,-0.5) and (2.5,-1) .. (3,-1.5);

			\fill[fill=gray!20]  (0,0) rectangle (2,-1);
		\end{scope}
		\fill[fill=gray!20] (-1,0.5) -- (0,-0.5) -- (-1,-1.5) -- cycle;
		\fill[fill=gray!20] (3,0.5) -- (3,-1.5) -- (2,-0.5) -- cycle;

		\draw (-1,0.5) .. controls (-0.5,0) and (0,-0.5) .. (0.25,-0.75) .. controls (0.5,-1) and (1.5,-1) .. (1.75,-0.75) .. controls (2,-0.5) and (2.5,0) .. (3,0.5);
		\draw (-1,-1.5) .. controls (-0.5,-1) and (0,-0.5) .. (0.25,-0.25) .. controls (0.5,0) and (1.5,0) .. (1.75,-0.25) .. controls (2,-0.5) and (2.5,-1) .. (3,-1.5);
		\node[above] at (1,0) {$ K_c\cup( L_{\uparrow/\downarrow}\times \gamma_{\uparrow/\downarrow}) $};
	\end{scope}

	\begin{scope}[shift={(11.5,9.5)}]
		\draw  (3.5,-2) rectangle (-1.5,1.5);
		\begin{scope}[]
			\clip (-1,0.5) .. controls (-0.5,0) and (0,-0.5) .. (0.25,-0.75) .. controls (0.5,-1) and (1.5,-1) .. (1.75,-0.75) .. controls (2,-0.5) and (2.5,0) .. (3,0.5);
			\clip (-1,-1.5) .. controls (-0.5,-1) and (0,-0.5) .. (0.25,-0.25) .. controls (0.5,0) and (1.5,0) .. (1.75,-0.25) .. controls (2,-0.5) and (2.5,-1) .. (3,-1.5);

			\fill[fill=gray!20]  (0,0) rectangle (2,-1);
		\end{scope}
		\fill[fill=gray!20] (-1,0.5) -- (0,-0.5) -- (-1,-1.5) -- cycle;
		\fill[fill=gray!20] (3,0.5) -- (3,-1.5) -- (2,-0.5) -- cycle;

		\draw (-1,0.5) .. controls (-0.5,0) and (0,-0.5) .. (0.25,-0.75) .. controls (0.5,-1) and (1.5,-1) .. (1.75,-0.75) .. controls (2,-0.5) and (2.5,0) .. (3,0.5);
		\draw (-1,-1.5) .. controls (-0.5,-1) and (0,-0.5) .. (0.25,-0.25) .. controls (0.5,0) and (1.5,0) .. (1.75,-0.25) .. controls (2,-0.5) and (2.5,-1) .. (3,-1.5);
		\node[above] at (1,0) {$ K^{pre}$};
		\draw (-0.25,-0.25) .. controls (-0.125,-0.375) and (-0.125,-0.625) .. (-0.25,-0.75);
		\draw (0.25,-0.75) .. controls (0.125,-0.625) and (0.125,-0.375) .. (0.25,-0.25);
	\end{scope}
	\begin{scope}[shift={(11.5,5)}]
		\draw  (3.5,-2.5) rectangle (-1.5,1);
		\draw[fill=gray!20] (-1,-1) .. controls (-0.5,-1) and (-0.5,-1) .. (-0.5,-1) .. controls (0,-1) and (0,-1) .. (0.25,-1.25) .. controls (0.5,-1.5) and (1.5,-1.5) .. (1.75,-1.25) .. controls (2,-1) and (2,-1) .. (2.5,-1) .. controls (3,-1) and (2.5,-1) .. (3,-1);
		\draw[fill=gray!20] (-1,-1) .. controls (-0.5,-1) and (-0.5,-1) .. (-0.5,-1) .. controls (0,-1) and (0,-1) .. (0.25,-0.75) .. controls (0.5,-0.5) and (1.5,-0.5) .. (1.75,-0.75) .. controls (2,-1) and (2,-1) .. (2.5,-1) .. controls (2.5,-1) and (2.5,-1) .. (3,-1);
		\node[above] at (1,-0.5) {$ K $};
		\node[scale=.7] at (-0.5,-1.5) {$L_\uparrow\# L_\downarrow\cup L_c$};
		\node[scale=.7] at (2.5,-1.5) {$L_\downarrow\# L_\uparrow\cup \theta_{H}(L_c)$};
		\draw (0.25,-1.25) .. controls (0.125,-1.125) and (0.125,-0.875) .. (0.25,-0.75);
		\draw (1.75,-0.75) .. controls (1.875,-0.875) and (1.875,-1.125) .. (1.75,-1.25);
	\end{scope}

\draw[->] (3,4.5) -- (4,4.5);
	\draw[->] (3,9) -- (4,9) ;
	\draw[->] (6.5,6) -- (6.5,7.5) ;
	\draw[->] (9,9) -- (10,9) ;
	\draw[->] (12.5,7.5) -- (12.5,6);
	\draw (13.25,9.25) .. controls (13.375,9.125) and (13.375,8.875) .. (13.25,8.75);
	\draw (13.75,9.25) .. controls (13.625,9.125) and (13.625,8.875) .. (13.75,8.75);
\node at (2.5,10.5) {(a)};
\node at (2.5,5.5) {(d)};
\node at (8.5,5.5) {(e)};
\node at (8.5,10.5) {(b)};
\node at (14.5,10.5) {(c)};
\node at (14.5,5.5) {(f)};
\draw[dotted] (-0.5,7.5) -- (-0.5,11) (1.5,7.5) --(1.5,11);
\node at (-0.5,11.5) {$t=0$};
\node at (1.5,11.5) {$t=t_0$};
\end{tikzpicture}
 }
	\caption{Steps to build our Lagrangian cobordism by projection to the cobordism $\CC$ parameter of  $X\times \CC$. }
	\label{fig:catseyepaths}
	\end{figure}
We will now glue to ($L_\downarrow\times \gamma_\downarrow)\sqcup (L_\uparrow\times \gamma_\uparrow)$ a Lagrangian which is topologically $L_c\times \RR$.
This cobordism will be the suspension of a time dependent Hamiltonian isotopy of $L_c$. 
We pick a Hamiltonian $H_t: L_c\times \RR\to \RR$ which satisfies the following properties:
\begin{itemize}
	\item The value of $H_t$ on $L_c \times \RR = [-S,S] \times S^1 \times \RR$ is pulled back from $[-S,S]$ and is increasing in $s$.
		\item There exists a constant $\epsilon$ so that the restriction of
	      \begin{align*}
		      H_t(L_c|_{s\in [-S, -S+\epsilon)})=  f_\downarrow(t) &&
		      H_t(L_c|_{s\in [S, S-\epsilon)})=    f_\uparrow(t)
	      \end{align*}
\end{itemize}
We define the cobordism $K_c$ to be the suspension of this Hamiltonian isotopy (see \cref{fig:catseyepaths} (d, e).
 By design, the Hamiltonian flow $X_{H^t}$ is constant on the region $[\pm S, \pm S \mp \epsilon)\times S^1$ of $L_c$.
Therefore the portion of $K_c$ corresponding to the $[\pm S, \pm S \mp \epsilon)\times S^1\times \RR$ is exactly $L_c|_{s\in [\pm S, \pm S \mp \epsilon)}\times \gamma_{\uparrow/\downarrow}$.
See \cref{fig:cobordismpieces,fig:catseyepaths}.
\begin{claim}
	The union $K_c\cup (L_{\uparrow} \times \gamma_{\uparrow})\cup(L_\downarrow\times \gamma_{\downarrow})$ is a smooth Lagrangian submanifold with 2 transverse self-intersections (\cref{fig:catseyepaths} (b)). 
	\label{claim:cobordismconstruction}
\end{claim}
We now construct the mutation cobordism from $K_c\cup (L_{\uparrow} \times \gamma_{\uparrow})\cup(L_\downarrow\times \gamma_{\downarrow})$.
First, we apply a Lagrangian surgery at both of the self-intersection points to obtain the Lagrangian submanifold $K^{pre}$ (\cref{fig:catseyepaths} (c)), which is an embedded Lagrangian submanifold.
The slices of this Lagrangian cobordism at times $0$ and $t_0$ are
\begin{align*}
	K^{pre}|_{0}= (L_\uparrow\#L_\downarrow\cup L_c) && K^{pre}|_{t_0}=( L_\downarrow\#L_\uparrow)\cup\theta_H( L_c).
\end{align*}
These are product and Chekanov tori respectively. 

We construct the \emph{mutation cobordism} $K_{\mu_D}$ (\cref{fig:catseyepaths}) by taking the portion of $K^{pre}$ between $0$ and $t_0$ whose projection to the real component lies in $[0, t_0]$.
To obtain a Lagrangian cobordism from this space, we take this portion and stretch it out to obtain a Lagrangian cobordism with ends. This stretching process is described in \cite[Section 3.1.1]{hicks2021lagrangian}, where it is called the truncation of $K^{pre}$ \cite[Definition 3.1.5]{hicks2021lagrangian}. In the notation of that paper, we write:
\[K_{\mu_D}:=K^{pre}\|_{[0, t_0]}.\]
\begin{remark}
	The construction of this Lagrangian cobordism can also be done by locally identifying the Lagrangians $L_\uparrow\times\gamma_\uparrow$ and $L_\downarrow\times \gamma_\downarrow$ with $\RR^2\times \RR, \jmath\RR^2\times \jmath\RR\subset \CC^2\times \CC=\CC^{2+1}$, then following the construction of the Lagrangian surgery trace in \cite[Lemma 6.1.1]{biran2013lagrangian}. The cutting and cylindrical neck attachment construction there is similar to the truncation procedure from \cite{hicks2021lagrangian}, and produces a Lagrangian isotopic submanifold.
\end{remark}

 	\subsection{General mutations}
	\label{subsec:generalmutations}
	In this section we look at a generalization of the above construction.
\label{sec:mutation}
\begin{theorem}[Haug, \cite{haug2015lagrangian}]
	Suppose that $D$ is a Lagrangian disk with boundary cleanly contained within an oriented Lagrangian submanifold $L$. 
	There exists an immersed Lagrangian $\alpha_D(L)\looparrowright X$ called the \emph{Lagrangian anti-surgery\footnote{The character $\alpha$ is chosen for anti-surgery as it looks like an immersed Lagrangian.} of $L$ along $D$}, which satisfies the following properties:	\begin{itemize}
		\item
		      $\alpha_D(L)$ is topologically obtained from $L$ by performing surgery along $\partial D$
		\item
		      $\alpha_D(L)$ agrees with $L$ outside of a small neighborhood of $D$
			  		\item
		      If $L$ was embedded and disjoint from the interior of $D$, then $\alpha_D(L)$ has a single self intersection point.
	\end{itemize}
\end{theorem}
Anti-surgery is inverse to Lagrangian surgery in the sense that if $L'$ is obtained from $L$ by resolving a self-intersection point, there exists an anti-surgery disk on $L'$ so that $\alpha_D (L')=L$. 
However, the choices in both surgery and anti-surgery mean that the process of applying anti-surgery followed by surgery need not construct the same Lagrangian. By combining anti-surgery with surgery, we can obtain a new embedded Lagrangian.
We now restrict to complex dimension two.
\begin{definition}[\cite{pascaleff2017wall}]
		Let $L$ be an embedded Lagrangian and $D$ a surgery disk. Let $\alpha_D(L)$ be obtained from $D$ by anti-surgery.
	A \emph{mutation of $L$ along $D$} is the Lagrangian $\mu_D^\epsilon(L)$  obtained from $\alpha_D(L)$ by resolving the resulting single self-intersection point in opposite direction. 
	The width $\epsilon$ of the mutation is the sum of the fluxes swept from $L$ to $\alpha_D( L)$ to $\mu_D^\epsilon( L)$.
	\label{def:mutation}
\end{definition}
An example of Lagrangians obtained by mutation are the monotone Chekanov and product tori in $\CC^2\setminus \{z_1z_2=1\}$, where surgery occurs along the disk $z_1=\bar z_2$.
It is expected that Lagrangians which are related by mutation give different charts on the moduli space of Lagrangian submanifolds in the Fukaya category and that these charts are related by a wall-crossing formula tied to cluster transformation \cite{pascaleff2017wall}.\footnote{We note here that our definition of mutation differs slightly from that in \cite{pascaleff2017wall} as in our setting there is some symplectic flux swept between $L, \alpha_D(L)$ and $\mu_D^\epsilon(L)$. Usually, one requires that the amount of symplectic flux between these Lagrangians is zero. This would be necessary to preserve monotonicity of a Lagrangian submanifold.}
Just as there exists a Lagrangian cobordism between the Chekanov and product tori in $\CC^2$, there exist Lagrangian cobordisms between Lagrangians related by mutation.
\begin{theorem}[Haug, \cite{haug2015lagrangian}]
	Let $L$ be a Lagrangian and let $D$ be an anti-surgery disk for $L$.
	There exists a Lagrangian cobordism between $L$ and $\mu_D^\epsilon(L)$.
\end{theorem}
\begin{definition}[Pascaleff and Tonkonog,\cite{pascaleff2017wall}]
	A \emph{mutation configuration} is a pair $(L, D)$, where $D$ is an anti-surgery disk for $L$ with boundary in a primitive homology class. The \emph{mutation cobordism} is the cobordism $K_{\mu_D}$ with ends on $L$ and $\mu^\epsilon_D(L)$.
	We call $[\partial D^2]\subset H_1(L, \ZZ)$ the mutation class or mutation direction.
	\label{def:mutationcobordism}
\end{definition}
A key feature of Lagrangian mutation is that the Lagrangians $L$ and $\mu^\epsilon_D(L)$ are Lagrangian isotopic.

We will show that the Lagrangian cobordism $K_{\mu_D}$ can induce a continuation map between the Lagrangians $L$ and $\mu^\epsilon_D L$.
The proof attempts to follow the toy model of the suspension cobordism considered in \cref{sec:continuation}. 
However, the mutation cobordism is not topologically a cylinder so we cannot simply copy our proof from \cref{sec:continuation}. 
To overcome this difference, we will prove the following replacements for cylindricity of the cobordism $K_{\mu_D}$:
\begin{itemize}
	\item We first describe the Morse theory of $K_{\mu_D}$ in \cref{subsub:Morsefunction}.
	\item
	\Cref{subsub:exceptionalholomorphicdisk} constructs a holomorphic disk with boundary on this cobordism and \cref{subsub:removingcurvature} shows that the $m^0$ contribution of this disk can be deformed away to higher filtration.
	\item
	We show that the restriction of the differential to the right end of the cobordism can be inverted in \cref{subsub:invertingends}, allowing us to apply a theorem on mapping cocylinders for $A_\infty$ algebras in \cref{subsub:applinghtt}. 
\end{itemize}
This will give us enough leverage to construct a mapping cocylinder from the Lagrangian cobordism $K_{\mu_D}$ and prove the following wall-crossing theorem:

\begin{theorem}
	Suppose that $(L, D)$ is an isolated mutation configuration (\cref{cond:isolatedmutation}).
	Then there exists a deforming cochain $d_\epsilon$ for $\CF(K_{\mu_D})$ so that $\CF_{d_\epsilon}(K_{\mu_D})$ is an $A_\infty$ mapping cocylinder between $\CF_{\beta^-_*(d_\epsilon)}(L^-)$ and $\CF_{\beta^+_*(d_\epsilon)}(L^+)$.  \label{thm:wallcrossingcobordism}
\end{theorem}
\Cref{cond:isolatedmutation} is required to rule out the possible appearance of holomorphic disks 
with boundary on $K_{\mu_D}$ which may appear during the ``suspension'' portion of the cobordism $K_{\mu_D}$. 
\begin{definition}
	Let $\epsilon = \omega(u_{ex})$.
	We say that the mutation $\mu_D$ is isolated if the only holomorphic disk $u$ with boundary on $K_{\mu_D}$ and $\langle[ \partial u], x^\pm \rangle\neq 0$ and $\omega(u)\leq \omega(u_{ex})$ is $u_{ex}$.
	\label{cond:isolatedmutation}
\end{definition}
	We anticipate that it should be possible to decompose a mutation cobordism into a concatenation of suspensions and mutation cobordisms so that the mutation is isolated. The cobordism decomposition framework studied in \cite{hicks2021lagrangian} should provide this decomposition in general.
	We will show that the prototypical example of a mutation cobordism between Chekanov and Product type torus  studied in \cref{app:mutationexamples}  is isolated.
\begin{notation}
	We will say deforming cochain to denote an element $d\in \CF(K_{\mu_D})$ which we use to deform the $A_\infty$ algebra (as in \cite[Definition 3.6.9]{fukaya2010lagrangian}) but which is not necessarily a bounding cochain (i.e. $m^0_{d}$ may or may not be  $0$). Usually we will employ these to deform the curvature in such a way that $\val(m^0_{d})$ larger than some specified amount.
\end{notation}
\subsection{Geometric input}
We will simplify our notation substantially by restricting to the setting where $L=T^2$. 
The Floer cochain complex of $K_{\mu_D}$ is dependent on the data of a Morse function.
By picking an appropriate Morse function, we can simplify the later discussion.
\begin{notation} In this section, let $L^-=L, L^+=\mu_D(L)$.
$K_{\mu_D}$ is constructed from $L^-$ by attachment of a 1 cell and a 0 cell.
However, there is a slightly more symmetric treatment to $K_{\mu_D}$.
Consider an intermediate Lagrangian $L^0$, which is obtained by performing anti-surgery on $L^-$.
This is an immersed Lagrangian 2-sphere.
Both $L^-$ and $L^+$ may be obtained from $L^0$ via attachment of a 1-cell.
By abuse of notation, we will also use $L^-, L^+$ and $L^0$ to refer to the submanifolds of $K_{\mu_D}$ which (upon inclusion into $X\times \CC$ and projection to $X$) correspond to the slices $L^-, L^+$ and $L^0$.
\end{notation}
\subsubsection{ Morse theory for $K_{\mu_D}$}
\label{subsub:Morsefunction}
Consider a Morse-Bott function $\tilde h: K_{\mu_D}\to \RR$ which has a minimum along $L^0$, maximums at the ends $L^\pm$ and two index 2 critical points corresponding to the attached 1-cells (\cref{fig:Morsebottcobordism}.)

We now give a decomposition of the Morse cochain complex, with the goal of understanding where $\CM(K_{\mu_D})$ fails to be a mapping cocylinder. 
The Morse complex  $\CM(K_{\mu_D})$ has the following shape:
\[
	\CM(K_{\mu_D})\;\;= \left(  \setlength\mathsurround{0pt}\begin{tikzcd}[column sep={1em}]
			\CM(L^-) \arrow{dr} \arrow{drr}  & & &  & \CM(L^+) \arrow{dl} \arrow{dll}\\
			& x^- \arrow{r}&  \CM(L^0)[1] & x^+ \arrow{l}
		\end{tikzcd}\setlength\mathsurround{.8pt}\right)
\]
where $x^-$ and $x^+$ are the critical points corresponding to the handles. 
We will call the subspace spanned by $x^-$ and $x^+ $
\[
	E^0:=\Lambda\langle x^-, x^+\rangle.
\]
\begin{figure}
	\centering
	\begin{tikzpicture}[scale=.4]

	\begin{scope}[shift={(-16,-4.5)}, scale=4]
		\fill[gray!20]  (4,-2) rectangle (-2,0);
		\draw[fill=gray!60] (-2,-1) .. controls (-0.5,-1) and (-0.5,-1) .. (-0.5,-1) .. controls (0,-1) and (0,-1) .. (0.25,-1.25) .. controls (0.5,-1.5) and (1.5,-1.5) .. (1.75,-1.25) .. controls (2,-1) and (2,-1) .. (2.5,-1) .. controls (3,-1) and (2.5,-1) .. (3,-1);
		\draw[fill=gray!60] (-2,-1) .. controls (-0.5,-1) and (-0.5,-1) .. (-0.5,-1) .. controls (0,-1) and (0,-1) .. (0.25,-0.75) .. controls (0.5,-0.5) and (1.5,-0.5) .. (1.75,-0.75) .. controls (2,-1) and (2,-1) .. (2.5,-1) .. controls (2.5,-1) and (2.5,-1) .. (4,-1);
		\node[above] at (-1.0833,-1.4169) {$ K $};

		\draw (0.25,-1.25) .. controls (0.125,-1.125) and (0.125,-0.875) .. (0.25,-0.75);
		\draw (1.75,-0.75) .. controls (1.875,-0.875) and (1.875,-1.125) .. (1.75,-1.25);
	\end{scope}

\begin{scope}[shift={(-1,0)}]

\draw (-1.5,0) .. controls (-1,0) and (-1,-1.5) .. (-1,-2);
\draw (0.5,0) .. controls (0,0) and (0,-1.5) .. (0,-2);
\draw[fill=white]  (-0.5,0.5) ellipse (4 and 2);

\begin{scope}[]
\fill[white]  (-2.5,0.5) ellipse (0.5 and 0.5);
\fill[white] (-0.5,3) .. controls (-0.5,3.5) and (0,4) .. (0.5,4) .. controls (1,4) and (3,4) .. (3.5,4) .. controls (4,4) and (4.5,3.5) .. (4.5,3) .. controls (4.5,2.5) and (4.5,-1.5) .. (4.5,-2) .. controls (4.5,-2.5) and (4,-3) .. (3.5,-3) .. controls (3,-3) and (1,-3) .. (0.5,-3) .. controls (0,-3) and (-0.5,-2.5) .. (-0.5,-2) .. controls (-0.5,-1.5) and (-0.5,-1) .. (-0.5,3);

\draw  (-0.5,0.5) ellipse (4 and 2);
\end{scope}
\begin{scope}[shift={(0,0)}]

\draw[dotted]  (-2.5,0.5) ellipse (0.25 and 0.25);
\draw[dotted]  (1.5,0.5) ellipse (0.25 and 0.25);
\draw[dotted]  (-0.5,0.5) ellipse (1.75 and 0.65);
\draw[dotted]  (-0.5,0.5) ellipse (2.25 and 0.8);
\end{scope}

\fill[white] (-1,3) .. controls (-1,2.5) and (-1,1) .. (-1.5,1) .. controls (-1,1) and (0,1) .. (0.5,1) .. controls (0,1) and (0,2.5) .. (0,3);

\draw (-1.5,1) .. controls (-1,1) and (-1,2.5) .. (-1,3);
\draw (0.5,1) .. controls (0,1) and (0,2.5) .. (0,3);
\draw(-1,-2) .. controls (-1,-3) and (-0.5,-3.5) .. (0,-3.5) .. controls (0.5,-3.5) and (3,-3.5) .. (3.5,-3.5) .. controls (4.5,-3.5) and (5,-3) .. (5,-2) .. controls (5,-1.5) and (5,2.5) .. (5,3) .. controls (5,4) and (4.5,4.5) .. (3.5,4.5) .. controls (3,4.5) and (0.5,4.5) .. (0,4.5) .. controls (-0.5,4.5) and (-1,4) .. (-1,3) ;
\draw (0,-1.5) .. controls (0,-2.25) and (0.25,-2.5) .. (0.5,-2.5) .. controls (1,-2.5) and (3,-2.5) .. (3.5,-2.5) .. controls (3.75,-2.5) and (4,-2.25) .. (4,-2) .. controls (4,-1.5) and (4,2.5) .. (4,3) .. controls (4,3.25) and (3.75,3.5) .. (3.5,3.5) .. controls (3,3.5) and (1,3.5) .. (0.5,3.5) .. controls (0.25,3.5) and (0,3.25) .. (0,3);

\end{scope}

\begin{scope}[shift={(-11.5,6)}]

\draw (-1.5,0) .. controls (-1,0) and (-0.85,-1.3) .. (-0.85,-2);
\draw (0.5,0) .. controls (0,0) and (-0.15,-1.3) .. (-0.15,-2);
\draw[fill=white]  (-0.5,0.5) ellipse (3.5 and 1.75);

\begin{scope}[]
\fill[white]  (-2.5,0.5) ellipse (0.5 and 0.5);
\fill[white] (-0.5,3) .. controls (-0.5,3.5) and (0,4) .. (0.5,4) .. controls (1,4) and (3,4) .. (3.5,4) .. controls (4,4) and (4.5,3.5) .. (4.5,3) .. controls (4.5,2.5) and (4.5,-1.5) .. (4.5,-2) .. controls (4.5,-2.5) and (4,-3) .. (3.5,-3) .. controls (3,-3) and (1,-3) .. (0.5,-3) .. controls (-0.15,-3) and (-0.5,-2.5) .. (-0.5,-2) .. controls (-0.5,-1.3) and (-0.5,-1) .. (-0.5,3);

\draw  (-0.5,0.5) ellipse (3.5 and 1.75);
\end{scope}
\begin{scope}[shift={(0,0)}]

\draw[dotted]  (-2.5,0.5) ellipse (0.25 and 0.25);
\draw[dotted]  (1.5,0.5) ellipse (0.25 and 0.25);
\draw[dotted]  (-0.5,0.5) ellipse (1.75 and 0.65);
\draw[dotted]  (-0.5,0.5) ellipse (2.25 and 0.8);
\end{scope}

\fill[white] (-0.85,3) .. controls (-0.85,2.5) and (-1,1) .. (-1.5,1) .. controls (-1,1) and (0,1) .. (0.5,1) .. controls (0,1) and (-0.15,2.5) .. (-0.15,3);

\draw (-1.5,1) .. controls (-1,1) and (-0.85,2.5) .. (-0.85,3);
\draw (0.5,1) .. controls (0,1) and (-0.15,2.5) .. (-0.15,3);
\draw(-0.85,-2) .. controls (-0.85,-3) and (-0.5,-3.35) .. (0,-3.35) .. controls (0.5,-3.35) and (3,-3.35) .. (3.5,-3.35) .. controls (4.5,-3.35) and (4.85,-3) .. (4.85,-2) .. controls (4.85,-1.5) and(4.85,0) .. (4.5,0.5) .. controls (4.85,1) and (4.85,2.5) .. (4.85,3) .. controls (4.85,4) and (4.5,4.35) .. (3.5,4.35) .. controls (3,4.35) and (0.5,4.35) .. (0,4.35) .. controls (-0.5,4.35) and (-0.85,4) .. (-0.85,3) ;
\draw (-0.15,-1.3) .. controls (-0.15,-2.25) and (-0.15,-2.65) .. (0.5,-2.65) .. controls (1,-2.65) and (3,-2.65) .. (3.5,-2.65) .. controls (3.75,-2.65) and (4.15,-2.25) .. (4.15,-2) .. controls (4.15,-1.5) and (4.15,0) .. (4.5,0.5) .. controls (4.15,1) and (4.15,2.5) .. (4.15,3) .. controls (4.15,3.5) and (3.75,3.65) .. (3.5,3.65) .. controls (3,3.65) and (1,3.65) .. (0.5,3.65) .. controls (-0.15,3.65) and (-0.15,3.25) .. (-0.15,3);

\end{scope}

\begin{scope}[shift={(-3.5,-1.5)}]

\begin{scope}[shift={(-11,0)}]

\draw[dotted]  (-2.5,0.5) ellipse (0.25 and 0.25);
\draw[dotted]  (1.5,0.5) ellipse (0.25 and 0.25);
\draw[dotted]  (-0.5,0.5) ellipse (1.75 and 0.65);
\draw[dotted]  (-0.5,0.5) ellipse (2.25 and 0.8);
\end{scope}

\draw  (-11.5,0.5) ellipse (3.5 and 1.5);
\end{scope}
\begin{scope}[shift={(-8,3.5)}]

\begin{scope}[shift={(-11,0)}]

\draw[dotted]  (-2.5,0.5) ellipse (0.25 and 0.25);
\draw[dotted]  (1.5,0.5) ellipse (0.25 and 0.25);
\draw[dotted]  (-0.5,0.5) ellipse (1.75 and 0.65);
\draw[dotted]  (-0.5,0.5) ellipse (2.25 and 0.8);
\end{scope}

\draw  (-11.5,0.5) ellipse (3 and 1.25);
\draw (-12,0.6) .. controls (-11.8,0.6) and (-11.6,0.6) .. (-11.5,0.5) .. controls (-11.4,0.6) and (-11.2,0.6) .. (-11,0.6);
\draw[dotted] (-11.7,0.3) .. controls (-11.6,0.3) and (-11.6,0.3) .. (-11.5,0.5) .. controls (-11.4,0.3) and (-11.4,0.3) .. (-11.3,0.3);
\end{scope}

\begin{scope}[shift={(-22.5,-0.5)}]
\draw[dotted]  (-2.5,0.5) ellipse (0.5 and 0.5);
\draw[dotted]  (1.5,0.5) ellipse (0.5 and 0.5);
\draw[]  (-0.5,0.5) ellipse (1.5 and 0.5);
\begin{scope}[]

\draw[dotted]  (-2.5,0.5) ellipse (0.25 and 0.25);
\draw[dotted]  (1.5,0.5) ellipse (0.25 and 0.25);
\draw[dotted]  (-0.5,0.5) ellipse (1.75 and 0.65);
\draw[dotted]  (-0.5,0.5) ellipse (2.25 and 0.8);
\end{scope}
\draw  (-0.5,0.5) ellipse (2.5 and 1);
\end{scope}

\node at (-23,1.5) {$CM^\bullet(L^-)$};
\node at (-15,1) {$CM^\bullet(L^0)$};

\node at (1,5) {$CM^\bullet(L^+)$};

\draw (-24,-21) .. controls (-22,-17) and (-22,-17) .. (-20,-17) .. controls (-18,-17) and (-14,-19) .. (-12,-19) .. controls (-10,-19) and (-6,-17) .. (-4,-17) .. controls (-2,-17) and (-2,-17) .. (0,-21);
\draw[<-]  (-26,-17) -- (-26,-23);
\draw[->] (-26,-23) -- (2,-23);
\node at (-27,-17) {$\tilde h$};
\node at (1,-22.5) {$t$};

\draw[dotted, ->] (-23,-1.5) .. controls (-23,-2.5) and (-23,-4.5) .. (-23,-5) .. controls (-23,-5.5) and (-23,-6) .. (-21.5,-6) .. controls (-20,-6) and (-20,-6.5) .. (-20,-7) .. controls (-20,-7.5) and (-20,-8) .. (-20,-8.5);
\draw[dotted, ->] (-19.5,2.5) .. controls (-19.5,1.5) and (-19.5,-2) .. (-19.5,-2.5) .. controls (-19.5,-3.5) and (-19.5,-4) .. (-17.5,-4) .. controls (-15.5,-4) and (-15.5,-4.5) .. (-15.5,-5.5) .. controls (-15.5,-6) and (-15.5,-7) .. (-15.5,-7.5);
\draw[dotted, ->] (-14,-3) .. controls (-14,-3.5) and (-14,-4) .. (-14,-4.5) .. controls (-14,-5) and (-14,-5.5) .. (-13,-5.5) .. controls (-12,-5.5) and (-12,-6) .. (-12,-7) .. controls (-12,-7) and (-12,-8) .. (-12,-8.5);
\draw[dotted, ->] (-10.5,2) .. controls (-10.5,1) and (-10.5,0.5) .. (-10.5,-0.5) .. controls (-10.5,-1) and (-10.5,-1.5) .. (-9.5,-1.5) .. controls (-8.5,-1.5) and (-8.5,-2) .. (-8.5,-2.5) .. controls (-8.5,-3) and (-8.5,-6) .. (-8.5,-7);
\draw[dotted, ->] (-3,-2) .. controls (-3,-2.5) and (-3,-4.5) .. (-3,-5) .. controls (-3,-5.5) and (-3,-6) .. (-3.5,-6) .. controls (-4,-6) and (-4,-6.5) .. (-4,-7) .. controls (-4,-7.5) and (-4,-8) .. (-4,-8.5);
\node at (-20,6) {$\langle x^-\rangle$};
\node at (-10,11) {$\langle x^+\rangle$};
\node at (-2,-11) {$\mathbb C$};
\end{tikzpicture} 	\caption{
	We draw the sublevel sets of the Morse-Bott function for $K_{\mu_D}$ at the critical values of $\tilde h$.
	Notice that $L^0$ separates the cobordism into two halves. 	}
	\label{fig:Morsebottcobordism}
\end{figure}
\begin{claim}
	The spaces spanned by 
	\begin{align*}
		\CF(L^-)\oplus \CF(L^0)\oplus E^0\\
		\CF(L^+)\oplus \CF(L^0)\oplus E^0		
	\end{align*}
	are filtered $A_\infty$ ideals of $\CF(K_{\mu_D})$.
\end{claim}
\begin{proof}
	These are the kernels of $\beta^\pm: \CF(K_{\mu_D})\to \CF(L^\pm)$, which by \cref{assum:pearlycompatibility} are $A_\infty$ homomorphisms. 
\end{proof}
\begin{notation}
	For this section, whenever we write an $A_\infty$ product map, it will always mean the $A_\infty$ product structure on $\CF(K_{\mu_D})$. 
\end{notation}
Setting $A^\pm=\CM(L^\pm)$ and $A^0=\CM(L^0)\oplus E^0$, we can write the differential on  $\CM(K_{\mu_D})=A^-\oplus A^0\oplus A^+$ as 
\[
	\underline m^1=\begin{pmatrix}
		(\underline{m})^{-;-} & 0         & 0            \\
		(\underline{m})^{-;0} & (\underline{m})^{0;0} & (\underline{m})^{+;0}\\
		0         		 & 0         & (\underline{m})^{+;+}
	\end{pmatrix}.
\]
where $(\underline{m})^{\pm;\pm}=m^1_{\CM(L^\pm)}$ are the Morse differentials.  
If the ideal $A^+\to A^0$ was null-homotopic --- for example if $(\underline{m})^{+;0}$ was an isomorphism --- then  $\CM(K_{\mu_D})$ would be a mapping cocylinder. We will further decompose this cochain complex by identifying subspaces of $\CM(L^\pm)$ which correspond to the mutation direction and tease out exactly how our complex fails to be a mapping cocylinder. 

The kernel of  $(\underline{m})^{+;0}$ is one dimensional. We pick an element $c^{u+}$ which generates the kernel so that we may write: 
\[
	\ker((\underline{m})^{+;0})=\Lambda\cdot \langle  c^{u+}\rangle. 
\]
Let us now restrict to the 1-cochains $E^\pm:=CM^1(L^\pm)$.Pick elements $c^{w\pm}\in E^\pm$ so that $(\underline{m})^{\pm ;0}(c^{w^\pm})=x^\pm$. Then the classes $c^{w\pm}$ and $c^{u\pm}$ now span the $E^\pm$.
\begin{remark} 
	The class $c^{u+}$ is determined completely by the cobordism as the Morse cohomology class whose downward flow space is the surgery disk.
	However, the selection of $c^{w+}$ requires a choice similar to the choices required to produce local coordinates on the moduli space of Lagrangian tori in \cref{sec:examples}.
\label{rem:choiceofsplitting}
\end{remark}
With this description, the non-cylindricity of $\CM(K_{\mu_D})$ exactly corresponds to the non-surjectivity of the map $(\underline{m})^{+;0}$ in the diagram:
\[E^-\xrightarrow{(\underline{m})^{-;0}} E^0\xleftarrow{(\underline{m})^{+;0}} E^+.\]
In the chosen basis $\{c^{u\pm},c^{w\pm}\} $ for $E^\pm$  and $\{x^-, x^+\}$ for $E_0$, the maps  $(\underline{m})^{\pm ;0}$ are 
\begin{align*}
(\underline{m})^{-;0}=\begin{pmatrix} 0& 1\\0 & 0\end{pmatrix} &  & (\underline{m})^{+;0}=\begin{pmatrix} 0 & 0 \\0&1
	\end{pmatrix}
\end{align*}
\begin{remark}
	In the more general case where $L$ is not a torus, identify $E^\pm$ as preimage of the above defined map $m^{\pm;0}:\CM(L^\pm)\to E^0$.
	This will include the kernel of $m^{\pm}$.
	A choice of basis $\{c^{u\pm},c^{w\pm}\} $ for $E^\pm$ and a splitting $H^1(L^\pm)=H^1(L)^\pm/E^\pm \oplus E^\pm$ will be required to construct a bounding cochain at a later step (see \cref{eq:choiceofsplittingII}).
		The portion which is independent of the handle, $H^1(L^\pm)/E^\pm$, is isomorphic to $H^1(L^0)$. 
	This corresponds to the ``cylindrical'' portion of the Morse theory of $K$, on which  $m^{+;0}$ is invertible at filtration zero.

	The construction of the deforming cochain $d_\epsilon$ in \cref{thm:wallcrossingcobordism} depends on these choices. 
	Taking different choices will yield different deforming cochains $d_\epsilon$, and thus identify $L^-$ and $L^+$ with a different choice of bounding cochain $\beta^\pm_*(d_\epsilon)$.
	We expect that looking at all possible choices in \cref{eq:choiceofsplittingII} identifies a locus of the Maurer-Cartan solutions $\mathcal MC(L^-)$ with a locus in $\mathcal MC(L^+)$, corresponding to the choices of bounding cochains which make $L^-$ and $L^+$ isomorphic objects of the Fukaya category.
	From a mirror symmetry perspective, when $L^-$ and $L^+$ are SYZ fibers we propose this locus is mirror to the intersection between the torus charts constructed around the mirror points to $L^-$ and $L^+$.

	A different viewpoint is  that the $A_\infty$ homomorphism constructed by $\CF_{d_\epsilon}(K_{\mu_D})$ between $\CF_{\beta^-_*(d_\epsilon)}(L^-)$ and $\CF_{\beta^+_*(d_\epsilon)}(L^+)$  induces a map between their Maurer-Cartan solutions. By then appropriately identifying the Maurer-Cartan elements of $\CF_{\beta^\pm_*(d_\epsilon)}(L^\pm)$ with those of $\CF(L^\pm)$, we obtain a pairing between certain Maurer-Cartan elements of $L^-$ with those of $L^+$.
	This is the approach that we take in \cref{app:mutationexamples}.
\end{remark}
\subsubsection{A Holomorphic Disk}
\label{subsub:exceptionalholomorphicdisk}
We now show that for a specific choice of almost complex structure there exists a pseudoholomorphic disk with boundary on $K_{\mu_D}$.
\begin{prop}
	There exists a choice of complex structure so that $K_{\mu_D}$ bounds a regular Maslov index $0$ disk 
	\[
		u_{ex}: (D^2, \partial D^2) \to (X\times \CC, K_{\mu_D}).
	\]
	The contribution of this disk to the $m^0$ term is $x^+ + x^-$ \label{prop:exceptionaldisk}.\end{prop}
\begin{proof}
	We prove the proposition by exhibiting the disk in the local model constructed in \cref{subsec:haugcobordism}.
	Following \cite[Section 2.3]{haug2015lagrangian} we can construct a standard neighborhood $U$ of the Lagrangian anti-surgery disk $D$ identifying $U$ with a subset of $\CC^2$ (which, by abuse of notation, we shall also call $U$).
	We choose a complex structure for $X$ which matches the standard complex structure on $U\subset \CC^2$. 
 	We can pick this identification so that the Lagrangians $L^-$ and $L^+$ when restricted to $U$ match the local model based on the Lefschetz fibration (see \cite{haug2015lagrangian} section 2.3 for implanting the local model).
	 In these coordinates $L_\downarrow\cap L_\uparrow = (0,0)$.
	\begin{align*}
		(L^-\cap U)= & ((L_\downarrow\#L_\uparrow)\cap U) \\
		(L^+\cap U)= & ((L_\uparrow\#L_\downarrow)\cap U).
	\end{align*}
	The restriction $K_U:=K_{\mu_D}\cap (U\times \CC)$ is the eye-shaped Lagrangian cobordism drawn in \cref{fig:catseye}.
	Recall that $K_U$ is obtained by first considering $( L_\downarrow\times \gamma_\downarrow )\cup (L_\uparrow\times \gamma_\uparrow )$, then resolving the self-intersections to obtain $K^{pre}_U$, and subsequently truncating to get a Lagrangian cobordism. 
	The chart $(L_\downarrow\times \gamma_\downarrow )\cup (L_\uparrow\times \gamma_\uparrow)$ contains a holomorphic 2-gon with ends limiting the two self intersection points.
	The goal is to show that when surgery is applied at both self-intersections, the corners of this bigon are rounded to obtain a holomorphic disk with boundary on $K^{pre}_U$.
	The truncation will occur away from the boundary of this holomorphic disk, so the holomorphic disk survives to give a holomorphic disk with boundary on $K_U$.

		The holomorphic bigon $u_2: D^2\setminus\{-1, 1\} \to U\times \CC$ with boundary on $(L_\downarrow\times \gamma_\downarrow )\cup (L_\uparrow\times \gamma_\uparrow)$ has a particularly nice description. Let $u_2' :\RR\times [0,1]\to \CC$ be the holomorphic bigon which parameterizes the area between the curves $\gamma_\downarrow$ and $\gamma_\uparrow$. 
	Then $u_2$ is given by 
	\begin{align*}
		u_2: D^2\setminus\{-1, 1\} \to& U\times \CC\\
		z \mapsto &((0,0), u_2'(z) ).
	\end{align*}
	From this observation, our candidate construction will be to show that $\pi_U^{-1}(0,0)$ has clean $S^1$ intersection with $K_U^{pre}$. The interior of the $S^1$ will then be a holomorphic disk with boundary on $S^1$.

	We need to describe the intersection of $\pi_U^{-1}(0,0)$ with the chart of $K_U^{pre}$ which corresponds to the surgery neck.
	Using the same surgery profile as in \cref{eq:surgeryneck}, the surgery chart for $K_U^{pre}$ is parameterized by
	\begin{equation}
		\{(( a(t)+\jmath b(t))(\hat x+\jmath\hat y), (a(t)+\jmath b(t)) (\hat y+\jmath\hat x),  (a(t)+\jmath b(t)) \hat z) \;:\; \hat x^2+\hat y^2+\hat z^2= 1\}
		\label{eq:neckparameterizationII}
	\end{equation}
	where the first two coordinates are in $U$ and the last coordinate factors through the cobordism parameter.
	By setting $\hat x= \hat y=0$ and $\hat z=1$, we obtain a smooth curve that passes through the surgery neck parameterizing the clean intersection between the complex line $\pi_U^{-1}(0,0)$ and the surgery neck. 
	We call the corresponding disk $u_{ex}:(D^2, \partial D^2)\to (U\times \CC, K_{\mu_D})$. 
			We now show that this disk is regular using a criterion of \cite{oh1995RiemannHilbert}.
	This requires computing the partial Maslov indices of the boundary of $u_{ex}$. 
	The boundary of $u_{ex}$ can be broken up into 4 components: 
	\begin{itemize}
		\item The ``neck portions'', which correspond to where the boundary passes through the surgery neck of $K^{pre}$ and;
		\item The ``fiber trivial portions'', where the boundary is contained in $L_\uparrow\times \gamma_\uparrow$ and $L_\downarrow\times \gamma_\downarrow$. 
	\end{itemize}
	For simplicity, we will assume that the profile curve $a(t)+\jmath b(t)$ draws a quarter circle even though this doesn't provide a smooth Lagrangian surgery neck.
	We also will assume that the curves $\gamma_{\uparrow/\downarrow}$ are quarter circles, and choose coordinates so that $K|_U$ has projection as drawn in \cref{fig:quarters}.
	\begin{figure}
		\centering
		\begin{tikzpicture}

\draw[fill=gray!20] (0,1) .. controls (0.5,1) and (1,1) .. (1.5,1.5) .. controls (1,1) and (1,0.5) .. (1,0);
\draw[fill=gray!20] (-1,0) .. controls (-1,-0.5) and (-1,-1) .. (-1.5,-1.5) .. controls (-1,-1) and (-0.5,-1) .. (0,-1);
\draw[fill=white]  (0,0) ellipse (1 and 1);
\draw[red, thick] (1,0) arc (0:90:1);
\draw[green, thick] (0,1) arc (90:180:1);
\draw[blue, thick] (-1,0) arc (180:270:1);
\draw[orange, thick] (0,-1) arc (270:360:1);
\draw[dotted] (-1.5,0)  -- (1.5,0);
\draw[dotted] (0,1.5) -- (0,-1.5);
\end{tikzpicture} 		\caption{Dividing $K_U$ into four quarters. }
		\label{fig:quarters}
	\end{figure}
	We now consider the a parameterization of the boundary of the holomorphic disk
	$\gamma(\theta): S^1 \to \partial u_{ex}$ so that we may parameterize $TK^{pre}|_{\partial u_{ex}}$ by maps $C\cdot A(\theta): \RR^3\to \CC^3$,
	Here $C$ is a constant matrix, and $A(\theta)$ describes the rotation of the tangent space as we progress through the coordinate $\theta$.
	We subdivide $S^1$ into four quarters corresponding to the  portions of $\partial u_{ex}$ which pass through the right surgery chart, $L_\uparrow \times \gamma_\uparrow$ chart,  left surgery chart, and $L_\downarrow\times \gamma_\downarrow$ chart. 
\begin{itemize}
	\item For $0\leq \theta < \pi/2$ , the curve $\gamma(\theta)$ passes through the right surgery chart.  A basis for the tangent space of \cref{eq:neckparameterizationII} along the curve $(0,0, e^{\jmath\theta})$ is given by 
	\[
		A(\theta)=\begin{pmatrix}
			e^{-\jmath\theta} & 0 & 0\\
			0 & e^{-\jmath\theta} & 0\\
			0 & 0 & e^{\jmath\theta}
		\end{pmatrix}.\]
	\item Next, for $\theta<\pi/2\leq \theta \leq \pi$ the curve $\gamma(\theta)$ passes through the chart parameterized by $L_\uparrow\times \gamma_\uparrow$. The tangent space in the fiber is constant, and we only see rotation of the tangent space in the base component of $X\times \CC$. 
	\[
		A(\theta)=	\begin{pmatrix}
				-\jmath & 0 & 0\\
				0 & -\jmath & 0\\
				0 & 0 & e^{\jmath\theta}
			\end{pmatrix}
	\]
		\end{itemize}
		The rotation of the tangent space through the left surgery chart and $L_\downarrow\times \gamma_\downarrow$ charts is analogous. This shows that the tangent bundle splits as a direct sum of rank 1 pieces.
			The first two pieces rotate a negative half turn, while the component corresponding to the cobordism parameter of $X\times \CC$ makes a full positive rotation.
	The partial Maslov indices of the loop $TK^{pre}|_{\partial u_ex}$ are $-1, -1,$ and $2$.
	Therefore, this is a disk of Maslov index 0, and it is regular by the criterion of \cite[Theorem II]{oh1995RiemannHilbert}.
	
	The homological class of this disk is the loop in the cobordism $K$ which is introduced from the two 1-handle attachments.
	Therefore, the upward flow spaces of $x^\pm$ each meet the boundary of this disk once.
	Since our local model was chosen to have the same $J$-structure as $X$, this gives us an isolated Maslov index $0$ disk $u_{ex}: D^2 \to X\times \CC$ whose boundary lies on the cobordism $K$, and whose boundary contributes to the Morse cohomology class $x^++x^-$. 
\end{proof}

\begin{figure}
	\centering
	\begin{tikzpicture}[rotate=270, scale=.5]
	\draw (-1.5,0) .. controls (-1,0) and (-1,-1.5) .. (-1,-2);
	\draw (0.5,0) .. controls (0,0) and (0,-1.5) .. (0,-2);
	\draw[fill=white]  (-0.5,0.5) ellipse (3 and 1.5);
	\fill[red!20]  (-0.5,0.5) ellipse (1.5 and 0.5);
	\fill[blue!20]  (4,0.5) ellipse (0.5 and 0.25);
	\begin{scope}[]
	\fill[green] (-1,3) .. controls (-1,2.5) and (-1,1) .. (-1.5,1) .. controls (-2,1) and (-2,1.5) .. (-2.5,1.5) .. controls (-3,1.5) and (-3.5,1) .. (-3.5,0.5) .. controls (-3.5,0) and (-3,-0.5) .. (-2.5,-0.5) .. controls (-2,-0.5) and (-2,0) .. (-1.5,0) .. controls (-1,0) and (-1,-1.5) .. (-1,-2) .. controls (-1,-3) and (-0.5,-3.5) .. (0,-3.5) .. controls (0.5,-3.5) and (2.5,-3.5) .. (3,-3.5) .. controls (4,-3.5) and (4.5,-3) .. (4.5,-2) .. controls (4.5,-1.5) and (4.5,2.5) .. (4.5,3) .. controls (4.5,4) and (4,4.5) .. (3,4.5) .. controls (2.5,4.5) and (0.5,4.5) .. (0,4.5) .. controls (-0.5,4.5) and (-1,4) .. (-1,3);
	\fill[white]  (-2.5,0.5) ellipse (0.5 and 0.5);
	\fill[white] (-0.5,3) .. controls (-0.5,3.5) and (0,4) .. (0.5,4) .. controls (1,4) and (2.5,4) .. (3,4) .. controls (3.5,4) and (4,3.5) .. (4,3) .. controls (4,2.5) and (4,-1.5) .. (4,-2) .. controls (4,-2.5) and (3.5,-3) .. (3,-3) .. controls (2.5,-3) and (1,-3) .. (0.5,-3) .. controls (0,-3) and (-0.5,-2.5) .. (-0.5,-2) .. controls (-0.5,-1.5) and (-0.5,-1) .. (-0.5,3);
	
	\draw  (-0.5,0.5) ellipse (3 and 1.5);
	\end{scope}
	\begin{scope}[dotted]
	\draw  (-2.5,0.5) ellipse (0.5 and 0.5);
	\draw  (1.5,0.5) ellipse (0.5 and 0.5);
	\draw  (-0.5,0.5) ellipse (1.5 and 0.5);
	\draw  (-0.5,0.5) ellipse (2.5 and 1);
	\end{scope}
	\fill[white] (-1,3) .. controls (-1,2.5) and (-1,1) .. (-1.5,1) .. controls (-1,1) and (0,1) .. (0.5,1) .. controls (0,1) and (0,2.5) .. (0,3);
	\fill[green] (-1,3) .. controls (-1,2.5) and (-1,1) .. (-1.5,1) .. controls (-1,1) and (0,1) .. (-0.5,1) .. controls (-0.5,1) and (-0.5,2.5) .. (-0.5,3);
	
	\draw (-1.5,1) .. controls (-1,1) and (-1,2.5) .. (-1,3);
	\draw (0.5,1) .. controls (0,1) and (0,2.5) .. (0,3);
	\draw (-1,-2) .. controls (-1,-3) and (-0.5,-3.5) .. (0,-3.5) .. controls (0.5,-3.5) and (2.5,-3.5) .. (3,-3.5) .. controls (4,-3.5) and (4.5,-3) .. (4.5,-2) .. controls (4.5,-1.5) and (4.5,2.5) .. (4.5,3) .. controls (4.5,4) and (4,4.5) .. (3,4.5) .. controls (2.5,4.5) and (0.5,4.5) .. (0,4.5) .. controls (-0.5,4.5) and (-1,4) .. (-1,3) ;
	\draw (0,-1) .. controls (0,-2.25) and (0.25,-2.5) .. (0.5,-2.5) .. controls (1,-2.5) and (2.5,-2.5) .. (3,-2.5) .. controls (3.25,-2.5) and (3.5,-2.25) .. (3.5,-2) .. controls (3.5,-1.5) and (3.5,2.5) .. (3.5,3) .. controls (3.5,3.25) and (3.25,3.5) .. (3,3.5) .. controls (2.5,3.5) and (1,3.5) .. (0.5,3.5) .. controls (0.25,3.5) and (0,3.25) .. (0,3);

	\draw[red, thick](-0.5, -1 ) .. controls (-0.5, -1.5) and   (-0.5, -1.5)    .. (-0.5,-2.1) .. controls (-0.5,-2.5) and (0,-3) .. (0.5,-3) .. controls (1,-3) and (2.5,-3) .. (3,-3) .. controls (3.5,-3) and (4,-2.5) .. (4,-2) .. controls (4,-1.5) and (4,2.5) .. (4,3) .. controls (4,3.5) and (3.5,4) .. (3,4) .. controls (2.5,4) and (1,4) .. (0.5,4) .. controls (0,4) and (-0.5,3.5) .. (-0.5,3) .. controls (-0.5,2.5) and (-0.5,1.5) .. (-0.5,1);
	\draw[red, thick, dotted] (-0.5,1) -- (-0.5,-1);
	
	\end{tikzpicture}
	 	\caption{The red cycle in the cobordism $K$ represents the boundary of an exceptional Maslov index 0 disk contributing to non-trivial $m^0_{K_{\mu_D}}$. The green 2 cycle represents a ``deforming cochain'' which  cancels out the lowest order contribution of this disk.  }
	\label{fig:boundingcochainonk}
\end{figure}

The boundary of this disk is drawn in \cref{fig:boundingcochainonk}. 
This is a holomorphic disk which is isolated in the following sense:
\begin{claim}
	For the standard holomorphic structure, the disk $u_{ex}$ is the only disk contained in $U$ with boundary on $K_{\mu_D}$ in this relative homology class.
	\label{claim:onlydiskoneye}
\end{claim}
\begin{proof}
	Every disk with boundary on $(K_{\mu_D}\cap (U\times \CC))\subset U\times \CC$ in the same relative homology class as $u_{ex}$  gives a class of disk in $H_2(\CC, \pi_\CC(K_{\mu_D}\cap (U\times \CC)))$. The boundary $\pi_\CC(\partial u_{ex})$ (drawn as the red loop in \cref{fig:catseye}) circles around a connected component of $\CC\setminus  \pi_\CC(K_{\mu_D}\cap (U\times \CC))$ whose area is $\omega_{U\times \CC}(u_{ex})$. 	
	The red loop is the minimal loop which encircles this connected component. Therefore, any other holomorphic disk $u: (D^2, \partial D^2)\to (U\times \CC, K_{\mu_D}\cap (U\times \CC))$ with boundary in the same class as $u_{ex}$ will have projection $\pi_\CC\circ u$ with area at least $\omega_\CC(u_{ex})$. 

	As $\omega_{U\times \CC}(u)= \omega_\CC(u)+\omega_U(u)$, we know that the projection $\pi_U\circ u$ has zero area. 
	However, $\pi_U\circ u$ is a holomorphic map so the zero-area condition means that $\pi_U\circ u$ is constant. 
	The proof is completed upon checking for which $p\in U$ the preimage $\pi^{-1}_U(p)\cap (K\cap U)$ contains a loop in the appropriate homology class.
	This homology class necessarily transverses the $(L_\downarrow\times \gamma_\downarrow )$ and $ (L_\uparrow\times \gamma_\uparrow)$ charts, so $p\in L_\uparrow\cap L_\downarrow=\{(0,0)\}$. Therefore $u_{ex}$ is the only such example.
\end{proof}
\subsection{Algebraic Manipulations}
\begin{notation}
	Let $\epsilon$ be the area $\omega_{X\times \CC}(u_{ex})$. 
\end{notation}
The presence of the disk $u_{ex}$ is the first step to proving that $\CF(K_{\mu_D})$ is a mapping cocylinder.
The disk deforms the Morse products in two significant ways.
\begin{itemize}
	\item
	      The holomorphic disk contributes to $m^0$ as $m^0_{[u_{ex}]}=T^{\epsilon}(x^++x^-)$.
	\item
		  The boundary $\partial u_{ex}$ intersects the downward flow space of $c^{u+}$ at a point and the upward flow space of $x^-$ at a point. 
		  		  This will deform the Morse differential on $\CM(K_{\mu_D})$.
		  		  The non-surjectivity of the map $\underline m^{+;0}_{K}$ in the Morse setting (in comparison to the cylindrical setting given by \cref{eq:morsedifferentialcylinder}) is an obstruction to giving $\CM(K_{\mu_D})$ the structure of a mapping cylinder. 
		  In particular, the contribution of $u_{ex}$ to the differential makes
		  $\langle m^1_{[u_{ex}]}(c^{w+}), x^-\rangle=T^\epsilon$, whereas the underlying Morse differential has $ \langle \underline m^1_{[u_{ex}]}(c^{w+}), x^-\rangle=0$.
		  When this deformation is considered, the map  $m^{+;0}:E^+\to  E^0$ is surjective, giving us some hope that $\CF(K_{\mu_D})$ is a mapping cylinder.
\end{itemize}
This nearly proves that $\CF(K_{\mu_D})$ is a mapping cocylinder. 
However, \cref{thm:cylfrommap} requires that  $(m^{+;0})^{-1}(m^0)$ live in positive filtration. 
As a result, we cannot immediately apply \cref{thm:curvedhtt}.
This leaves the following steps to construct a mapping cocylinder:
\begin{enumerate}
	\item
		  \textbf{Deforming to increase  $\val(m^0|_{E^0})$, \cref{subsub:removingcurvature}} 
		  We first equip $K_{\mu_D}$ with a deforming cochain $d_\epsilon$ which will cancel out the curvature contribution of the disk $u_{ex}$, so that the energy of $m_{d_\epsilon}^0$ will be greater than $\epsilon$. \label{item:findboundingcochain}
	\item
		  \textbf{Inverting $(m_{d_\epsilon})^{+;0} $, \cref{subsub:invertingends}} 
		  We show that the map $(m_{d_\epsilon})^{+;0} $ can now be inverted. 
		  This is complicated by the introduction of the deforming cochain $d_\epsilon$. \label{item:applyhpl}
	\item
		\textbf{Applying \cref{thm:cylfrommap}, \cref{subsub:applinghtt}} We conclude that  $\CF_{d_\epsilon}(K_{\mu_D})$ is a filtered mapping cocylinder. 
		As a corollary, we conclude that $K_{\mu_D}$ is unobstructed whenever $L^-$ is. \label{item:makecylinder}
\end{enumerate}
We will make a simplifying assumption (\cref{cond:isolatedmutation}) to complete \cref{item:applyhpl}.

\subsubsection{Deforming to remove $m^0|_{E^0}$.}
\label{subsub:removingcurvature}
When working over the Novikov field we can prove that a map is invertible by showing that it is invertible at low filtration. We will denote by $\val: \Lambda\to \RR\cup\{\infty\}$ the Novikov valuation, with the convention that $\val(0)=\infty$. Following \cite{usher2016persistent}, we will use \emph{non-Archimedean normed vector spaces}, which are $\Lambda$-vector spaces equipped with a filtration function $\ell: A\to \RR\cup\{\infty\}$ satisfying\footnote{To match the conventions from \cite{charest2015floer}, our filtration function has the opposite sign convention of \cite{usher2016persistent}.} the axioms of \cite[Definition 2.2]{usher2016persistent}.

We introduce some notation which will help us use this method of proof. 

\begin{definition}
    Let $(A, \fil_A)$ and $(B, \fil_B)$ be non-Archimedean normed vector spaces.
    The \emph{filtration} of a map $\Theta: A\to B$ is the largest jump in the filtration map under $\Theta$. 
    \[
        \fil_{A, B}(\Theta):= \sup_{v\;:\; \fil_A(v)=0} \{\fil_B(\Theta(v))\}.
    \]
    The \emph{leading order} of a map $\Theta: A\to B$ is the largest drop in filtration map under $\Theta$. 
    \[
        \ord_{A,B}(\Theta):= \inf_{v\;:\; \fil_A(v)=0} \{\fil_B(\Theta(v))\}.
    \]
\end{definition}
Our reason for using filtration of maps will be to construct inverses.
\begin{claim}
    If $\fil_{A, B}(\Theta)<\infty$, then $\Theta$ has a left inverse.
    If this is also a right inverse then 
    \begin{align*}
        \fil_{A, B}(\Theta)=-\ord_{B, A}(\Theta^{-1}) && \fil_{B, A}(\Theta^{-1})=-\ord_{A, B}(\Theta).
    \end{align*}
    \label{claim:filtrationinvertable}
\end{claim}

\begin{proof}
    The condition $\fil_{A, B}(\Theta)<\infty$ states that $\Theta(v)\neq 0$ for any $v\neq 0$, so our map has a left inverse.

    Suppose that $\Theta^{-1}$ is the inverse of $\Theta$. 
    Let $v_i$ be a sequence of vectors with $\fil_A(v_i)=0$ so that $\lim_{i\to \infty }\fil_B(\Theta(v_i))=\fil_{A, B}(\Theta)$.
    Since $\fil_A(\Theta^{-1}\circ \Theta(v_i))=\fil_A(v_i)=0$, we obtain that $\ord_{B, A}(\Theta^{-1})\geq-\fil_{A, B}(\Theta)$. 
    Similarly, let $w_i$ be a sequence of vectors with $\fil_B(w_i)=0$ so that  $\lim_{i\to \infty }\fil_B(\Theta^{-1})(w_i)=\ord_{B, A}(\Theta^{-1})$.
    Since $\fil_B(\Theta\circ \Theta^{-1}(w_i))=\fil_B(w_i)=0$, we obtain that $\fil_{A, B}(\Theta)\leq -\ord_{B, A}(\Theta^{-1})$. 
\end{proof}
If one possesses a bound on the filtration and order of a map, then one can obtain a bound on their sum. 
\begin{claim}
    Suppose that $\fil_{A, B}(\Theta_1)< \ord_{A, B}(\Theta_2)$.
    Then $\fil_{A, B}(\Theta_1+\Theta_2)=\fil_{A, B}(\Theta_1)$.
\end{claim}
\begin{proof}
    First, we note that for any vector $v$,
     \[\fil_B(\Theta_1(v))\leq \fil_{A, B}(\Theta_1)<\ord_{A, B}(\Theta_2)\leq \fil_B(\Theta_2(v))\] and so by application of non-Archimedean triangle inequality:
     \[\fil_B(\Theta_1(v)+\Theta_2(v))=\fil_B(\Theta_1(v)).\]
    Let $v_i$ be a sequence of $\fil_A(v_i)=0$ vectors realizing $\lim_{i\to\infty}\fil_B(\Theta_1(v_i))=\fil_{A, B}(\Theta_1)$. 
    \begin{align*}
        \fil_{A, B}(\Theta_1+\Theta_2)\geq &\lim_{i\to\infty}\fil_{B}(\Theta_1(v_i)+\Theta_2(v_i))
        =\lim_{i\to\infty}\fil_B(\Theta_1(v_i))=\fil_{A, B}(\Theta_1). 
    \end{align*}
    For the other direction, let $w_i$ be a sequence of vectors zero-filtration vectors for which  $\lim_{i\to\infty}\fil_B((\Theta_1+\Theta_2)(w_i))=\fil_{A, B}(\Theta_1+\Theta_2)$ is achieved. 
    \begin{align*}
        \fil_{A, B}(\Theta_1+\Theta_2)=\lim_{i\to\infty}\fil_B(\Theta_1(w_i)+\Theta_2(w_i))
        = \lim_{i\to\infty} \fil_B(\Theta_1(w_i))\leq \fil_{A, B}(\Theta_1).
    \end{align*}
\end{proof}
\begin{definition}
    Let $\Theta: A\to B$ be a map. We say that $\Theta$ has leading terms  $\Theta_{\leq \lambda}$ of filtration $\lambda$ and write 
    \[\Theta = \Theta_{\leq \lambda}+ O(\lambda)\]
    if $\Theta = \Theta_{\leq \lambda}+ R$, with
        \[\fil_{A, B}(\Theta_{\leq \lambda}) = \lambda <\ord_{A, B}(R).\]
\end{definition}
 
We will now assume that $\mu_D$ is an isolated mutation. 

\begin{claim}
	Suppose that $\mu_D( L)$ is an isolated mutation.
	Let $\pi_{E^0}: \CF(K_{\mu_D})\to E_0$ be the standard projection.
	Define $\lambda:=\fil_{E_0}(\pi_{E^0}\circ m^0)$.
		There exists a deforming cochain  $d \in \CF(K_{\mu_D})$ which increases the filtration of the curvature,
	\[
		\fil_{E_0}( \pi_{E^0} \circ m^0_{d})> \lambda,
	\]
	with $\fil(d)\geq \lambda$. 
	\label{claim:boundingcochainexistence}
\end{claim}
\begin{proof}	The restriction $\pi_{E^0}\circ m^1|_{E^-\oplus E^+}: E^-\oplus E^+\to E^0$ surjects, and has a right inverse $D$.
	After choosing this right inverse we can construct 
	\begin{equation}
		d:=-D(m^0)|_{\Lambda\langle x^-, x^+\rangle}.
		\label{eq:choiceofsplittingII}
	\end{equation}
	Because the order of $D$ is zero,  $\fil(d)\geq \lambda$. 
	\begin{align*}
		\fil_{E_0}( \pi_{E^0}\circ m^0_{d})&= \fil_{E_0}\left(  \pi_{E^0}\circ m^0+ \pi_{E^0}\circ m^1(d)+\sum_{k}  \pi_{E^0}\circ m^k(d^{\tensor k})\right)                   \\
																 & \geq \min\left( \fil_{E_0}(\pi_{E^0}\circ(m^0+m^1(d))), \fil_{E_0}\left(\sum_{k}  \pi_{E^0}\circ m^k(d^{\tensor k})\right)\right) \\
	\intertext{By \cref{cond:isolatedmutation}, the lowest energy term of $m^0$ is $T^{\omega(u_{ex})}\cdot(x^-+x^+)$, which exactly cancels $\underline{m}^1(d)$. Therefore, $\fil_{E_0}(\pi_{E^0}(m^0+m^1(d))>\lambda.$
	Since $\fil (m^k(d^{\tensor k}))\geq k\lambda$, we conclude:}
	& > \min(\lambda, 2\lambda) =\lambda.
	\end{align*}
	\end{proof}

\subsubsection{Inverting $m^{+;0}|_{E^+}$.}
\label{subsub:invertingends}

We now choose the deforming cochain from \cref{claim:boundingcochainexistence} so that  $\fil(\langle m^0_{d_\epsilon}, E^0\rangle)>\epsilon$.
Because our mutation is assumed to be isolated (\cref{cond:isolatedmutation}), we obtain a lower bound for the filtration of the deforming cochain,
\begin{equation}
	\val(d_\epsilon)\geq \epsilon.
	\label{eqn:valuationofcochain}
\end{equation}

\begin{claim}
	The map $(m^{+;0})_{d_\epsilon}: E^+\to E^0$ is invertible, with $\ord_{E^+, E^0}(((m^{+;0})_{d_\epsilon})^{-1})\geq -\epsilon$.
	\label{claim:invertingmap}
\end{claim}

\begin{proof}
	To do so, we expand $ (m^{+;0})_{d_\epsilon}$ term wise:
	\[
		(m_{d_\epsilon})^{+;0}|_{E^+}=  m^{+;0}+\left.  m^2_0(d_{\epsilon}\tensor \id +\id\tensor d_\epsilon)\right|_{E^+} \pi_{E^0}\circ\left.\left(\sum_{k> 2}  m^k(d_{\epsilon}^{\tensor k_1}\tensor \id \tensor d_{\epsilon}^{\tensor k_2})\right)\right|_{E^+}
	\]
	and compute the valuation of each term.
		We will compute this map with respect to the basis $\{c^{u+}, c^{w+}\} , \{x^-, x^+\}$. 
	\begin{itemize}
		\item
		The first term $m^{+;0}$ can be explicitly computed as the isolated mutation condition means that the space of pearly flow lines of energy less than or equal to $\epsilon$ are cut out regularly, and the property \cref{lemma:geometricflowlines} means that these are the only pearly flow lines which contribute to the differential at an order of $\epsilon$ :
	\[m^{+;0}=\begin{pmatrix} T^\epsilon & 0\\ T^\epsilon &1\end{pmatrix}+ O(\epsilon).\]
		\item
		The terms $\left.  m^2(d_{\epsilon}\tensor \id +\id\tensor d_\epsilon)\right|_{E^+}$ can be further split by the homology class of the disks which deforms the product $m^2$. 
		      \[
			      \left. m^2\circ(d_\epsilon\tensor \id )\right|_{E^+} = \sum_{\substack{\beta\in H^2(X, L)\\ \omega(\beta)\geq 0}}\left. m^2_\beta\circ(d_\epsilon\tensor \id )\right|_{E^+}
		      \]
			  Whenever $\omega(\beta) \geq 0$, the term $m^2_\beta(d_\epsilon \tensor  c^{u+})$ has filtration at least $\val(d_\epsilon)+\val(\beta)>\epsilon$.
			  So we need only worry about the classical portion product, $(\underline{m})^2(d_\epsilon\tensor  \id)$.
			  Since  $\CM(L^+)\oplus \langle x^+ \rangle \oplus \CM(L^0)$ is an ideal of $\CM(K_{\mu_D})$, 
			  \[\langle( \underline{m})^2(d_{\epsilon}\tensor  c^{u+}), x^-\rangle =0.\]
			  So we can write 
	\[\underline{m}^2(d_\epsilon \tensor \id)+\id\tensor d_\epsilon)|_{E^+}= \begin{pmatrix} 0&0\\ A_{u^+ x^+}& A_{w^+ ,x^+}\end{pmatrix}+ O(\epsilon).\]
		where either $A_{u^+ x^+}, A_{w^+ ,x^+}$ have  filtration $\epsilon$ or are zero.
		\item
		The remaining terms
		$\langle \sum_{k> 2}  m^k(d_\epsilon^{\tensor k_1}\tensor \id \tensor d_\epsilon^{\tensor k_2}), x^-\rangle $ necessarily have filtration at least $2\epsilon$ because the product is filtered and  \cref{eqn:valuationofcochain}.
		
	\end{itemize}
	In summary, in the basis $\{c^{u+}, c^{w+}\} , \{x^-, x^+\}$,
		\[
		\pi_{E^0}\circ (m_{d_\epsilon})^{+;0}|_{E^+}= \begin{pmatrix}
			T^{\epsilon} & 0 \\
			T^{\epsilon} +A_{u^+ x^+} & 1 + A_{w^+ x^+}
		\end{pmatrix}+ O(\epsilon).
	\]
	Since the order $\epsilon$ portion of this map is invertible by \cref{claim:filtrationinvertable}, the map $\pi_{E^0}\circ (m^{+;0})_{d_\epsilon}|_{E^+}$ is invertible. The order of the inverse is at least $-\epsilon$.
\end{proof}	
Note that a similar argument shows that $(m_{d_\epsilon})^{+;0}:\CF_{d_\epsilon}(L^+)\to \CF_{d_\epsilon}(L^0)\oplus E^0$ is invertible. 
\subsubsection{Checking conditions of \cref{thm:cylfrommap}}
\label{subsub:applinghtt}
To show that $\CF(K_{\mu_D})$ is an $A_\infty$ mapping cocylinder (\cref{def:mappingcylinder}) set $A^\pm= \CF_{d_\epsilon}(L^{\pm})$ and $A^0=\CM_{d_\epsilon}(L^0)\oplus E^0$. 
We've proven that $\CF_{d_\epsilon}(K_{\mu_{D}})=A^-\oplus A^0\oplus A^+$ as a vector space. 
By \cref{assum:pearlycompatibility} differential on this complex is of the form 
\[
	\begin{pmatrix}
		(m_{d_\epsilon})^1_{A^-} & 0         & 0            \\
		(m_{d_\epsilon})^{-;0}  & (m_{d_\epsilon})^{0;0} & (m_{d_\epsilon})^{+;0}\\
		0         & 0         & (m_{d_\epsilon})^1_{A^+}
	\end{pmatrix}.
\]
By \cref{claim:invertingmap}, the map $ (m_{d_\epsilon})^{+;0}$ is invertible with order greater than $-\epsilon$. 
By \cref{cond:isolatedmutation}, $\val((m_{d_\epsilon})^{+;0}\circ m^0_{d_\epsilon})\geq 0$. 
We therefore satisfy the conditions for a mapping cocylinder (\cref{def:mappingcylinder} ) and may apply \cref{thm:cylfrommap}.
This concludes the proof of \cref{thm:wallcrossingcobordism}.

\section{Examples: wall-crossings and mutations}
	\label{app:mutationexamples}
	We now explore some applications of \cref{thm:cylindricityofcobordism} and \cref{thm:wallcrossingcobordism}. 
	We review an example from \cite{auroux2007mirror}, computing wall-crossing for Chekanov and product tori in the complement of an anticanonical divisor.
Consider the Lefschetz fibration with total space $\CC^2\setminus\{z_1z_2=1\}$,  
\begin{align*}
	W: \CC^2\setminus \{z_1z_2=1\}\to       & (\CC \setminus \{1\} )  \\
	(z_1, z_2)\mapsto & z_1z_2
\end{align*}
as drawn in Figure \ref{fig:wallcrossing}.  
We symplectically inflate this manifold at the removed divisor by taking the completion along the removed hypersurface. 
This resulting manifold has the same topology as $(\CC^*)^2\setminus \{z_1z_2=1\}$. The new symplectic form can be chosen so that the projection $W:(\CC^2\setminus \{xy=1\})\to (\CC\setminus\{1\})$ remains a symplectic fibration. 
The regular fibers of this map have the topology of $\CC^*$ and can be given an SYZ fibration
\begin{align*}
	\syz_{W^{-1}(z)}: W^{-1}(z)\to \RR &&
	(z_1, z_2)\mapsto |z_1|^2-|z_2|^2.
\end{align*}
which is a restriction of the global Hamiltonian to a fiber.
 The Lefschetz fibration has a single degenerate fiber $z_1z_2=0$.  
The base of the fibration can also be equipped with an SYZ fibration.
We will take the fibration of the base given by loops $\gamma_r(\theta):=1+r e^{2\pi \jmath \theta}$.
The symplectic parallel transport map given by the Lefschetz fibration preserves the isotopy class of SYZ fibers of $W^{-1}(z)$;
as a result, one can build an SYZ fibration for the total space $\CC^2\setminus \{z_1z_2=1\}$ by taking the circles  $\syz_{W^{-1}(z)}^{-1}(s)$  and parallel transporting them along circles $\gamma_r(\theta)$ of the second fibration to obtain Lagrangian tori
\[
	L_{\gamma_r, |w| }= \left\{(z_1,z_2)\in \CC^2\setminus\{z_1z_2=1\} \; \middle| \; z_1z_2\in \gamma_r, |z_1|^2-|z_2|^2= \log{|w|}\right\} .
\]
The resulting SYZ fibration has one degenerate fiber which occurs when $\log|w|=0$ and $r$ approaches 1. 
The degenerate fiber $L_{\gamma_1, 1}$ is the Whitney sphere, an immersed Lagrangian sphere with a single double point.
\begin{figure}
	\centering
	\begin{tikzpicture}

    \begin{scope}[shift={(1,-1)}]
    
    \draw  (13.5,0) ellipse (0.5 and 1);
    \draw (13.5,-1) -- (19.5,-1);
    
    \draw (13.5,1) -- (19.5,1);
    \end{scope}
    
    \begin{scope}[shift={(16.65,0)}, scale=0.25]
    \draw  (-1.5,7.5) ellipse (1 and 0.5);
    \draw (-2.5,7.5) .. controls (-2,6) and (-2,6) .. (-2,5.5) .. controls (-2,5) and (-2,5) .. (-2.5,3.5);
    \draw  (-1.5,3.5) ellipse (1 and 0.5);
    \draw (-0.5,7.5) .. controls (-1,6) and (-1,6) .. (-1,5.5) .. controls (-1,5) and (-1,5) .. (-0.5,3.5);
    \draw[red]  (-1.5,6.5) ellipse (0.65 and 0.25);
    \end{scope}

    \begin{scope}[shift={(18.55,0)}, scale=0.25]
    \draw  (-1.5,7.5) ellipse (1 and 0.5);
    \draw  (-1.5,3.5) ellipse (1 and 0.5);
    \draw (-2.5,7.5) -- (-0.5,3.5);
    \draw (-2.5,3.5) -- (-0.5,7.5);
    \node[fill, circle, scale=.3, purple] at (-1.5,5.5) {};
    \end{scope}
    \draw[->] (16.3,0.5) -- (16.3,-1);
    \draw[->] (18.2,0.5) -- (18.2,-0.95);
\begin{scope}[shift={(-0.7,2.5)}]
    \begin{scope}[]

	\clip  (16.5,-5) rectangle (15.5,-2);
    \draw[dashed, red]  (16.5,-3.5) ellipse (0.5 and 1);
\end{scope}
\begin{scope}[]

\clip  (16.5,-5) rectangle (17.5,-2);
    \draw[red]  (16.5,-3.5) ellipse (0.5 and 1);
\end{scope}
\end{scope}
\begin{scope}[shift={(4,2.5)}]
    \begin{scope}[]

	\clip  (16.5,-5) rectangle (15.5,-2);
    \draw[dashed]  (16.5,-3.5) ellipse (0.5 and 1);
\end{scope}
\begin{scope}[]

\clip  (16.5,-5) rectangle (17.5,-2);
    \draw[]  (16.5,-3.5) ellipse (0.5 and 1);
\end{scope}
\end{scope}
\begin{scope}[shift={(2.75,2.5)}]
    \begin{scope}[]

	\clip  (16.5,-5) rectangle (15.5,-2);
    \draw[dashed, blue]  (16.5,-3.5) ellipse (0.5 and 1);
\end{scope}
\begin{scope}[]

\clip  (16.5,-5) rectangle (17.5,-2);
    \draw[blue]  (16.5,-3.5) ellipse (0.5 and 1);
\end{scope}
\end{scope}
\draw[<->] (15.7,1.6) -- (15.7,1.4);
\node[left] at (15.6,1.5) {$\log|w|$};
\node[right] at (16.3,-1) {$\gamma$};
\node at (17.85,-1) {$0$};
\node at (17.5,-1.5) {$\mathbb C$};
\node[left] at (14,-1) {$W=1$};
\node[right] at (21,-1) {$W=\infty$};
\node[red] at (16,-2.5) {Chekanov};
\node[blue] at (19.5,-2.5) {product};
\node at (18.2,-1) {$\times$};
\end{tikzpicture} 	\caption{Symplectic fibration over $W: (\CC^2\setminus\{z_1z_2=1\})\to (\CC\setminus\{1\})$.
	Lagrangian tori are created by parallel transporting cycles in the fibers by loops in the base. Loops on the left of the critical value are of Chekanov type, while those on the right are product type.  
}
\label{fig:wallcrossing}
\end{figure}
We may generalize this construction to Lagrangians $L_{\gamma, |w|}$ for curves $\gamma: S^1\to \CC\setminus \{1\}$ which wind around the removed point a single time.
Such curves are divided into three types: the \emph{Chekanov} type curves which additionally wind around the origin,  the \emph{product} type curves which do not, and those curves $\gamma$ which contain the origin.
If the curve $\gamma$ contains the origin, we say the Lagrangian $L_{\gamma, |w|}$ is on the wall between Chekanov and product type.
By lifting these curves to Lagrangian submanifolds via parallel transport of cycles in the fibers, we obtain the Chekanov and product type Lagrangian tori in $(\CC^*)\setminus {z_1z_2=1}$. 
We will denote the Chekanov-type Lagrangian tori as $L^-_{\gamma, |w|}$ and the product type tori as $L^+_{\gamma,|w|}$. 
\subsection{Uncorrected charts on the moduli space}
\label{subsubsec:choiceofcoordinates}
Wall-crossing for Lagrangian submanifolds is phenomenon which occurs when we try to parameterize the  space of these Lagrangian tori with coordinates.
These coordinates should be constructed over the Novikov ring, although for the purposes of this exposition we will use complex coordinates and unitary local systems. 
We consider Lagrangian branes, which are tuples $(L_{\gamma, |w|},b)$, where $b\in H^1(L, \jmath\RR)$ gives us a unitary local system on $L_{\gamma}$ via deformation of the Floer cohomology following \cite[Lemma 4.1]{auroux2007mirror}.
The space of Hamiltonian isotopy classes of Chekanov (respectively product) Lagrangian branes comes with local coordinates from measuring the flux of an isotopy, which we now describe.

Let $(L_0, b_0)$ and $(L_1, b_1)$ be two Lagrangian submanifolds equipped with local systems and let $L_t$ be a Lagrangian isotopy between these two Lagrangians. 
Fix $c\in C_1(L_0)$. The Lagrangian isotopy gives a cylinder $c \times I\subset X$ with boundary $(c\times\{0\})\sqcup (c\times \{1\})\subset L_0\sqcup L_1$. 
The \emph{flux} of this isotopy along $c_0$ is the quantity
\[
	\Flux_{L_t}(c):=-\left( \int_{c\times \{0\}}b_0\right)+ \left( \int_{c\times \{1\}}b_1\right)+ \left(\int_{c_0\times I} \omega \right).
\]
This defines a complex valued cohomological class $\Flux_{L_t}\in H^1(L_0, \CC)$. 
After picking a base point for our moduli space and a basis for homology, the flux cohomology class gives us local coordinates on the moduli space of Lagrangians up to Hamiltonian isotopy. 

In our example of $X=\CC^2\setminus \{z_1z_2=0\}$, one choice of basis comes from compactifying $X$ to $\CC^2$ and noting that  $L^{\pm }_{\gamma, |w|}\subset(\CC^2)$ now bound holomorphic disks whose boundary class pick out a basis for homology.  
\begin{itemize}
	\item
	 On the Chekanov family, we call these two classes $c_w$ and $c_u$. 
	The class $c_w\in H_1(L^-_{\gamma, |w|})$ is the class of the circle in a fiber of the moment map which is parallel transported around to obtain the Lagrangian $L^-_{\gamma, |w|}$.
	When $|w|=1$, this is the vanishing cycle of the Lefschetz fibration. 
	The second class is obtained by compactifying the total space to $\CC^2$. In $\CC^2$, there is a family of Maslov-2 disks holomorphic disks with boundary on $L^-_{\gamma, |w|}\subset\CC^2$.
	The homology class of the boundary of such a class is called $c_u$. 
	\item
	For the product family, we call these two classes $c_{r}, c_{s}$. 
	They are both obtained from considering $L^+_{\gamma, |w|}$ inside the compactification $\CC^2$, where the Lagrangian torus bounds 2 families of Maslov-2 holomorphic disks.
	If $L^+_{\gamma,|w|}$ is the standard product torus $(r_1e^{\jmath\theta_r}, r_2 e^{\jmath\theta_s})$, then the classes $c_{r}$ and $c_{s}$ correspond to the meridional and longitudinal classes of the product torus. 
	We call the corresponding classes of disks $c_r, c_s\subset H_1(L^+_{\gamma, |w|})$. 
\end{itemize}
Once we have fixed these homology classes, we can construct coordinates on the space of Chekanov (respectively product) Lagrangian branes by measuring flux against a fixed Lagrangian. 
Fix the loop $\gamma_0=1+e^{\jmath\theta}$. 
This loop is neither of Chekanov or product type; however, it still makes sense to measure the flux against the Whitney sphere $L_{\gamma_0, 1}$ for both the Chekanov and product Lagrangians. 
\begin{claim}
	Let $L_t$ be a Lagrangian isotopy with $L_0=L_{\gamma_0, 1}$ and $L_t$ of Chekanov (resp product) type for all $t\in (0, 1]$. 
	Then the flux class $\Flux_{L_t}\in H^1( L_1, \CC)$ depends only on the Lagrangian $L_1$. 
\end{claim}
The flux constructs local coordinates on the space of Chekanov (resp product) type tori. 
\begin{definition}
	For $(u, w)\in (\CC^*)^2$, define the classes
	\[
		[L_{u, w}^-]_{\Flux}:= \{ (L_{\gamma, |w|}^-, b)\;:\; \text{(*)}, \exp(\Flux_{L_t}(c_u))=u,\exp (\Flux_{L_t})(c_w)=w \}.
	\]
	For $(r,s)\in (\CC^*)^2$, define the class of Lagrangians
	\[
		[L_{r, s}^+]_{\Flux}:= \{ (L_{\gamma, |w|}^+, b)\;:\; \text{(**)}, \exp(\Flux_{L_t}(c_r))=r, \exp(\Flux_{L_t}(c_s))=s \}.
	\]
	The condition $(*)$ means that $L_t$ is an isotopy of Chekanov type tori starting at $L_{\gamma_0, 1}$ and ending at $(L_{\gamma, |w|}^-, b)$. 
	The condition $(**)$ means that $L_t$ is an isotopy of product type tori starting at $L_{\gamma_0, 1}$ and ending at $(L_{\gamma, |w|}^+, b)$. 
\end{definition}
These classes described are subsets of the equivalence classes of Lagrangians under the relation of Hamiltonian isotopy. 
\begin{claim}
	The classes  $[L_{u,w}^-]_{\Flux}$ (resp. $[L_{r, s}^+]_{\Flux}$) are the equivalence classes of Chekanov (resp. product) tori under the equivalence relation of Hamiltonian isotopy through Chekanov (resp. product) tori. 
	Furthermore, with the standard complex structure, no member of such an isotopy will bound a holomorphic disk in $(\CC^*)^2\setminus \{z_1z_2=1\}$. 
\end{claim}
The classes $[L_{u, w}^-]_{\Flux}$ and $[L_{r, s}^+]_{\Flux}$ allow us to use the coordinates $(u, w)$ and $(r, s)$ to parameterize charts on the moduli space Lagrangian tori.
We will frequently refer to a specific representative of each class as $L_{u, w}^-$ or $L_{r, s}^+$.
Notice that each Lagrangian torus $L_{\gamma_r, |w|}$  from \cref{fig:wallcrossing} is either of Chekanov type,  product type, or lies on a ``wall'' between these two types of Lagrangian tori.
If $L_{\gamma_r}$ is not on the wall, then  $L_{\gamma_r, |w|}$ belongs to a distinct $[L^\pm]_{\Flux}$ class.  
 
	For purposes of exposition, we will assume for the remainder of this discussion that we can work safely with complex coefficients and avoid convergence issues. 
	Consider the identification between bounding cochains and local systems:
		\begin{equation}
			(L^+_{r, s}, a \cdot c^{r}+ b\cdot c^s) \mapsto (L^+_{r\cdot \exp( a), s\cdot \exp (b)},0)
		\label{eq:wishful}
	\end{equation}
	When the underlying chain model for the Lagrangian is the de Rham complex, this identification is known: we first learned of this in \cite[Lemma 4.1]{auroux2009special} and \cite{fukaya2010cyclic} constructs a version of Lagrangian Floer theory based on the de Rham complex for which this holds.
	We will assume that the same identification holds in the pearly model. 
	\begin{assumption}
		The identification given in \cref{eq:wishful} holds for the pearly model.
	\end{assumption}
	To justify the appearance of the exponential in the pearly-model, one needs to compute the equivalence between the flow-tree model for Morse theory and the de Rham model; an equivalence between these two chain models is exhibited in \cite{kontsevich2001homological}.
	\subsection{Correction from Lagrangian cobordism}
	\label{subsec:monotonecrossing}
	\label{sec:examples}
	
The main computation for this example is the following sharpening of the minimal energy disk requirement of \cref{cond:isolatedmutation}, which was necessary to apply our wall-crossing computation, to the specific example of $K_{\mu_D}.$
We note that there is a relation between the flux of the surgeries used to construct the cobordism, the flux relative boundary of the isotopy $\theta_{H_t}$ used to construct $K_c$ from \cref{claim:cobordismconstruction}, and the area of the holomorphic disk from \cref{prop:exceptionaldisk}. 
Let $\width(L_\downarrow \# L_\uparrow)$ be the flux of the Lagrangian isotopy between the surgery and the immersed Lagrangian $L_{\downarrow}\cup L_{\uparrow}$. 
We arrange for $L_{\downarrow}\cup L_{\uparrow}\cup L_c$ to be exactly homotopic to $L_{\gamma_0, 1}$, the standard Whitney sphere from which we compute our flux coordinates.
Define the quantities: 
    \begin{align*}
    u_1=\exp(-\width_{L_\downarrow\# L_\uparrow}) && u_2= \exp(\Flux_{\theta_{H_t}(L_c)}(\ell_c))\\
    s=\exp(\width_{L_\uparrow\# L_\downarrow})&&
    z=\exp\left(\int_{\partial u_{ex}}b+\int_{u_{ex}}\omega\right)
    \end{align*}
    where $\ell_c$ is the list of the curve $\ell_c$ from \cref{fig:cobordismpieces} to the $L_c$ chart. 
    From the definition of flux and surgery width, we see that the product side of the mutation cobordism $K_{\mu_D}$ is the Lagrangian $L^+_{s,s}$.
    The flux between the Chekanov end of the mutation cobordism and the Whitney sphere is the amount swept out by the surgery neck, along with the flux relative boundary of the component $L_c$. Write $u=u_1\cdot u_2$ so that the Chekanov end of the mutation cobordism is $L^-_{u,1}$
     From relations in the homology of $K_{\mu_D}$, we observe that $u_2=(u_1)^{-1}\cdot s\cdot z$, which allows us to reexpress
    \begin{equation}
        s=u/z.
    \label{eq:modulicoordinates}
    \end{equation}
\begin{prop}
    For the standard choice of complex structure, the only pseudoholomorphic disk with boundary on $K_{\mu_D}$ is the one described in \cref{prop:exceptionaldisk} and its multiple covers. 
    \label{prop:uniquedisk}
\end{prop}
\begin{proof}
    We show that this disk is unique. 
    We look at the Lagrangian cobordism $K_{\mu_D}\subset X\times \CC$ under the projections $\pi_X: X\times \CC\to X$, $\pi_\CC:X\times \CC\to \CC$ and $W: X\to \CC$. 
    The first projection that we look at is $W\circ \pi_X: K_{\mu_D}\to \CC$.
    The regions where we have performed Hamiltonian isotopy and surgery sweep out flux corresponding to shaded regions in the projection \cref{fig:restrictingdisks}.
    We let $U$ be the neighborhood drawn in \cref{fig:restrictingdisks}.
    Every disk with boundary on $K_{\mu_D}$ must either have boundary contained within the red region $U$, or will have boundary completely disjoint from the red region $U$ by the open mapping principle. 

    \begin{figure}
        \centering
        \begin{tikzpicture}
	\begin{scope}[shift={(5,0)}]
		\draw  (2,2) rectangle (-2,-2);
		\node at (0.5,0) {$\times$};
		\draw (0.5,1) .. controls (0.5,0.5) and (0,0.5) .. (0,0) .. controls (0,-0.5) and (0.5,-0.5) .. (0.5,-1);
\node at (0,-2.5) {$W(L^-_{u,0})$};
		\draw (0.5,-1) .. controls (0.5,-1.5) and (-0.5,-1.5) .. (-0.5,-1) .. controls (-0.5,-0.5) and (-1.5,-1) .. (-1.5,0) .. controls (-1.5,1) and (-0.5,0.5) .. (-0.5,1) .. controls (-0.5,1.5) and (0.5,1.5) .. (0.5,1);
	\end{scope}

	\begin{scope}[shift={(0,0)}]
		\draw  (2,2) rectangle (-2,-2);
		\node at (0.5,0) {$\times$};
		\node at (0,-2.5) {$W(L^+_{s, s})$};
		\draw (0.5,1) .. controls (0.5,0.5) and (1,0.5) .. (1,0) .. controls (1,-0.5) and (0.5,-0.5) .. (0.5,-1);
		\draw (0.5,1) .. controls (0.5,1.5) and (-0.5,1.5) .. (-0.5,1) .. controls (-0.5,0.5) and (-0.5,-0.5) .. (-0.5,-1) .. controls (-0.5,-1.5) and (0.5,-1.5) .. (0.5,-1);
	\end{scope}
	\begin{scope}[shift={(10,0)}]
	\fill[red!20] (0,1) .. controls (-0.25,1) and (-0.25,1) .. (-0.25,0.75) .. controls (-0.25,0.5) and (-0.25,-0.5) .. (-0.25,-0.75) .. controls (-0.25,-1) and (-0.25,-1) .. (0,-1) .. controls (0.25,-1) and (0.75,-1) .. (1,-1) .. controls (1.25,-1) and (1.25,-1) .. (1.25,-0.75) .. controls (1.25,-0.5) and (1.25,0.5) .. (1.25,0.75) .. controls (1.25,1) and (1.25,1) .. (1,1) .. controls (0.75,1) and (0.25,1) .. (0,1);

	\draw[fill=gray!20] (-0.5,1) .. controls (-0.5,0.5) and (-0.5,-0.5) .. (-0.5,-1) .. controls (-0.5,-0.5) and (-1.5,-1) .. (-1.5,0) .. controls (-1.5,1) and (-0.5,0.5) .. (-0.5,1);
\draw[fill=gray!20] (0.5,1) .. controls (0.5,0.5) and (1,0.5) .. (1,0) .. controls (1,-0.5) and (0.5,-0.5) .. (0.5,-1) .. controls (0.5,-0.5) and (0,-0.5) .. (0,0) .. controls (0,0.5) and (0.5,0.5) .. (0.5,1);

\draw (0.5,1) .. controls (0.5,0.5) and (1,0.5) .. (1,0) .. controls (1,-0.5) and (0.5,-0.5) .. (0.5,-1);
		\draw (0.5,1) .. controls (0.5,1.5) and (-0.5,1.5) .. (-0.5,1) .. controls (-0.5,0.5) and (-0.5,-0.5) .. (-0.5,-1) .. controls (-0.5,-1.5) and (0.5,-1.5) .. (0.5,-1);

		\draw  (2,2) rectangle (-2,-2);
		\node at (0.5,0) {$\times$};
		\draw (0.5,1) .. controls (0.5,0.5) and (0,0.5) .. (0,0) .. controls (0,-0.5) and (0.5,-0.5) .. (0.5,-1);

		\node at (0,-2.5) {$W\circ \pi_X(K_{\mu(D)})$};
		\draw (0.5,-1) .. controls (0.5,-1.5) and (-0.5,-1.5) .. (-0.5,-1) .. controls (-0.5,-0.5) and (-1.5,-1) .. (-1.5,0) .. controls (-1.5,1) and (-0.5,0.5) .. (-0.5,1) .. controls (-0.5,1.5) and (0.5,1.5) .. (0.5,1);
	\end{scope}
\node at (11,0.75) {$U$};
\end{tikzpicture}         \caption{Disks in the Lagrangian cobordism $K_{\mu_D}$ must either have boundary contained inside of the red region, or outside of the red region.}
                \label{fig:restrictingdisks}

    \end{figure}
    We now look at $K_{\mu_D}\cap (W\circ \pi_X)^{-1}(U)$. This is the eye-shaped cobordism from \cref{fig:catseye}.
    By \cref{claim:onlydiskoneye} the only disks which appear here are $u_{ex}$ and its multiple covers.
    The complement of the region given by $U$ cannot bound holomorphic disks for topological reasons.
    This characterizes the disks which may appear on $K_{\mu_D}$. 
\end{proof}
\begin{corollary}
    The mutation cobordism $K_{\mu_D}:L^-_{u, 1}\rightsquigarrow L^+_{s,s}$ is an isolated mutation.
\end{corollary}

\subsubsection{Wall-Crossing computation}

\label{subsec:wallcrossingformulafromcobordism}
We now compute $\pi_*^\pm( d_\epsilon)$ from \cref{thm:wallcrossingcobordism}. 
From our discussion on orientations, we have that the Lagrangian cobordism $K_{\mu_D}$ does not identify $L^-_{u,1}$ with $L^+_{s, s}$ with the standard choices of spin structures.
We therefore will denote the ends of the cobordism by $\mathring L^-_{u,1}$ and $\mathring L^+_{s,s}$ to signify that these Lagrangians have a different spin structure than the Lagrangians previously considered.
This is important as Lagrangians equipped with different spin structures can represent distinct objects in the Fukaya category.
We choose the Morse functions for the $\mathring L^-_{u,1}$ and $\mathring L^+_{s,s}$  matching the one chosen in \cref{subsubsec:choiceofcoordinates} so that the coordinates for the moduli spaces can be identified with our previous computation.

The vanishing cycle on the Chekanov side is $c^u$, while the vanishing cycle on the product side is $c^r+c^s$. 
To construct the continuation map from a mutation cobordism, we must pick a splitting of the vector spaces $E^\pm$ from \cref{subsub:Morsefunction}. This involves picking classes $c^{w\pm}$ as in \cref{rem:choiceofsplitting}
On the Chekanov side we take the class $c^{w-}:=c^w$ and on the product side we choose the class $c^{w+}:=c^r$.

With this choice of Morse function, we show with \cref{assumption:singlediskexistence} the curvature term is:
\begin{equation}
m^0_{K_{\mu_D}}= \log(1+z)(x^++x^-).
\label{eq:curvatureterm}
\end{equation}
\begin{assumption}
    For each $k$, there exists a domain-dependent perturbation $J_z^k: D^2\to \End(TX\times \CC)$ so that there is a unique $J_z^k$ pseudoholomorphic disk $u_k: (D^2, \partial D^2)\to (X\times \CC, K_{\mu_D})$ with the property that $[u_k]=k[u]$.
    \label{assumption:singlediskexistence}

\end{assumption}
\begin{proof}[Proof of \cref{eq:curvatureterm} given \cref{assumption:singlediskexistence}]
    The $A_\infty$ product structure defined in \cite{charest2015floer} is defined by the count of adapted treed holomorphic disks weighted by the number of interior leaves.
    The curvature term, up to sign, is given by
    \[m^0= \sum_{x_0, u\in \mathcal M(X, L, D, x_0)_0} \pm  (\sigma(u)!)^{-1} T^{\omega(u)} x_0\]
    where the $\pm$ sign is determined by the orientation of the moduli space, and $\sigma(u)$ is the number of interior leaves.

    Consider the preimages  of the stabilizing divisor $\{z_1, \ldots z_k\}=u^{-1}_k(D^2)\subset D^2$. We now define a domain-dependent perturbation for the combinatorial type $\underline \Gamma^k$ which has 1 disk with $k$-interior marked points and 1 output flow-line.
    Each $C\in \mathcal M(\underline \Gamma^k)$ with disk component $D^2_C$ is labeled with interior marked points $p_1, \ldots p_k$ and attachment point $s_0\in \partial D^2_C$ for the flow line. 
    For each disk, there exists a unique automorphism $\phi_C:D^2_C\to D^2$ sending 
    \begin{align*}
        \phi_C(p_1)=z_1 && \phi_C(s_0)=1.
    \end{align*}
    We define the perturbation of almost complex structure $\mathcal P(\underline \Gamma^k): \mathcal U\to TX$ by pullback of the almost complex structure from \cref{assumption:singlediskexistence},  
    \[\mathcal P(\underline \Gamma^k)(C, z)=J_{\phi_C(z)}^k\]
    By \cref{assumption:singlediskexistence}, for this choice of perturbation,  there exists a unique pseudoholomorphic treed disk $u_C: C\to X$ with boundary on $L$ and flow line limiting to $x_0$.
    However, most of these disks are not adapted to the divisor $D$, as there is no reason that $u^{-1}_C(D)=\phi^{-1}(\{z_1, \ldots, z_k\})$ should be the same as the set of marked points $p_1, \ldots, p_k$.
    In fact, we are only guaranteed that $\phi^{-1}(z_1)=p_1$ by construction.
    
    The only difference between the domains $C$ are the positions of the marked points, and so the moduli space of $\mathcal P$-perturbed $D$-adapted pseudoholomorphic treed disks are in bijection with the domains for which $\phi^{-1}(\{z_1, \ldots z_k\})=\{p_1, \ldots, p_k\}$.
    There are $(k-1)!$ such domains (as $p_1$ must be mapped to $z_1$, and the remaining $k-1$ marked points can be chosen freely).

    This means that (given \cref{assumption:singlediskexistence}), there exist choices of perturbation datum so that $\#\mathcal M_{\mathcal P}(X, L, D, \underline \Gamma^k)=(k-1)!$.
    This gives us ($\mod 2$) the formula for multiple covers, and (up to a sign) 
    \[\langle m^0, x_0 \rangle=\sum \frac{(k-1)!}{k!} T^{k\cdot \omega [u]}= \log(1+z).\]
\end{proof}
The deformation from  \cref{claim:boundingcochainexistence} can be extended to a bounding cochain:
\[
    b= \log(1+z) \cdot c^w+ \log(1+z)\cdot c^{s}
\]
By restricting $b_\epsilon$ to the ends of the cobordism, we obtain a continuation map between $\CF_{b^-} (\mathring L^-_{u, 1})\to \CF_{b^+} (\mathring L^+_{s,s})$. 

If we assume convergence over complex coefficients and identify the bounding cochain with a local system, we obtain a correspondence between Lagrangians\begin{align*}
    (\mathring L^-_{u, 1+z}) \sim  (L^-_{u,1} ,\log(1+z) \cdot c^w))\sim&(\mathring L^+_{s, s},\log(1+z)\cdot c^{s})\sim (\mathring L^+_{s, s(1+z)})\\
    \intertext{Setting $w=(1+z)$ and noting that $s=u/z$ from \cref{eq:modulicoordinates}}
    (\mathring L^-_{u, w})\sim & \mathring L^+_{u/(w-1), uw/(w-1)}
\end{align*}
This closely matches the identification from \cite{auroux2007mirror}, with the only difference being the signs.
This discrepancy can be explained by the spin structures on the Lagrangian cobordism $K_{\mu_D}$.
As the Lagrangian cobordism $K_{\mu_D}$ admits an embedding into $\RR^3$, it has a trivial tangent bundle and is therefore spin. However, there is no spin structure on  $K_{\mu_D}$ which restricts to the standard spin structures \cite{cho2004holomorphic} on both $L^{\pm}$.   

	\appendix
	
\section[Mapping cylinders for curved A-infinity algebras]{Mapping cylinders for curved $A_\infty$ algebras}
	\label{app:homological}
		\subsection{A brief review of filtered $A_\infty$ algebras and homomorphisms. }
We will use notation for filtered $A_\infty$ algebras and homomorphisms following \cite{fukaya2010lagrangian}.
We make some modifications to extend to the setting of weakly-filtered $A_\infty$ algebras and homomorphisms.
\begin{notation}
	We will sometimes write the composition $f\circ g$ as $\begin{bmatrix} g\\f\end{bmatrix}$. This both saves space, and makes some of the algebraic manipulations easier to follow.
\end{notation}
A \emph{ filtered $A_\infty$ algebra}  $(A^\bullet, m^k)$ is a free graded $\Lambda$-vector space $A^\bullet$ with $\Lambda$-linear cohomologically graded higher products for each $k \geq 0$
\[
	m^k:(A^\bullet)^{\tensor k}\to (A^{\bullet+2-k})
\]
with extra data and satisfying the filtered $A_\infty$ algebra axioms given in \cite[Definition 3.2.20]{fukaya2010lagrangian}.
The key relation is the quadratic $A_\infty$ relation, which we will denote by:
\[
	\sum_{k_1+k'+k_2=k} (-1)^{\clubsuit(\underline x, k_1)} \begin{bmatrix}(\id^{\tensor k_1}\tensor m^{k'}\tensor id^{\tensor k_2})\\ (m^{k_1+1+k_2})\end{bmatrix}(x_1, \cdots ,x_k)
\]
where the sign is determined by the cohomological grading of the domain
\[\clubsuit(\underline x,k_1):= k_1+\sum_{j=1}^{k_1} \deg(x_j).\]
\begin{notation}
	We will suppress $\underline x =(x_1, \cdots, x_k)$ in future computations, but continue to use the notation $\clubsuit(\underline x, i)$ for the purposes of determining signs.
\end{notation}
If $m^0=0$, then $(A^\bullet, m^1)$ is a chain complex and we say that $A^\bullet$ is \emph{uncurved} or \emph{tautologically unobstructed}.
There is a filtration map $\fil: A^\bullet\to \RR$ which is defined by the smallest $\lambda$ such that $v\in F^\lambda A$. Up to a change of sign convention, this makes $A$ a non-Archimedean normed vector space in the sense of \cite[Definition 1.2.3]{usher2016persistent}.
The filtration zero portion of products $\underline{m}^k: A^{\tensor k}\to A$ are a tautologically unobstructed $A_\infty$ algebra. 

We from now on suppress the cohomological index and, when the product structure is clear, we will simply notate an $A_\infty$ algebra by $A$. 
An \emph{ideal} of $A$ is a subspace $I\subset A$ so that for every $b\in I$ and $a_1, \ldots, a_{k-1}\in A$, 
	\[
		m^k(a_1\tensor\cdots \tensor  a_j\tensor b \tensor a_{j+1}\tensor \cdots \tensor a_{k-1})\in I.
	\]
Note that we \emph{do not require} $m^0\in I$. 
The set $A_{> 0}$ of elements with positive filtration is an $A_\infty$ ideal and $A/A_{\geq 0}$ is a tautologically unobstructed $A_\infty$ algebra.

We take a variation of \cite[Definition 5.2.1]{fukaya2010lagrangian} and adapt it to \cite[Definition 3.2.7]{fukaya2010lagrangian}
\begin{definition}
	Let $(A, m_A^k)$ and $(B, m_B^k)$ be $A_\infty$ algebras. A \emph{weakly-filtered $A_\infty$ homomorphism} from $A$ to $B$ is a sequence of	 graded maps
	\[
		f^k:A^{\tensor k}\to B
	\]
	satisfying the following conditions:
	\begin{itemize}
		\item
		      \emph{Weakly Filtered:} The maps nearly preserve energy
		      \[
			      f^k(F^{\lambda_1}A, \cdots , F^{\lambda_k}A)\subset F^{-c\cdot k+\sum_{i=1}^k \lambda_i}B
		      \]
		      for some fixed constant $c$ called the \emph{energy loss of $f$} with $c<\fil(m^0_A)$. 
		\item
			\emph{Quadratic $A_\infty$ relations:} The $f^k, m^k_A$ and $m^k_B$ mutually satisfy the quadratic filtered $A_\infty$ homomorphism relations
			\[\sum_{k_1+k'+k_2=k}(-1)^{\clubsuit(\underline x, k_1)}\begin{bmatrix}
				\id^{\tensor k_1}\tensor m^{k'}_A \tensor \id^{k_2}\\
				f^{k_1+1+k_2}\end{bmatrix}
				=\sum_{i_1+\cdots i_j=k}\ \begin{bmatrix}f^{i_1}\tensor \cdots \tensor f^{i_j}\\m^j_{B} \end{bmatrix}.
\]
	\end{itemize}
	\label{def:ainfinityhomomorphism}
\end{definition}
The composition of weakly-filtered $A_\infty$ homomorphisms is again a weakly filtered $A_\infty$ homomorphism.

We follow \cite[Definition 3.6.9]{fukaya2010lagrangian}  and write  $m^k_a: A^{\tensor k}\to A$ for the deformation of the $A_\infty$ structure at a element $a\in A$ for positive valuation. We say that $b$ is a bounding cochain if $(A^\bullet, m_b)$ is tautologically unobstructed, and write 
\[
	\mathcal MC_{> c}(A) := \{b \in A \;:\; m^0_b=0, \fil(b)>c\}.
\]
for the space of Maurer-Cartan elements of filtration at least $c$ \cite[Definition 3.6.4]{fukaya2010lagrangian}. 
Given $f: A\to B$  a weakly filtered $A_\infty$ morphism of energy loss $c$, there exists a pushforward map $f_*: \mathcal MC_{> c}(A) \to \mathcal MC_{>0}(B)$. 

It will be convenient for us to use the following notations:
\begin{claim}
	Let $f: A\to B$ be a filtered $A_\infty$ algebra morphism.
	Then there exists an $A_\infty$ homomorphism
	\[ f_{\flat}: (A, m_A)\to (B, (m_B)_{f_*(0)})\]
	where $f_{\flat}$ is defined
	by
	\[
		f^{k}_\flat = \left\{ \begin{array}{cc} f^k& \text{for $k>0$}\\ 0 &\text{if $k=0$}\end{array} \right.
	\]
	\label{claim:fflat}
\end{claim}

\begin{claim}
	Let $f: A\to B$ be a filtered $A_\infty$ algebra morphism.
	Let $a\in A$ be a deforming element.
	Then there is a map
	\[
		f_a: (A, (m^k_A)_a)\to (B, m^k_B).
	\]
\end{claim}

\subsection{Curved Homotopy Transfer Theorem}
\begin{theorem}[Weakly Filtered Whitehead Theorem]
	Let  $A$ and $B$ be filtered $A_\infty$ algebras. 
	Suppose that we have maps :
	\[\setlength\mathsurround{0pt}\begin{tikzcd}[ampersand replacement=\&]
		A\arrow[shift left =.75ex]{r}{\alpha} \&B \arrow[shift left =.75ex]{l}{\beta^k} \arrow[out=30,in=330,loop,swap, "h"]
		\end{tikzcd}\setlength\mathsurround{.8pt} \]
	so that $\beta^k$ is an $A_\infty$ homomorphism, with $\beta^k=0$ for all $k\neq 1$, and the following conditions hold

			\begin{align}
				\beta \circ m^1_B-m^1_A\circ  \beta=& \beta \circ m^1_B \circ m^1_B \circ h 				\label{eq:pichainmap}\\
				\alpha\circ m^1_A-m^1_B\circ     \alpha=&h\circ m^1_B\circ m^1_B\circ \alpha 				\label{eq:ichainmap}\\
				h \circ m^1_B + m^1_B \circ h=& \alpha\circ \beta - \id- h \circ m^1_B\circ m^1_B\circ h	\label{eq:homotopyatcrit}
			\end{align}
			\begin{align*}
									 \beta\circ h =0 &&
									h \circ\alpha= 0 \\
									\beta \circ\alpha= \id_A&&
									\fil(h\circ m^0_B)>0
			\end{align*}
	Then we can extend
	\begin{itemize}
		 \item  $\alpha$ to a weakly filtered $A_\infty$ morphism $\alpha^k: A^{\tensor k}\to B.$
		 \item  $h$ to a weakly filtered $A_\infty$ homotopy $h^k: B^{\tensor k}\to B$.
	\end{itemize}
	\label{thm:curvedhtt}
\end{theorem}
	Note that this relates to the homotopy transfer theorem of \cite{kadeishvili1980homology} when $m^0_B=0$, as \cref{eq:pichainmap,eq:ichainmap,eq:homotopyatcrit} become the chain map and homotopy relations. 
	The additional relations are analogues of the strong deformation retract conditions.
	In the filtered setting they tell us that $\beta, \alpha$ intertwine the curvatures of $m^1_B$ and $m^1_A$:
	\[\beta \circ m^1_B\circ m^1_B\circ \alpha = m^1_A \circ m^1_A.\]

	\Cref{thm:curvedhtt} is also related to \cite[Theorem 4.2.45]{fukaya2010lagrangian}, which proves that if $\alpha, \beta$ are filtered $A_\infty$ homomorphisms, and  $h$ is filtered, then $h$ extends to an $A_\infty$ homotopy between $\alpha, \beta$ whenever \cref{eq:pichainmap,eq:ichainmap,eq:homotopyatcrit} hold at zero valuation.
	We now examine the possibility of adapting \cite[Theorem 4.2.45]{fukaya2010lagrangian} with the weaker condition that $h$ is weakly-filtered.

	The proof of \cite[Theorem 4.2.45]{fukaya2010lagrangian} first proves the theorem for unfiltered $A_\infty$ algebras. The proof for the unfiltered $A_\infty$ algebra proceeds by constructing the maps $m^k_\alpha: A^{\tensor k} \to A$ inductively over the index of $k$.
	This proof necessarily uses the fact that $f^0$ and $m^0_A$ are trivial, as otherwise higher terms $f^{k+i}$ show up in the $k$-input quadratic $A_\infty$ relations.

	The filtered case is proven by induction on the filtration, with the unfiltered setting serving as a base case.
	Two difficulties occur when extending to weakly filtered homotopies.
	In \cite[Theorem 4.2.45]{fukaya2010lagrangian} the definition of homotopy is given in terms of filtered $A_\infty$ homomorphism $h: A\to \mathfrak A$, where $\mathfrak A$ is a model of $[0, 1]\times A$.
	One can generalize these to weakly-filtered $A_\infty$ homotopies by considering weakly-filtered homomorphisms $h: A\to \mathfrak A$ and weakly filtered $A_\infty$ structures. Difficulties with these approaches are commented on in the discussion surrounding \cite[Theorem 5.2.35]{fukaya2010lagrangian}, where a similar statement is considered for weakly filtered maps of $A_\infty$ bimodules. 

	Additionally, when $h$ is only weakly filtered, we can no longer use the proof strategy for \cite[Theorem 4.2.45]{fukaya2010lagrangian} as $h$ no longer provides a weak homotopy between $\alpha \circ \beta$ and $\id_B$ on the tautologically unobstructed $A_\infty$ algebra $B/B_{>0}$. A possible approach to the weak-homotopy case is given in \cite[Section 7.2.6]{fukaya2010lagrangian}, which examines $G$-filtered $A_\infty$ algebras.
	 A $G$-filtered $A_\infty$ structure \cite[Definition 7.2.67]{fukaya2010lagrangian}, is a collection of $m^{k, \beta}: A^{\tensor k}\to A$ indexed by $k\in \NN$ and $\beta$ in a monoid $G$. The monoid $G$ should have a unique invertible element. We also require that every element of $G$ possess only finite factorization, that is for all $\beta\in G$,
	\[\|\beta\|:= \sup\{ k \;:\; \beta = \beta_1+\cdots \beta_k = \beta, \beta_i \neq e.\}<\infty\]
	 The $A_\infty$ products are given by
	\[m^k= \sum_{\beta\in G} T^{\lambda(\beta)} m^{k, \beta}.\]
	satisfying the $A_\infty$ which respects the structure of $G$. 
	Roughly, the proof of \cite[Theorem 7.2.72]{fukaya2010lagrangian} inducts on $k$ and $\|\beta\|$, and uses this finite factorization property. 
	In spirit, our approach is much like this one.

	Instead of strengthening the condition on the filtration, we choose instead to strengthen the condition on the hypothesis of the homotopy equivalence (\cref{eq:pichainmap,eq:ichainmap,eq:homotopyatcrit}).
	The upshot is that we can reexpress $m^K_A$ in terms of the maps $\beta, \alpha, m^K_B$ and $h$, allowing us to decompose the map based on the number of $m^0_B$ terms which occur in the expansion. This is used as a replacement for the monoid structure, and we induct on this quantity, rather than the filtration, in our proof.
We now give the construction of the maps in \cref{thm:curvedhtt}. 

\subsection*{Proof of \cref{thm:curvedhtt}}
Our proof is a modification of the inductive proof given in Markl's Christmas carp paper \cite{markl2006transferring}, using the sign conventions from \cite{fukaya2010lagrangian}. 
We include the steps in the induction proof which differ from the proof in the unfiltered setting. Notably, our construction inducts over two variables, which correspond degree filtration and the number of $B$-curvature terms that appear in the $A_\infty$ relations. 
Outside of these changes, the only difference between the weakly-filtered and unfiltered proof is the additional care that must be spent with the ranges of indices due to the presence of curvature terms.
Place a partial order on pairs $(n, k)$ so that $(k', n')\leq (k, n)$ if $n'\leq n$ and $k'\leq k$. Define the maps:
 \begin{align*}
	\alpha^{1, 0}:= \alpha && \alpha^{0,1}=  h\circ m^0_B && m_A^{1,0}:=m^1_A && m_A^{0,1}= \beta\circ m^0_B\\
	\chi^{0,1}=m^0_B&& \chi^{0,0}=0 && \chi^{1,0}=0
\end{align*}
Suppose that for all $(k', n')\leq (k, n)$, we've defined the maps $\alpha^{k', n'}$ and $\chi^{k',n'}$. 
Then inductively define the kernel
\[
	\rho^{k, n}=-\sum_{r\geq 2} \sum_{\substack{k_1+ \ldots+ k_r=k\\n_1+ \cdots + n_r=n\\ (0,0)\leq (n_i, k_i)}}\begin{bmatrix}\alpha^{k_1, n_1}\tensor \cdots \tensor \alpha^{k_r, n_r}\\  m^r_B \end{bmatrix}
\]
 error terms
\[
	\chi^{k, n}= \sum_{r\geq 2} \sum_{\substack{k_1+ \ldots+ k_r=k\\n_1+ \cdots + n_r=n\\ (0,0)\leq (n_i, k_i)}}(-1)^{\clubsuit(\underline x, \sum_{j=1}^{i-1} k_j)}\begin{bmatrix}\alpha^{k_1, n_1}\tensor\cdots \chi^{k_i, n_i} \tensor \cdots \tensor \alpha^{k_r, n_r}\\  h \circ m^r_B \end{bmatrix}
	\]
and maps
	\begin{align*}	
	\alpha^{k, n}=- h\circ \rho^{k,n}
	&&m_A^{k, n}= -\beta\circ \rho^{k, n}.
\end{align*}
The first two  error terms are 
\begin{align}
	\chi^{1,1}= h\circ m^B(m^0_B\tensor \alpha^{1, 0} +(-1)^\clubsuit \alpha^{1, 0} \tensor m^0)= h\circ m^1_B\circ m^1_B\circ \alpha^{1,0}
	\label{eq:firsterrorterm}\\
	\chi^{0,2}= h\circ m^2(\alpha^{0,1}\tensor \chi^{0,1}+(-1)^\clubsuit\alpha^{0,1}\tensor \chi^{0,1})= h\circ m^1_B\circ m^1_B\circ h \circ m^0_B\label{eq:seconderrorterm}
\end{align}

The filtration of these maps are increasing in $n$; as we can prove inductively that
\[\fil(\rho^{k, n})(a_1, \ldots, a_k) \geq n\cdot \fil(m^0)-(n+k)\fil(h)+\sum_{i=1}^k \fil(a_i),\]
and by the condition that $\fil(h\circ m^0_B)>0$ we see that $\sum_{n=0}^\infty \rho^{n, k}$ converges over Novikov coefficients.
Most of the $m^{k, n}_A$ terms end up being trivial, as by applying the $A_\infty$ homomorphism relations for $\beta$ we obtain
\begin{align*}	m^{k, n}= \begin{bmatrix}h\circ \rho^{k_1, n_1}\tensor \cdots \tensor h\circ \rho^{k_r, n_r}\\  \beta\circ m^r_B \end{bmatrix}
	=\begin{bmatrix}h\circ \rho^{k_1, n_1}\tensor \cdots \tensor h\circ \rho^{k_r, n_r}\\  m^r_A\circ (\beta^{\tensor r}) \end{bmatrix}
\end{align*}
which vanishes whenever the composition $\beta\circ h$ appears.
So, this term is zero unless $(k, n)=(k, 0)$.
We can therefore write 
\[m^k_A= m^k_A\circ (\beta\circ \alpha)^{\tensor k}= \beta\circ  m^{k, 0}_B\circ (\alpha)^{\tensor k }= m^{k, 0}_A= \sum_{n} m^{k,n}_A.\]	

We will prove by induction on the product order for $(k, n)$ that the following relation $I(k,n)$ holds:
\begin{equation}\sum_{\substack{k_1+\cdots k_r=k\\ n_1+\ldots n_r=n}}\begin{bmatrix} \alpha^{k_1, n_1}\tensor \cdots \tensor \alpha^{k_r, n_r}\\ m^r_B\end{bmatrix} = \sum_{\substack{j_1+j'+j_2=k\\ n_1+n'=n}}(-1)^{\clubsuit(\underline x, j_1)} \begin{bmatrix} \id^{\tensor j_1}\tensor m_A^{j', n'}\tensor \id^{\tensor j_2}\\ \alpha^{j_1+j_2+1, n_1}\end{bmatrix}- \chi^{k, n}+ \chi^{k, n+1}.
	\label{eq:inductiverelation}
\end{equation}
We note that summing all of the $I(k,n)$ relations for fixed $k$ yields the $A_\infty$ homomorphism relation for $\alpha^k:=\sum_{n} \alpha^{k, n}$.

\begin{proof}
The base cases for our induction are when $(k, n)\in \{(0,0), (0,1), (1,0)\}.$
\begin{itemize}
	\item The relation $I(0,0)$  states 
	\[m^0_B = \chi^{0,1}\]
	\item A computation using the relation \cref{eq:firsterrorterm} shows that  $I(1,0)$ follows from the relation \cref{eq:ichainmap}
	\item A computation using \cref{eq:seconderrorterm} shows that $I(0,1)$ follows from the relation \cref{eq:homotopyatcrit}.
\end{itemize}
Now assume that we've shown that $I(k', n')$ holds for all $(k', n')<(k, n)$. As $(k, n)\neq (0,0), (1, 0), (0,1)$, we know that $r\geq 2$ in the left hand side of $I(k,n)$. We can therefore reexpress the LHS of $I(k,n)$ as
\begin{align*}
	&\sum_{\substack{k_1+\cdots k_r=k\\ n_1+\ldots n_r=n}}\begin{bmatrix} \alpha^{k_1, n_1}\tensor \cdots \tensor \alpha^{k_r, n_r}\\ m^r_B\end{bmatrix}=-\rho^{k,n}+m^1_B\circ \alpha^{k, n} =\mathcolorbox{red!20}{-\rho^{k,n}}-m^1_B\circ h \circ \rho^{k,n}
	\intertext{By applying $m^1_B\circ h\circ \rho^{k, n}= -\alpha\circ  m_A^{k, n}-\mathcolorbox{red!20}{\rho^{k,n}}- h \circ m^1_B\circ m^1_B\circ \alpha^{k,n}-h\circ m^1_B\circ \rho^{k, n}$ from \cref{eq:homotopyatcrit} and cancelling the highlighted terms}
	=& \alpha \circ m_A^{k, n}+h \circ m^1_B\circ m^1_B\circ \alpha^{k,n} +	\sum_{r\geq 2} \sum_{\substack{k_1+ \ldots+ k_r=k\\n_1+ \cdots + n_r=n\\ (0,0)\leq (n_i, k_i)}}\begin{bmatrix}\alpha^{k_1, n_1}\tensor \cdots \tensor \alpha^{k_r, n_r}\\ h\circ m^1_B\circ  m^r_B \end{bmatrix}\\
	=& \alpha \circ m_A^{k, n} +\sum_{r\geq 1} \sum_{\substack{k_1+ \ldots+ k_r=k\\n_1+ \cdots + n_r=n\\ (0,0)\leq (n_i, k_i)}}\begin{bmatrix}\alpha^{k_1, n_1}\tensor \cdots \tensor \alpha^{k_r, n_r}\\ h\circ m^1_B\circ  m^r_B \end{bmatrix}
	\intertext{Applying the $A_\infty$ relations for $m^1\circ m^r_B$}
	=& \alpha \circ m_A^{k, n}+\sum_{\substack{r_1+r'+r_2= r\geq 2\\ r_1+r_2>0}}\sum_{\substack{k_1+ \ldots+ k_r=k\\n_1+ \cdots + n_r=n}}(-1)^{\heartsuit(\underline x, \underline k, r_1)}\begin{bmatrix} \alpha^{k_1, n_1}\tensor \cdots \tensor \alpha^{k_r, n_r} \\\id^{\tensor r_1}\tensor m^{r'}_B\tensor \id^{\tensor r_2}\\h\circ m^{r_1+r_2+1}_B  \end{bmatrix}\\
	\intertext{ Here, $\heartsuit(\underline x, \underline k, r_1):=r_1+\sum_{i=1}^{r_1}\left(1-k_i+\sum_{j=1}^{k_i} \deg(x_j)\right)$.  We regroup the terms which lie above $m^{r'}_B$ in the composition}
	=&\alpha^{1,0}\circ  m_A^{k, n}\\&+\sum_{\substack{r_1+r'+r_2= r\geq 1\\ r_1+r_2\geq0\\ j_1+j'+j_2=k\\m_1+m'+m_2=n}} \sum_{\substack{
		k_1+\cdots+ k_{r_1}=j_1\\
		k_{r_1+1}+\cdots +k_{r_2+r'}=j'\\
		k_{r-r_1}+\cdots +k_r=j_2\\
		n_1+\cdots+ n_{r_1}=m_1\\
		n_{r_1+1}+\cdots +n_{r_2+r'}=m'\\
		n_{r-r_1}+\cdots +n_r=m_2
		}}(-1)^{\heartsuit(\underline x, \underline k, r_1)}\begin{bmatrix}\id^{\tensor j_1}\tensor \mathcolorbox{orange!20}{
			\begin{bmatrix} {\displaystyle \bigotimes_{i=r_1+1}^{r_1+r'} (\alpha^{k_i, n_i}) }\\ m^{r'}_B  \end{bmatrix} 
		}\tensor \id^{\tensor j_2}\\
			{\displaystyle \bigotimes_{i=1}^{r_1} (\alpha^{k_i, n_i})\tensor \id \tensor\bigotimes_{i=r-r_2}^{r} (\alpha^{k_i, n_i})}\\ h \circ m^{r_1+r_2+1}_B \end{bmatrix}	
	\intertext{When taken modulo 2,  $\heartsuit(\underline x, \underline k, r_1)\equiv \sum_{i=1}^{r_1}k_i + \sum_{j=1}^{j_1}\deg(x_i)\equiv \clubsuit(\underline x, j_1)$. When grouped over the index $(j', m')$, the orange block is the left hand side of the relation $I(j', m')$. If $m'=n$, then $j_1, j_2\neq 0$ and therefore $j'<k$. Therefore, $(j', m')< (j, m)$.
	Let $\underline x_{j_1, j_2}$ be the tuple $(x_{j_1+1}, \ldots, x_{j_1+j'})$.
	We can apply the inductive hypothesis: }
		=&\alpha^{1,0}\circ  m_A^{k, n}\\
		&+\sum_{\substack{r_1+r'+r_2= {r\geq 1}\\ r_1+r_2>0\\ j_1+j'+j_2=k\\m_1+m'+m_2=n}}(-1)^{ \clubsuit(\underline x, j_1)} \sum_{\substack{
			k_1+\cdots+ k_{r_1}=j_1\\
			k_{r-r_1}+\cdots +k_r=j_2\\
			n_1+\cdots+ n_{r_1}=m_1\\
			n_{r-r_1}+\cdots +n_r=m_2\\
			l_1+l'+l_2=j'\\
			o_1+o_2=m'
			}} (-1)^{\clubsuit(\underline x_{j_1, j_2}, l_1)}\begin{bmatrix}\id^{\tensor j_1}\tensor \mathcolorbox{orange!20}{\begin{bmatrix} \id^{\tensor l_1}\tensor m_A^{l', o_1}\tensor \id^{\tensor l_2}\\
				\alpha^{l_1+l_2+1, o_2} \end{bmatrix} }\tensor \id^{\tensor j_2}\\
				{\displaystyle \bigotimes_{i=1}^{r_1} (\alpha^{k_i, n_i})\tensor \id \tensor\bigotimes_{i=r-r_2}^{r} (\alpha^{k_i, n_i})}\\ h \circ m^{r_1+r_2+1}_B \end{bmatrix}\\
		&+\sum_{\substack{r_1+r'+r_2= {r\geq 1}\\ r_1+r_2>0\\ j_1+j'+j_2=k}}(-1)^{ \clubsuit(\underline x, j_1)}\sum_{\substack{
			k_1+\cdots+ k_{r_1}=j_1\\
			k_{r_1+1}+\cdots +k_{r_2+r'}=j'\\
			k_{r-r_1}+\cdots +k_r=j_2\\
			\\n_1+ \cdots + n_r=n}}\begin{bmatrix}\id^{\tensor j_1}\tensor \mathcolorbox{orange!20}{(-\chi^{r', n'}+\chi^{r', n'+1} )}\tensor \id^{\tensor j_2}\\
				{\displaystyle \bigotimes_{i=1}^{r_1} (\alpha^{k_i, n_i})\tensor \id \tensor\bigotimes_{i=r-r_2}^{r} (\alpha^{k_i, n_i})}\\ h \circ m^{r_1+r_2+1}_B \end{bmatrix}
\end{align*}
The remainder of the proof proceeds as in the tautologically unobstructed setting by applying the recursive definition of $\chi^{n,k}$ and $\rho^{n,k}$, and using the relation  
$(-1)^{ \clubsuit(\underline x, j_1)} \cdot (-1)^{\clubsuit(\underline x_{j_1, j_2}, l_1)} = (-1)^{\clubsuit(\underline x, j+l_1)}$.

We now define the homotopy, again following \cite{markl2006transferring}. Inductively define the maps
\[\phi^{k}:= \sum_{\substack{
	k_1+k'+k_2= k\\
	j_1+\cdots j_r=k_1\\
	r+k_2\geq 1 , k'\geq 1
}}(-1)^{\clubsuit(\underline x,k_1 )}\begin{bmatrix} 
	 \beta^{\tensor k_1}\tensor h^{k'} \tensor \id^{\tensor k_2}\\ 
	 \alpha^{ j_1}\tensor \cdots \tensor \alpha^{j_r} \tensor \id^{\tensor k_2+1}\\m^{r+1+k_2} \end{bmatrix}\]
where $h^{k}:=h\circ \phi^{k}$.
The $A_\infty$ homotopy relations
\begin{align*}
	\sum_{\substack{
	k_1+k'+k_2= k\\
	j_1+\cdots j_r=k_1\\
	r+k_2\geq 0, k'\geq 1
}}&(-1)^{\clubsuit(\underline x,k_1 )}\begin{bmatrix} 
	 (\alpha\circ \beta)^{j_1}\tensor \cdots \tensor (\alpha\circ \beta)^{j_r}\tensor h^{k'}\tensor \id^{\tensor k_2}\\ m^{r+1+k_2} \end{bmatrix}\\
	 =& (\alpha\circ \beta)^k - \id^k -
	 \sum_{\substack{k_1+k'+k_2=k, k_1+k_2\geq 0}}(-1)^{\clubsuit(\underline x,k_1 )}\begin{bmatrix} \id^{\tensor k_1} \tensor m^{k'}_B\tensor \id^{\tensor k_2}\\ h^{k_1+k_2+1}\end{bmatrix}
\end{align*}
are proved inductively. 
In the above expression, $\id^k=\id$ if $k=1$, and is zero otherwise. On the right hand side $k'$ may be equal to $0$.
\begin{align*}
	&\sum_{\substack{
		k_1+k'+k_2= k\\
		j_1+\cdots j_r=k_1\\
		r+k_2\geq 0, k'\geq 1
	}}(-1)^{\clubsuit(\underline x,k_1 )}\begin{bmatrix} 
		 \beta^{\tensor k_1}\tensor h^{k'} \tensor \id^{\tensor k_2}\\ 
		 \alpha^{ j_1}\tensor \cdots \tensor \alpha^{j_r} \tensor \id^{\tensor k_2+1}\\m^{r+1+k_2} \end{bmatrix}
	=m^1_B \circ h \circ \phi^{k}+\phi^{k}\\
	=& \alpha \circ \beta \circ \phi^{k} - h\circ m^1_B\circ \phi^{k} - h\circ m^1_B\circ m^1_B\circ h \circ \phi^{k}\\
=& \alpha \circ \beta \circ \phi^{k} + \sum_{\substack{
		k_1+k'+k_2= k\\
		j_1+\cdots j_r=k_1\\
		r+k_2\geq 0, k'\geq 1
	}}(-1)^{\clubsuit(\underline x,k_1 )}\begin{bmatrix} 
		\beta^{\tensor k_1}\tensor h^{k'} \tensor \id^{\tensor k_2}\\ 
		\alpha^{ j_1}\tensor \cdots \tensor \alpha^{j_r} \tensor \id^{\tensor k_2+1}\\
		h\circ m^1_B\circ m^{r+1+k_2} \end{bmatrix}\\
=& \alpha \circ \beta \circ \phi^{k} + \sum_{\substack{
		k_1+k'+k_2= k\\
		j_1+\cdots j_r=k_1\\
		i_1+i'+i_2=r+k_2+1\geq 1\\
		 k'\geq 1
	}}(-1)^{\clubsuit(\underline x,k_1 )+\clubsuit(\underline y, j_1)}\begin{bmatrix} 
		\beta^{\tensor k_1}\tensor h^{k'} \tensor \id^{\tensor k_2}\\ 
		\alpha^{ j_1}\tensor \cdots \tensor \alpha^{j_r} \tensor \id^{\tensor k_2+1}\\
		\id^{\tensor i_1}\tensor m^{i'}_B\tensor \id^{\tensor i_2}\\
		h\circ  m^{i_1+i_2+1} \end{bmatrix}\\
\end{align*}
 where $\underline y= y_1\tensor \cdots \tensor y_{r+k_2+1}$ is the partial composition $
( \alpha^{ j_1}\tensor \cdots \tensor \alpha^{j_r} \tensor \id^{\tensor k_2+1})\circ  (\beta^{\tensor k_1}\tensor h^{k'} \tensor \id^{\tensor k_2})\circ (\underline x)$.
Collecting into terms for which $j_1+j'<r, j_1<r<j_1+j'$ and $r<j_1+1$ allows us to either apply the $A_\infty$ homomorphism for $\alpha^k$ or the induction hypothesis, to push the $m^{j'}_B$ term through to the top of the composition. The remainder of the proof is standard.
\end{proof}
\begin{remark}
	The relations \cref{eq:pichainmap,eq:ichainmap,eq:homotopyatcrit} require some motivation. The $A_\infty$ relations can be proved non-inductively by unpacking the definition of $\rho$ into the sum over compositions defined over stable trees, as in \cite{kontsevich2001homological}. 
	The $A_\infty$ and homotopy relations correspond to contractions and expansions of these trees at internal vertices and edges; the argument in some crucial way uses the fact that the contractions of stable trees at internal edges are again stable trees.
	
	To generalize to the filtered setting, one must consider stable trees with $k$-internal and $n$-external leaves, where the internal leaves are labelled with curvature terms coming from $m^0$.
	The $A_\infty$ and homotopy relations again correspond to contractions and expansions. However, the contraction of a stable tree at an internal leaf is not necessarily stable. The terms added to \cref{eq:pichainmap,eq:ichainmap,eq:homotopyatcrit} cancel out these contributions coming from non-stable trees.

	We give the non-inductive proof of \cref{thm:curvedhtt} in the preprint version of this paper.
\end{remark}
\subsection{Morphisms are mapping cocylinders}
In the category of chain complexes, there is a dictionary between morphisms and mapping cocylinders.
We now extend this dictionary to filtered $A_\infty$ algebras.
\begin{definition}
	Let $A^+$ and $A^-$ be two filtered $A_\infty$ algebras.
	A \emph{cocylinder} from  $A^+$ to  $A^-$ is a filtered $A_\infty$ algebra $B$  which:
	\begin{itemize}
		\item as a vector space is isomorphic to $A^-\oplus A^0\oplus A^+$ ;
	\item has differential of the form:
		 \[
			\begin{pmatrix}
				m^{-;-} & 0         & 0            \\
				m^{-;0}  & m^{0;0} & m^{+;0}\\
				0         & 0         & m^{+;+}
			\end{pmatrix}.
	\]
			  where $m^{+;0}$ is an isomorphism with inverse satisfying $\fil((m^{+;0})^{-1}\circ m^0)>0$;
		\item
		      has projections of chain complexes
		      \[
			      \setlength\mathsurround{0pt}\begin{tikzcd}
				      \; & B \arrow{dl}{\beta^{-}} \arrow{dr}{\beta^+} \\
				      A^-& & A^+
			      \end{tikzcd}\setlength\mathsurround{.8pt}
		      \]
		      which can be extended to  $A_\infty$ homomorphisms $\beta^k_\pm$, with $\beta^k_\pm=0$ for all $k\neq 1$.
	\end{itemize}
	We denote such a mapping cocylinder
	\[
		A^-\leftrightarrow B \to A^+.
	\]
	\label{def:mappingcylinder}
\end{definition}

The cylinders from $A^-$ to $A^+$ are in correspondence with morphisms $f: A^-\to A^+$.

\begin{theorem}[Cylinders are mapping cocylinders]
	Let $A^-$ and $A^+$ be two filtered $A_\infty$ algebras.
	\begin{enumerate}
		\item
		      To every cocylinder $A^-\leftrightarrow B \to A^+$, we can associate a morphism $\Theta_B:A^-\to A^+ $. \label{item:cyltomorph}
		\item
		      To every morphism $f: A^-\to A^+$, we can associate a cocylinder
		      \[
			      A^-\leftrightarrow B_f\to A^+.
			  \]
			  \label{item:morphtocyl}
		\item
			  These constructions are compatible in the sense that $\Theta_{B_f}= f$.
			  \label{item:compatibility}
	\end{enumerate}
	\label{thm:cylfrommap}
\end{theorem}
\begin{proof}[Outline of Proof]
	The construction of morphism associated to a mapping cocylinder proceeds by showing that the projection $\beta_-$, and map $\alpha: A^-\to B$ and homotopy $h: B\to B$ given by
	\begin{align*}
		\alpha(x)=(x, 0, (m^{+;0})^{-1}\circ m^{-;0}(x))&& h(x)=(0,0,(m^{+;0})^{-1}(x))
	\end{align*}
	satisfy the conditions of \cref{thm:curvedhtt}, giving us an $A_\infty$ extension of $\alpha$. The morphism associated to the cocylinder is the ``pullback-pushforward'' composition $\Theta_B:=\beta_+\circ \alpha$. 

	The construction of a mapping cocylinder associated to a morphism follows from the following more general construction.
	Let $C$ be any filtered $A_\infty$ algebra, and let $f: A^-\to C$ and $g: A^+\to C$ be filtered $A_\infty$ homomorphisms.
	Then $C$ can be equipped with the structure of a filtered $(A^+, A^-)$ bimodule, whose bimodule product 
	\[m^{k_1|1|k_1}_{A^-|C|A^+}: (A^-)^{\tensor k_1}\tensor C\tensor (A^+)^{\tensor k_2}\to C.\]
	is given by (up to sign) 
	\begin{align*}
	    m^{k_1|1|k_2}_{A^-|C|A^+}=\sum_{\substack{h_1+\cdots h_{\alpha_1} =k_1\\ i_{1}+\cdots + i_{\alpha_2}=k_2}} m^{\alpha_1+1+\alpha_2}_C\circ (f^{h_1}\tensor \cdots \tensor f^{h_{\alpha_1}}\tensor\id_C \tensor g^{i_1}\tensor \cdots \tensor g^{i_{\alpha_2}}).
    \end{align*}
	The proof of the $A_\infty$ bimodule relations uses the $A_\infty$ homomorphism relations of $f$ and $g$.
	Given a $A_\infty$ bimodule $C$, we define (up to sign) an $A_\infty$ product on $A^-\oplus C[1]\oplus A^+$ by :
	\begin{align*}
		m^k_{A^-\times_C A^+}=&(m^k_{A^+}\circ \pi_{A^-|C|A^+}^{k|0|0})\\
		&\oplus\left( f^k\circ\pi_{A^-|C|A^+}^{k|0|0}+\left(\sum_{k_1+1+k_2=k} m^{k_1|1|k_2}_{A^-|C|A^+}\circ \pi^{k_1|1|k_2}\right)+g^k\circ\pi^{0|0|k}\right)\\
		&\oplus( m^k_{A^+} \circ\pi^{0|0|k}).
	\end{align*}
	for which checking the $A_\infty$ relations becomes a check of the bimodule structure.
	Furthermore, the natural projections $\beta_{\pm}:A^-\oplus C[1]\oplus A^+\to A^{\pm}$ are $A_\infty$ homomorphisms.
	
	To obtain the mapping cocylinder, we let $C=A^+$, and $g=\id$.
\end{proof}
\begin{remark}
	We include the following remark, which we could not find written in the literature. We expect that the $A_\infty$ algebra  $A^-\oplus C[1]\oplus A^+$ is the homotopy fiber product, $A^-\times_C A^+$ in the category of filtered $A_\infty$ algebras. Some specializations are:
	\begin{itemize}
		\item If $A=A^-=A^+=C$, and $f=g=\id_{A}$, then $A\times_A A$ is a model of $A\times [0,1]$, similar to that constructed in \cite[Lemma 4.2.25]{fukaya2010lagrangian}.
		\item For filtered $A_\infty$ algebras, a morphism $f: A\to C$ does not make $C$ a left $A$-module, but rather only an $(A, A)$ bimodule. We note that if there exists a morphism $0\to C$, then we can make $C$ a $(A, 0)$ bimodule, and so it becomes a left $A$ module. This observation is related to the existence of mapping cone and the existence of bounding cochains.
		
		By existence of a pushforward map on bounding cochains, $C$ is unobstructed if and only if there exists a filtered $A_\infty$ homomorphism $0: 0\to C$. We may therefore restate the above observation as $f: A\to C$ has a mapping cone if and only if $C$ is unobstructed.
	\end{itemize}
\end{remark}
 \section{Domain and label dependent perturbation systems}
	\label{subsec:prebrokentrees}
We adopt language from \cite[Section 4.2]{charest2015floer}.
Pick $h: L\to \RR$ a Morse function. 
We denote by $\Gamma$ a combinatorial type of treed disk.
We use $C$ to denote a treed disk, $\mathcal M(\Gamma)$ for  the moduli space of treed disks of fixed combinatorial type,
 $\overline{\mathcal M}(\Gamma)$ to denote its closure, and $\overline {\mathcal U}(\Gamma)$ for the universal treed disk comprising of pairs $(C, x)$, where $x\in C$ is a point. We will denote the forgetful map $C:\overline{\mathcal U}(\Gamma)\to \overline{ \mathcal M}(\Gamma)$. The universal treed disk is covered with two sets, $\overline {\mathcal U}(\Gamma)=\overline{ \mathcal S}(\Gamma )\cup \overline{ \mathcal T}(\Gamma)$, which consists of pairs $(C, x)$ where $x$ lies in either the surface-or-tree component of $C$ \cite[Section 4.2]{charest2015floer}.
There are morphisms $\mathcal M( \Gamma)\to \mathcal M( \Gamma')$ arising from morphisms of underlying combinatorial types \cite[Definition 4.12]{charest2015floer}; they are called collapsing edges/making an edge finite or nonzero; cutting edges; locality; and forgetting forgettable edges. 
For this section, we will focus on the morphisms of making an edge finite and cutting edges.
Let $\ver(\Gamma)$ denote the number of external leaves of $\Gamma$, $N(\Gamma)$ to denote the number of interior marked points of $\Gamma$.
We use $\ell(e)$ for the length of an edge $e$.

A \emph{labeled} combinatorial type $\underline \Gamma$ is a combinatorial type $\Gamma$, along with labels:
\begin{itemize} 
    \item   $\underline x=\{x_0; x_1, \ldots, x_{\ver(\Gamma)-1}\;:\; x_i\in \Crit(h)\}$ a sequence of critical points of a Morse function $h$, indexed by the leaves of $\Gamma$. We reserve $x_0$ for the label of the root of the tree and
    \item At each $k$-broken edge $e$ a sequence of breaking labels $B_e:=\{b_1, \ldots , b_k\}$.
\end{itemize}
The moduli spaces of domains of fixed labeled type cannot be compactified by simply considering labeled types as the resulting compactification will not be Hausdorff. When an edge length goes to infinity the domain will not know what label to assign the breaking in the compactification stratum. Therefore, we help the domains remember how they are supposed to break. 
A \emph{pre-broken} combinatorial type $\underline \Gamma_{P}$ is a labeled combinatorial type $\underline \Gamma$, along with the data of pre-breaking labels $P$, which consists of:
\begin{itemize}
    \item An assignment to each edge $e$ a pre-breaking number $k'(e)$,
    \item A labelling at each $k'$ pre-broken edge or vertex  a sequence of pre-breaking labels $P_e:=\{x_{1}, \ldots, x_{k'}\}$, and 
    \item At each edge $e$, an ordering of $B_e\sqcup P_e$.
\end{itemize}
We will sometimes treat the data of $P$ as a set, so that $P'\subset P$ means that the pre-breaking labels of $\underline \Gamma_{P'}$ are a subset of the labels of $\underline \Gamma_P$, or $\underline \Gamma_{P\setminus P'}$ is the pre-broken combinatorial type where we've removed the $P'$ pre-breakings.

A \emph{pre-broken} disk domain $\underline C_P$ of combinatorial type $\underline \Gamma_P$ is a treed disk with pre-broken marks, which is of a sequence of labeled points $(t_1, x_1), \ldots, (t_{k'(e)}, x_{k'(e)})$ at each edge $e$. Here $0<t_1< \cdots < t_{k'(e)}< \ell(e)$ is a sequence of points on the edge $e$ and the pre-breaking labels are given by $P_e=\{x_i\}_{i=1}^{k'(e)}$. 
There exists a moduli space of pre-broken combinatorial types $\mathcal M(\underline \Gamma_P)$, as well as a universal curve of pre-broken combinatorial types $\mathcal U(\underline \Gamma_P)$.

\subsubsection{Morphisms of pre-broken combinatorial types}
We now look at morphisms of pre-broken combinatorial type. We include the morphisms from \cite[Definition 4.6]{charest2015floer}: cutting edges, collapsing edges, making an edge length finite, making an edge length non-zero, forgetting tails, and making an edge weight finite or non-zero. We add in a new morphism, forgetting a pre-breaking. 
The main changes between morphisms of pre-broken combinatorial types and combinatorial types are in cutting edges, making an edge length finite, and forgetting a pre-breaking. We discuss these changes, as well as the corresponding maps on the moduli space of pre-broken treed disks and their universal curves.

\textbf{Cutting an edge (i):} If $\underline \Gamma_{P}$ has a  breaking $b_i$ occurring at edge $e$ then  $\underline \Gamma'_{P}\to \underline \Gamma_{P}$, the morphism of cutting an edge $e$ at $b_i$
\begin{itemize}
    \item  Replaces $\Gamma$ with a disconnected graph $ \Gamma'$. Let $v_1, v_2$ be the edges of $e$. The graph $\Gamma'$ removes  the edge $e$ and adds in two vertices $w_1, w_2$ connected to edges $e_1=v_1w_1$ and $e_2=v_2w_2$.
    \item The labels of $\underline \Gamma'$ are the same as on $\underline \Gamma$, except at the new vertices $w_1$ and $w_2$ which both obtain the label $b_i$.
    \item Labels for pre-breakings and breakings of $\underline \Gamma_{P}'$ are otherwise inherited from $\underline \Gamma_P$.
\end{itemize}
There are induced homeomorphisms  $\mathcal M(\underline \Gamma_P')\to\mathcal M(\underline \Gamma_P)$  and $\mathcal U(\underline \Gamma_P')\to \mathcal U(\underline \Gamma_P)$ as in \cite{charest2015floer}. 
See arrows labelled (i) in \cref{fig:labelledmodulie} for a visualization of this identification.

\textbf{Making an edge finite (ii):} We now discuss the morphism of making an infinite broken edge finite. We say that $\Gamma$ is obtained from $\Gamma'$ by making an edge finite if the edges of $\Gamma$ and $\Gamma'$ agree away from a single edge $e$, where $\rho(e)=\infty$ and $\rho'(e)\in (0, \infty)$.
In this case, $\mathcal M( \Gamma')$ embeds in $\mathcal M( \Gamma)$ as a codimension one boundary, $\mathcal M(\Gamma')\times [0, \epsilon)$

Given pre-broken combinatorial types $\underline \Gamma_P$ and $\underline \Gamma'_{P'}$, we say that $\underline \Gamma_P$ is obtained from $\underline \Gamma'_{P'}$ by making an edge finite at breaking $b_i$ if 
\begin{itemize}
    \item The underlying unlabeled combinatorial types are related by such a morphism.
    \item The edge $e\in \underline \Gamma'_{P'}$ has one additional pre-broken marking $x_i$ which matches the marking of the breaking $b_i\in B_e$ for $\underline \Gamma_P.$
\end{itemize}

Unfortunately, it is no longer the case that $\mathcal M(\underline \Gamma'_{P'})$ embeds as a codimension one boundary into $\mathcal M(\underline \Gamma_P)$ due to the extra choice in determining where to place the pre-broken point. Therefore, we cannot construct a compactification of $\mathcal M(\underline \Gamma_P)$ by simply including broken configurations. 
However, there remains enough structure for us to later define what a coherent perturbation is in this setting, and thus obtain moduli spaces of pseudoholomorphic treed disks which have the appropriate broken configurations in \cref{subsection:coherencetofloer}.
For $P'\subset P$ a set of markings, let  $\mathcal M^{P'>1/2}(\underline \Gamma_P)$ denote the portion of the moduli space of pre-broken treed disks where the distance between the pre-breakings $t_i\in e$ of $P'$  and the boundary of $e$ is greater than $1/2$. When $P'$ consists of a single marking $(x_i, t_i)$, this space is homeomorphic to 
\[\mathcal M^{\{x_i\}>1/2}(\underline \Gamma_P)=\{(\underline {C}'_{P'}, r, t_s)\;:\; \underline {C}'_{P'}\in \mathcal M(\underline \Gamma'_{P'}) ,r\in [0, \epsilon), t_i\in (1/2, \ell(e)-1/2)\subset e\}.\]
By abuse of notation, we shall denote this set as $\mathcal M(\underline \Gamma')\times [0, \epsilon)\times (1/2, \ell(e)-1/2)$.
There is similarly an identification of the portion of the universal curve living over the region $\mathcal U^{\{x_i\}>1/2}(\underline \Gamma)$ with $\mathcal U(\underline \Gamma')\times[0, \epsilon)\times (1/2, \ell(e)-1/2).$
See the arrows labelled (ii) \cref{fig:labelledmodulie} for a visualization of this identification.

\textbf{Forgetting pre-breakings (iii):}
Let $\underline \Gamma_P$ a pre-broken combinatorial type, and let $P'\subset P$ be a subset of the labels. 
Then we have forgetful maps on the moduli space of treed disks $\mathcal M(\underline \Gamma_P)\to \mathcal M( \underline \Gamma_{P'})$ as well as on the universal curve
\[\res_{P'\subset P}:\mathcal U(\underline \Gamma_P)\to \mathcal U( \underline \Gamma_{P'}).\]

Let $P'\subset P$ denote a subset of the pre-breakings.
We say that a pre-broken treed disk $\underline C_P$ belongs to  $\mathcal M^{P'<1/4}(\underline \Gamma_P)$ if the markings $(t,p)\in P'$ all have a distance less than $1/4$ from the boundary of the edges.
Similarly, let $\mathcal U^{P'<1/4}(\underline \Gamma_P)$ denote the portion of the universal curve whose pre-broken domain belongs to $\mathcal M^{P'<1/4}(\underline \Gamma_P)$.
When we look at a single pre-breaking labelled $x_i$, the neighborhood $\mathcal U^{\{x_i\}<1/4}(\underline \Gamma_P)$ can be identified with 
\[\{(C, x, t_i)\;:\; (C, x)\in \mathcal U(\underline \Gamma_{P\setminus \{t_i\}}), t_i\in [0, \min(1/4, \ell(e))\sqcup (\max(\ell(e)-1/4,0), \ell(e)]\}.\]
See the arrows labeled (iii) in \cref{fig:labelledmodulie} for a diagram of the forgetting pre-breakings identification.
\subsubsection{Boundary regions on pre-broken disks}
\label{subsub:newboundarytypes}
The moduli space of pre-broken curves of fixed combinatorial type is non-compact. In addition to the usual sequences of treed disks which do not converge in the moduli space of treed disks,  the pre-breakings lead to new kinds of sequences that do not have convergent subsequences. There are two new kinds of phenomena we must account for:
\begin{enumerate}
    \item The distance between two pre-breakings goes to zero; or
    \item One of the pre-breakings moves towards the boundary of an edge. \label{subsubitem:newboundarytypes}
\end{enumerate}
We will later force the pre-breakings to remain disjoint from each other, so 
the first kind of non-compactness will not affect the non-compactness of the moduli spaces of pseudoholomorphic treed disks we construct.

The second kind of non-compactness we handle only when constructing the moduli space of pseudoholomorphic treed disks. We will consider perturbation datum which is only domain and label (but not pre-breaking) dependent over $\mathcal M^{P'<1/4}(\underline \Gamma_P)$ for every $P'\subset P$ (and thus independent of the position of pre-breakings in $P'$). 
Such a choice of perturbation datum has the property that whenever we look at a sequence of pseudoholomorphic treed disks with $P'$ pre-breakings moving towards the boundary of the edge, we arrive in a portion of the moduli space of pseudoholomorphic treed disks which is cut out by the $P\setminus P'$ - dependent $J$-holomorphic curve. 
In this chart, it is no longer a problem that the $P'$ pre-breakings move towards the edge. 

\begin{figure}
    \centering
    \scalebox{.6}{\begin{tikzpicture}

\begin{scope}[]
\begin{scope}[]
\draw[thick] (-4,-4) -- (4,-4);
\begin{scope}[shift={(0,-1)}]
\fill[fill=red!20] (-4,-10.5) .. controls (-2,-8) and (1,-6) .. (2,-5.5) .. controls (2,-5) and (2,-5) .. (4,-5);
\fill[fill=red!20] (-4,-5) .. controls (-2,-6) and (1,-6) .. (2,-5.5) .. controls (2.5,-5.5) and (2.5,-5.5) .. (2.5,-5);
\draw (-4,-5) -- (4,-5) -- (-4,-10.5) -- cycle;
\draw[fill=red!20] (2,-5.5) .. controls (1,-6) and (-2.5,-8) .. (-4,-10.5);
\draw[fill=red!20] (2,-5.5) .. controls (1,-6) and (-2,-6) .. (-4,-5);
\draw[fill=blue!20] (-4,-10.5) .. controls (-2.5,-8) and (-3,-6) .. (-4,-5);
\node at (-2,-7) {$\mathcal M(\underline \Gamma_{\{x_a\}})$};
\node at (-0.5,-5.5) {$\mathcal M^{x_a<1/4}(\underline \Gamma_{\{x_a\}})$};
\node[rotate=30] at (-0.5,-7.5) {$\mathcal M^{x_a<1/4}(\underline \Gamma_{\{x_a\}})$};
\node[rotate=270] at (-3.5,-7.5) {$\mathcal M^{x_a>1/2}(\underline \Gamma_{\{x_a\}})$};

\end{scope}
\begin{scope}[yscale=-1, shift={(0,7)}]
\fill[fill=red!20] (-4,-10.5) .. controls (-2,-8) and (1,-6) .. (2,-5.5) .. controls (2,-5) and (2,-5) .. (4,-5);
\fill[fill=red!20] (-4,-5) .. controls (-2,-6) and (1,-6) .. (2,-5.5) .. controls (2.5,-5.5) and (2.5,-5.5) .. (2.5,-5);
\draw (-4,-5) -- (4,-5) -- (-4,-10.5) -- cycle;
\draw[fill=red!20] (2,-5.5) .. controls (1,-6) and (-2.5,-8) .. (-4,-10.5);
\draw[fill=red!20] (2,-5.5) .. controls (1,-6) and (-2,-6) .. (-4,-5);
\draw[fill=blue!20] (-4,-10.5) .. controls (-2.5,-8) and (-3,-6) .. (-4,-5);
\node at (-2,-7) {$\mathcal M(\underline \Gamma_{\{x_b\}})$};
\node at (-0.5,-5.5) {$\mathcal M^{x_b<1/4}(\underline \Gamma_{\{x_b\}})$};
\node[rotate=-30] at (-0.5,-7.5) {$\mathcal M^{x_b<1/4}(\underline \Gamma_{\{x_b\}})$};
\node[rotate=270] at (-3.5,-7.5) {$\mathcal M^{x_b>1/2}(\underline \Gamma_{\{x_b\}})$};

\end{scope}

\node at (3,-4.5) {$\mathcal M(\underline \Gamma)$};
\begin{scope}[shift={(0,0)}]
\node at (-8,-4) {$\mathcal M(\underline \Gamma')$};
\node[circle, scale=.5, fill] at (-8,-3.5) {};

\draw  (-7,-4.5) rectangle (-9,-3);
\end{scope}
\begin{scope}[shift={(0,4.5)}]
\node at (-8,-4) {$\mathcal M(\underline \Gamma^b)$};
\node[circle, scale=.5, fill] at (-8,-3.5) {};

\draw  (-7,-4.5) rectangle (-9,-3);
\end{scope}
\begin{scope}[shift={(0,-4.5)}]
\draw  (-7,-4.5) rectangle (-9,-3);
\node at (-8,-4) {$\mathcal M(\underline \Gamma^a)$};
\node[circle, scale=.5, fill] at (-8,-3.5) {};

\end{scope}

\draw[blue!20, thick] (-4,-4) -- (-2.5,-4);

\begin{scope}[shift={(1,-1.5)}]
\draw  (-12,-6) rectangle (-15,-7.5);
\node[] at (-13.5,-6.5) {$\times$};
\node[circle, scale=.5, fill] at (-14,-6.5) {};
\node[circle, scale=.5, fill] at (-13,-6.5) {};
\node at (-13.5,-7) {$\mathcal M(\underline \Gamma^{1a})\times \mathcal M(\underline\Gamma^{a2})$};

\end{scope}
\begin{scope}[shift={(1,7.5)}]
\node[] at (-13.5,-6.5) {$\times$};
\node[circle, scale=.5, fill] at (-14,-6.5) {};
\node[circle, scale=.5, fill] at (-13,-6.5) {};
\node at (-13.5,-7) {$\mathcal M(\underline\Gamma^{1b})\times \mathcal M(\underline\Gamma^{b2})$};

\draw  (-12,-6) rectangle (-15,-7.5);
\end{scope}
\begin{scope}[shift={(1,3)}]
\draw  (-12,-6) rectangle (-15,-7.5);
\node[] at (-13.5,-6.5) {$\times$};
\node[circle, scale=.5, fill] at (-14,-6.5) {};
\node[circle, scale=.5, fill] at (-13,-6.5) {};
\node at (-13.5,-7) {$\mathcal M(\underline \Gamma^{1})\times \mathcal M(\underline \Gamma^{2})$};

\end{scope}

\draw[ultra thick, dotted] (-4,-7) .. controls (-3.5,-7) and (-2.5,-7.5) .. (-2.5,-7) .. controls (-2.5,-6.5) and (-2.5,-6.5) .. (-2,-6);
\draw[ultra thick, dotted] (-2,-4) .. controls (-1.5,-4) and (0.5,-4) .. (1,-4);
\draw[ultra thick, dotted] (1,-2) .. controls (1.5,-1.5) and (-1,-0.5) .. (-1.5,-0.5) .. controls (-2,-0.5) and (-3.5,-1) .. (-4,-1);

\node[draw, circle, fill=white] at (-4,-7) {1};
\node[draw, circle, fill=white] at (-2.5,-7) {2};
\node[draw, circle, fill=white] at (-2,-6.5) {3};
\node[draw, circle, fill=white] at (0,-4) {4};
\node[draw, circle, fill=white] at (-1.5,-0.5) {5};
\node[draw, circle, fill=white] at (-4,-1) {6};

\end{scope}

\end{scope}

\begin{scope}[shift={(-14,-15.)}]

\draw[fill opacity=0.5,fill=gray!20]   (1.5,1.5) ellipse (0.25 and 0.25);
\draw[fill opacity=0.5,fill=gray!20]   (1.5,-1) ellipse (0.25 and 0.25);
\draw (1.5,1.25) -- (1.5,0.75);
\draw (1.5,-1.25) -- (1.5,-1.5);
\draw (1.5,1.75) -- (1.5,2);

\node[right] at (1.5,2) {$x_1$};
\node[below] at (1.5,-1.5) {$x_2$};
\node at (1.5,2) {$\times$};
\node at (1.5,-1.5) {$\times$};

\draw  (0.5,2.5) rectangle (2.5,-2);
\node at (2,-1.5) {$\underline \Gamma^{a2}$};
\node at (2,1.5) {$\underline \Gamma^{1a}$};
\node[right] at (1.5,-0.25) {$x_a$};
\node[right] at (1.5,0.75) {$x_a$};
\node at (1.5,-0.25) {$\times$};
\node at (1.5,0.75) {$\times$};
\end{scope}

\begin{scope}[shift={(0,-15)}]

\draw[fill opacity=0.5,fill=gray!20]   (1.5,1) ellipse (0.5 and 0.5);
\draw[fill opacity=0.5,fill=gray!20]   (1.5,-0.5) ellipse (0.5 and 0.5);
\draw (1.5,0.5) -- (1.5,0);
\draw (1.5,-1) -- (1.5,-1.5);
\draw (1.5,1.5) -- (1.5,2);

\node[right] at (1.5,2) {$x_1$};
\node[below] at (1.5,-1.5) {$x_2$};
\node at (1.5,2) {$\times$};
\node at (1.5,-1.5) {$\times$};

\draw  (0.5,2.5) rectangle (2.5,-2);
\node at (2,-1.5) {$\underline \Gamma$};
\node at (1.5,0.25) {};
\end{scope}

\begin{scope}[shift={(-4.5,-15)}]

\draw[fill opacity=0.5,fill=gray!20]   (1.5,1) ellipse (0.5 and 0.5);
\draw[fill opacity=0.5,fill=gray!20]   (1.5,-0.5) ellipse (0.5 and 0.5);
\draw (1.5,0.5) -- (1.5,0);
\draw (1.5,-1) -- (1.5,-1.5);
\draw (1.5,1.5) -- (1.5,2);

\node[right] at (1.5,2) {$x_1$};
\node[below] at (1.5,-1.5) {$x_2$};
\node at (1.5,2) {$\times$};
\node at (1.5,-1.5) {$\times$};

\draw  (0.5,2.5) rectangle (2.5,-2);
\node at (2,-1.5) {$\underline \Gamma_{x_a}$};
\node[right] at (1.5,0.25) {$x_a$};
\end{scope}

\begin{scope}[shift={(-9.5,-15)}]

\draw[fill opacity=0.5,fill=gray!20]   (1.5,1) ellipse (0.5 and 0.5);
\draw[fill opacity=0.5,fill=gray!20]   (1.5,-0.5) ellipse (0.5 and 0.5);
\draw (1.5,0.5) -- (1.5,0);
\draw (1.5,-1) -- (1.5,-1.5);
\draw (1.5,1.5) -- (1.5,2);

\node[right] at (1.5,2) {$x_1$};
\node[below] at (1.5,-1.5) {$x_2$};
\node at (1.5,2) {$\times$};
\node at (1.5,-1.5) {$\times$};

\draw  (0.5,2.5) rectangle (2.5,-2);
\node at (2,-1.5) {$\underline \Gamma^a$};
\draw (1.375,0.25) -- (1.625,0.25);
\end{scope}

\draw[<-] (-10.5,0.75) -- node[midway, fill=white, rounded corners, draw=black]{i} (-9.5,0.75);
\draw[<-] (-10.5,-3.75) -- node[midway, fill=white, rounded corners, draw=black]{i }(-9.5,-3.75);
\draw[<-] (-10.5,-8.25) -- node[midway, fill=white, rounded corners, draw=black]{i} (-9.5,-8.25);

\draw[<-] (-6.5,0.5) -- node[midway, fill=white, rounded corners, draw=black]{ii} (-5,0.5);
\draw[<-] (-6.5,-4) -- node[midway, fill=white, rounded corners, draw=black]{ii }(-5,-4);
\draw[<-] (-6.5,-8.5) -- node[midway, fill=white, rounded corners, draw=black]{ii} (-5,-8.5);

\draw[->] (-2,-2.5) -- node[midway, fill=white, rounded corners, draw=black]{iii}  (-2,-3.5);
\draw[<-] (-2,-4.5) -- node[midway, fill=white, rounded corners, draw=black]{iii}  (-2,-5.5);
\draw (-12.5,-15.75) -- (-12.5,-15.25);
\node at (-7.5,-14.75) {$x_a$};
\end{tikzpicture} }
    \caption{Describing the moduli space of pre-broken regular treed disks. Here $\underline \Gamma$ is the combinatorial labelled type between two critical points $x_1$ and $x_2$ with two interior vertices. $\underline \Gamma_{\{t_a\}}$ and $\underline \Gamma_{\{t_b\}}$ have the same underlying combinatorial type, but additionally have pre-breakings at $t_a$ and $t_b$ respectively, labeled with classes $x_a, x_b\in \Crit(h)$. $\underline\Gamma', \underline \Gamma^a$, and $\underline \Gamma^b$ are broken types, whereas $\underline \Gamma_{1a},\underline \Gamma_{a2},\underline \Gamma_{1b}$ and $\underline \Gamma_{b2}$ are combinatorial types between $x_1, x_2, x_a$ and $x_b$ with indices indicating the labels. The points marked  $1-6$ refer to the pseudoholomorphic treed disks which appear in \cref{fig:pseudoholomorphicmoduli}.}
    \label{fig:labelledmodulie}
\end{figure}
\begin{figure}
    \centering
    \scalebox{.6}{\begin{tikzpicture}\begin{scope}[]
\begin{scope}[]
\begin{scope}[]

\draw[thick, dotted] (-12.5,-23.5) .. controls (-11.5,-23.5) and (1.5,-23.5) .. (2.5,-23.5) .. controls (4,-23.5) and (4,-17.5) .. (2.5,-17.5) .. controls (1.5,-17.5) and (-11.5,-17.5) .. (-12.5,-17.5);
\node[right] at (3.75,-20.5) {$\mathcal M_{\mathcal P}(X, L, D, \underline \Gamma)$};

\node[fill=black, circle, scale=.5] at (-12.5,-17.5) {};
\node[fill=black, circle, scale=.5] at (-12.5,-23.5) {};
\node[left] at (-12.5,-17.5) {};
\node[left] at (-12.5,-23.5) {};
\end{scope}
\begin{scope}[shift={(-11,-22)}]
\draw[fill=white]  (-1.5,3) rectangle (2.5,-2);
\node at (-1,-1.5){$L$};
\draw (-1.5,0.5) -- (2.5,0.5) (-1.5,1.5) -- (2.5,1.5);
\draw[fill opacity=.5,fill= gray!20]  (-0.5,1.5) ellipse (0.5 and 0.5);
\draw[fill opacity=.5,fill= gray!20]  (-0.5,-0.5) ellipse (0.5 and 0.5);
\draw (-0.5,1) -- (-0.5,0);
\draw (-0.5,-1) -- (0.5,-1.5);
\draw (-0.5,2) -- (0.5,2.5);

\node at (-0.5,0.5) {$\times$};
\node[below right] at (-0.5,0.5) {$x_a$};
\node[right] at (0.5,2.5) {$x_1$};
\node[below] at (0.5, -1.5) {$x_2$};
\node at (0.5,2.5) {$\times$};
\node at (0.5,-1.5) {$\times$};
\node[above] at (2,1.5) {$H_{x_b}$};
\node[below] at (-1,0.5) {$H_{x_a}$};
\node at (1.5,1.5) {$\times$};
\node[below right] at (1.5,1.5) {$x_b$};
\node at (-1,2.5) {$u_1$};
\end{scope}
\begin{scope}[shift={(0,-16)}]
\draw[fill=white]  (-1.5,3) rectangle (2.5,-2);
\node at (-1,-1.5){$L$};
\draw (-1.5,1) -- (2.5,1) (-1.5,2) -- (2.5,2);
\draw[fill opacity=.5,fill= gray!20]   (0.5,1.5) ellipse (0.75 and 0.75);
\draw[fill opacity=.5,fill= gray!20]   (0.5,0) ellipse (0.5 and 0.5);
\draw (0.5,0.75) -- (0.5,0.5);
\draw (0.5,-0.5) -- (0.5,-1.5);
\draw (0.5,2.25) -- (0.5,2.5);
\node at (-1,2.5) {$u_4$};

\node at (-0.5,1) {$\times$};
\node[below right] at (-0.5,1) {$x_a$};
\node[right] at (0.5,2.5) {$x_1$};
\node[below] at (0.5, -1.5) {$x_2$};
\node at (0.5,2.5) {$\times$};
\node at (0.5,-1.5) {$\times$};
\node[above] at (2,2) {$H_{x_b}$};
\node[below] at (-1,1) {$H_{x_a}$};
\node at (1.5,2) {$\times$};
\node[below right] at (1.5,2) {$x_b$};
\end{scope}
\begin{scope}[shift={(-5.5,-16)}]
\draw[fill=white]  (-1.5,3) rectangle (2.5,-2);
\node at (-1,-1.5){$L$};
\draw (-1.5,1) -- (2.5,1) (-1.5,2) -- (2.5,2);

\draw[fill opacity=.5,fill= gray!20]   (1.5,1) ellipse (0.5 and 0.5);
\draw[fill opacity=.5,fill= gray!20]   (1.5,-0.5) ellipse (0.5 and 0.5);
\draw (1.5,0.5) -- (1.5,0);
\draw (1.5,-1) -- (0.5,-1.5);
\draw (1.5,1.5) -- (0.5,2.5);
\node at (-1,2.5) {$u_5$};

\node[below left] at (1,2) {$t_b$};
\node[circle, fill, scale=.5] at (1,2) {};
\node at (-0.5,1) {$\times$};
\node[below right] at (-0.5,1) {$x_a$};
\node[right] at (0.5,2.5) {$x_1$};
\node[below] at (0.5, -1.5) {$x_2$};
\node at (0.5,2.5) {$\times$};
\node at (0.5,-1.5) {$\times$};
\node[above] at (2,2) {$H_{x_b}$};
\node[below] at (-1,1) {$H_{x_a}$};
\node at (1.5,2) {$\times$};
\node[below right] at (1.5,2) {$x_b$};
\end{scope}
\begin{scope}[shift={(-11,-16)}]
\draw [fill=white] (-1.5,3) rectangle (2.5,-2);
\node at (-1,-1.5){$L$};
\draw (-1.5,1) -- (2.5,1) (-1.5,2) -- (2.5,2);
\draw[fill opacity=.5,fill= gray!20]   (1.5,1) ellipse (0.5 and 0.5);
\draw[fill opacity=.5,fill= gray!20]   (1.5,-0.5) ellipse (0.5 and 0.5);
\draw (1.5,0.5) -- (1.5,0);
\draw (1.5,-1) -- (0.5,-1.5);
\node at (-1,2.5) {$u_6$};

\node at (-0.5,1) {$\times$};
\node[below right] at (-0.5,1) {$x_a$};
\node[right] at (0.5,2.5) {$x_1$};
\node[below] at (0.5, -1.5) {$x_2$};
\node at (0.5,2.5) {$\times$};
\node at (0.5,-1.5) {$\times$};
\node[above] at (2,2) {$H_{x_b}$};
\node[below] at (-1,1) {$H_{x_a}$};
\node at (1.5,2) {$\times$};
\node[below right] at (1.5,2) {$x_b$};
\end{scope}
\begin{scope}[shift={(0,-22)}]
\draw[fill=white]  (-1.5,3) rectangle (2.5,-2);
\node at (-1,-1.5){$L$};
\draw (-1.5,0.5) -- (2.5,0.5) (-1.5,1.5) -- (2.5,1.5);
\draw[fill opacity=.5,fill= gray!20]  (0.5,1.15) ellipse (0.5 and 0.5);
\draw[fill opacity=.5,fill= gray!20]  (0.5,-0.35) ellipse (0.5 and 0.5);
\draw (0.5,0.65) -- (0.5,0.15);
\draw (0.5,-0.85) -- (0.5,-1.5);
\draw (0.5,1.65) -- (0.5,2.5);
\node at (-1,2.5) {$u_3$};

\node at (-0.5,0.5) {$\times$};
\node[below right] at (-0.5,0.5) {$x_a$};
\node[right] at (0.5,2.5) {$x_1$};
\node[below] at (0.5, -1.5) {$x_2$};
\node at (0.5,2.5) {$\times$};
\node at (0.5,-1.5) {$\times$};
\node[above] at (2,1.5) {$H_{x_b}$};
\node[below] at (-1,0.5) {$H_{x_a}$};
\node at (1.5,1.5) {$\times$};
\node[below right] at (1.5,1.5) {$x_b$};
\node[below right] at (0.5,0.5) {$t_a$};
\node[circle, fill, scale=.5] at (0.5,0.5) {};
\end{scope}
\begin{scope}[shift={(-5.5,-22)}]
\draw[fill=white]  (-1.5,3) rectangle (2.5,-2);
\node at (-1,-1.5){$L$};
\draw (-1.5,0.5) -- (2.5,0.5) (-1.5,1.5) -- (2.5,1.5);
\draw[fill opacity=.5,fill= gray!20]  (0,1.5) ellipse (0.5 and 0.5);
\draw[fill opacity=.5,fill= gray!20]  (0,-0.5) ellipse (0.5 and 0.5);

\draw (0,1) -- (0,0);
\draw (0,-1) -- (0.5,-1.5);
\draw (0,2) -- (0.5,2.5);
\node at (-1,2.5) {$u_2$};

\node at (-0.5,0.5) {$\times$};
\node[below right] at (-0.5,0.5) {$x_a$};
\node[right] at (0.5,2.5) {$x_1$};
\node[below] at (0.5, -1.5) {$x_2$};
\node at (0.5,2.5) {$\times$};
\node at (0.5,-1.5) {$\times$};
\node[above] at (2,1.5) {$H_{x_b}$};
\node[below] at (-1,0.5) {$H_{x_a}$};
\node at (1.5,1.5) {$\times$};
\node[below right] at (1.5,1.5) {$x_b$};
\node[above right] at (0,0.5) {$t_a$};
\node[circle, fill, scale=.5] at (0,0.5) {};
\end{scope}
\draw  plot[smooth, tension=.7] coordinates {(-10.5,-13.5)};
\draw (-10.5,-13.5) .. controls (-10.2,-13.75) and (-9.5,-13.7) .. (-9.5,-14) .. controls (-9.5,-14.5) and (-9.5,-14) .. (-9.5,-14.5);

\end{scope}

\end{scope}

\node[circle, draw=black] at (-9,-23.5) {1};
\node[circle, draw=black] at (-3.5,-23.5) {2};
\node[circle, draw=black] at (2,-23.5) {3};
\node[circle, draw=black] at (2,-17.5) {4};
\node[circle, draw=black] at (-3.5,-17.5) {5};
\node[circle, draw=black] at (-9,-17.5) {6};
\end{tikzpicture} }
    \caption{An example of the moduli space of labeled pseudoholomorphic pearly flow lines of type $\underline \Gamma$. The moduli space is represented by the dotted line, and the diagrams represent the image of $u_i(T)$ and $u_i(\partial S)$ inside of $L$ for various maps $u_i$. We also mark critical points $x_1, x_2, x_a, x_b$ of a Morse function $h:L\to \RR$, as well as level sets $H_{x_i}$. The projection $\underline C_P(\mathcal M^{un}_{\mathcal P}(X, L, D, \underline \Gamma))$ to the moduli space of types (drawn in \cref{fig:labelledmodulie}) is represented by the dotted line in that figure, with the circles numerals 1-6 identifying the combinatorial type of each of these pseudoholomorphic treed disks. }
    \label{fig:pseudoholomorphicmoduli}
\end{figure}

\subsection{Perturbation Datum and Coherence}
\label{subsec:perturbationdatum}
A perturbation datum for type $\underline \Gamma_P$ \cite[Definition 4.10]{charest2015floer}  is an assignment of domain and pre-breaking dependent Morse function (i.e. map $h_\Gamma: \overline {\mathcal T}(\underline \Gamma_P)\times L\to \RR$) and domain and pre-breaking dependent tame almost complex structure ($J_\Gamma: \overline {\mathcal S}(\underline \Gamma_P) \times X \to \End(TX)$) to each pre-broken treed disk $\underline C_P$. These are required to agree with the prescribed Morse function $h: L\to \RR$ and tame almost complex structure $J: TX\to TX$ near the boundaries of the edges and disks. We will denote perturbation datum for type $\underline \Gamma_P$ by $\mathcal P(\underline \Gamma_P)$, so that 
\[\mathcal P(\underline \Gamma_P)(\underline C_P, x)=\left\{ \begin{array}{cc}
    J(\underline \Gamma_P)(\underline C_P,x)\in \End(TX)& \text{ if $(\underline C_P,x)\in \overline {\mathcal S}(\underline \Gamma_P)$ }\\
    h(\underline \Gamma_P)(\underline C_P, x)\in C^\infty(L, \RR)&  \text{ if $(\underline C_P, x)\in \overline {\mathcal T}(\underline \Gamma_P)$}
\end{array}\right. \] 

\begin{definition}
A \emph{stabilizing Morse function} is a Morse function $h: L\to \RR$ whose critical values are distinct; for $x_i\in \Crit(h)$, denote the level set of $h$ containing the critical point $x_i$ by  
$H_{x_i}:=\{h^{-1}(h(x_i))\}.$

A perturbation datum $\mathcal P(\underline \Gamma_P)$ is \emph{stabilizing with respect to $D$} if it satisfies the condition of \cite[Definition 4.26]{charest2015floer}; this roughly states that every solution to the $\mathcal P(\underline \Gamma_P)$-perturbed Floer equation will meet the stabilizing divisor in the appropriate number of intersection points (see \cref{subsec:cobordismstablebackground}).
The perturbations of almost complex structure must be chosen in such a way that the $D$ remains an almost complex hypersurface.

A perturbation datum $\mathcal P(\underline \Gamma_P)$ is stabilizing with respect to  $h$ if, away from the critical points, the gradient $\nabla h(\underline \Gamma_P)(C_P, x)$ is transverse to $h^{-1}(\Crit(h))$.
We say that $\mathcal P(\underline \Gamma_P)$ is stabilizing if it is stabilizing for both $D$ and $h$.

\label{def:stabilizingstuff}
\end{definition}

Given a regularizing perturbation datum, one can define $\mathcal P(\underline \Gamma_P)$- perturbed pseudoholomorphic treed disks, which are maps $u: \underline C_P\to X$ satisfying the conditions of \cite[Definition 4.13]{charest2015floer}. 
\begin{definition}[Modification of Definition 4.17, \cite{charest2015floer}]
    We say that a stable pre-broken pseudoholomorphic treed disk $u: \underline C_P\to X$ with boundary on $L$ is adapted to $D$ and $h$ if it satisfies the stable surface (a) and leaf (b) axiom  from \cite[Definition 4.17]{charest2015floer} and additionally
    \begin{description}
        \item[(c) (Pre-broken axiom)] Each pre-broken point $t_i\in \underline C_P$ with label $x_i$ is mapped to the $H_{x_i}$. Furthermore, each connected component of $u^{-1}(H_{x_i})$ contains a pre-broken point $t_i$ with label $x_i$.
    \end{description}
\end{definition}
One upshot of using labeled treed disks is that as $\underline \Gamma$ contains the data of which critical points of $h$ its flow-lines limit to, we may now define $\ind(\underline \Gamma)$ the expected dimension of the moduli space of labeled treed holomorphic disks. 
We denote by $\mathcal M_{\mathcal P}  (X, L, D,h, \underline \Gamma_P)$ the moduli space of $\mathcal P$-perturbed pseudoholomorphic pre-broken labeled treed disks which are adapted to $D$ and $h$.
Because the perturbations that we choose for Morse functions are still required to be transverse to the level sets $H_{x_i}$ for $x_i \in \Crit(h)$, transversality for the conditions imposed by the pre-broken axiom are automatically satisfied on the interior of $\mathcal M_{\mathcal P}(X, L, D, h, \underline \Gamma_P)$.
As before, we will denote  the forgetful map $\underline C_P: \mathcal M_{\mathcal P}(X, L, D,h, \underline \Gamma_P)\to \mathcal M(\underline \Gamma_P)$.
We say that the perturbation datum $\mathcal P(\underline \Gamma_P)$ is \emph{regular} if the moduli space $ \mathcal M_{\mathcal P}(X, L, D, h,\underline \Gamma_P)$ is cut out transversely.
\begin{remark}
    Suppose that $\underline \Gamma_{P}$ is a combinatorial type with pre-breakings. Let $\mathcal M_{\mathcal P}^{un}(X, L, D, \underline \Gamma_P)$ be the moduli space of pseudoholomorphic treed disks which are adapted to $D$, but only satisfy the weakening of the pre-broken axiom:
    \begin{description}
        \item[(c') (Weak pre-broken axiom)] Each pre-broken point $t_i\in \underline C_P$ with label $x_i$ is mapped to the $H_{x_i}$.
    \end{description}
    We note that if $P'\subset P$ is a subset of labels, and $\mathcal P(\underline \Gamma_P)$ arises from $\mathcal P(\underline \Gamma_{ P'})$ by pullback under forgetting pre-breakings, then every adapted $\mathcal P(\underline \Gamma_P)$ pseudoholomorphic disk is an example of a $\mathcal P(\underline \Gamma_{P'})$ weakly adapted pseudoholomorphic disk, simply given by forgetting the position of the marked points in $P\setminus P'$. Since the pre-broken axiom determines the positions of these marked points uniquely, we obtain an inclusion
    \[\mathcal M_{\mathcal P}(X, L, D,h, \underline \Gamma_P)\subset \mathcal M_{\mathcal P}^{un}(X, L, D, \underline \Gamma_{P'}).\]
    which is a diffeomorphism onto its image.
    \label{remark:weaklyadapted}
\end{remark}

We will want perturbations that are chosen consistently between different labeled combinatorial types. Let $\mathbb T$ denote the set of all pre-broken combinatorial types for $(L, h)$, and use $\mathbb I\subset \mathbb T$ to denote a subset of the pre-broken combinatorial types. 
A \emph{perturbation system}  $\mathcal P_{\mathbb I}$ is a choice of perturbation datum $\mathcal P_{\mathbb I}(\underline \Gamma_P)$ for every pre-broken combinatorial type $\underline \Gamma_P\in \mathbb I$.
If $\mathcal P_{\mathbb T}$ is a perturbation system for all pre-broken combinatorial types in $\mathbb T$, we will call $\mathcal P_{\mathbb T}$ a \emph{full perturbation system}.
For a fixed perturbation system $\mathcal P_{\mathbb I}$ and pre-broken combinatorial type $\underline \Gamma_P\in \mathbb I$, we write  $\mathcal M_{\mathcal P_{\mathbb I}}(X, L, D, h,\underline \Gamma_P)$ for the moduli space of $\mathcal P_{\mathbb I}(\underline \Gamma_P)$-perturbed pseudoholomorphic treed disks.

Adopting notation from the discussion preceding \cite[Theorem 4.19]{charest2015floer}, we write $\underline\Gamma'_{P'}\prec \underline \Gamma_P$ if there exists a
\begin{itemize}
    \item Morphism of combinatorial type $\underline\Gamma'_{P'}\to \underline\Gamma_P$ corresponding to collapsing edges/making an edge or weight finite or non-zero; or
    \item Morphism of combinatorial types $\underline \Gamma_P\to \underline\Gamma'_{P'}$ given by cutting edges; or
    \item  $\underline\Gamma_P=\underline\Gamma^1_{P_1}\cup \underline \Gamma^2_{P^2}$ is a disconnected type and $\underline\Gamma'_{P'}=\underline\Gamma^1_{P_1}$ or $\underline\Gamma^2_{P_2}$; or
    \item Morphism of combinatorial types $\underline \Gamma_P\to \underline \Gamma'_{P'}$ given by forgetting pre-broken labels.
\end{itemize}
\begin{definition}
If $\mathbb I$ satisfies the property that $\underline \Gamma_P\in \mathbb I$ and $\underline \Gamma'_P \prec \underline \Gamma_P$ implies $\underline \Gamma'_P\in \mathbb I$, then we say that $\mathbb I$ is \emph{downward closed.}
\label{def:downwardclosed}
\end{definition}

\begin{definition}[Following Definition 4.12 \cite{charest2015floer}] \label{def:coherence} We say that $\mathcal P(\underline \Gamma)$ and $\mathcal P(\underline \Gamma')$ are coherent over  $\underline \Gamma'_{P'}\prec \underline \Gamma_P$ if 
\begin{enumerate}
    \renewcommand{\theenumi}{\alph{enumi}}
    \item If there is a morphism of type  (collapsing edges/making an edge non-zero) $\underline \Gamma'_{P'}\to \underline\Gamma_P$ with $\underline \Gamma_P,\underline \Gamma'_{P'}\in \mathbb I$, then $\mathcal P(\underline \Gamma'_P)$ is the pullback of the perturbation $\mathcal P(\underline \Gamma_P)$; or
    \item If there exists a morphism of cutting edges $\underline\Gamma_P\to\underline \Gamma'_{P'}$, we require $\mathcal P(\underline \Gamma'_P)$ is the pushforward of the perturbation $\mathcal P(\underline \Gamma_P)$; or if $\underline \Gamma'_{P'}$ is a disjoint union of $\underline \Gamma^1_{P^1}$ and $\underline \Gamma^2_{P^2}$, then we require the perturbation arises as pullback.
    \item,(d) The Locality and Forgettable edges axiom from \cite[Definition 4.12]{charest2015floer}.
\end{enumerate}
In addition to these coherence conditions, we modify the condition for breaking (so that it now incorporates the pre-breaking point) and add a new condition for forgetting labels.
\begin{enumerate}
    \renewcommand{\theenumi}{\alph{enumi}}
    \addtocounter{enumi}{4}
    \item (Coherence under making an infinite edge finite) Suppose that $\underline \Gamma'_{P'}\to \underline\Gamma_P$ is a morphism of making an edge finite. 
    Recall we have an identification $\mathcal M(\underline \Gamma'_{P'})\times [0, \epsilon)\times (1/2, \ell(e)-1/2)\subset \mathcal M(\underline \Gamma_P)$, where $P'$ has one more pre-breaking than $P$.
    Given $t_s\in (1/2,\infty)\sqcup (-\infty, -1/2)$, consider the inclusion
    \begin{align*}
        i_{t_s}: \mathcal M(\underline \Gamma'_{P'})\into& \mathcal M(\underline \Gamma_P)\\
                C\mapsto (C, 0, t_s)
    \end{align*}
    Denote $i_{t_s}: \mathcal U(\underline\Gamma'_{P'})\into \mathcal U(\underline \Gamma_P)$ the corresponding inclusion of the universal curve. 
    We require that $\mathcal P(\underline \Gamma'_{P'})=i_{t_s}^*\mathcal P(\underline \Gamma_P)$ for all $t_s$, so that near the boundary the perturbation is not dependent on the position of $t_s$.\label{item:coherenceoverfinite}
    \item (Coherence under forgetting labels) Whenever $P'\subset P$, so that  we have a forgetful map  $\mathcal U^{P'<1/4}(\underline \Gamma_P)\to \mathcal U^{P\setminus P'}(\underline \Gamma_{P\setminus P'})$, we require the perturbation datum arises as pullback over the forgetting pre-markings map, that is
    \[\mathcal P|_{\mathcal U^{P'<1/4}}(\underline \Gamma_P)= \res_{P'\subset P}^*\mathcal P(\underline \Gamma_{P\setminus P'})).\]
    
\end{enumerate}
We say that $\mathcal P_{\mathbb I}$ is coherent if $\mathbb I$ is downward closed, and $\mathcal P_{\mathbb I}$ is coherent over all relations $ \underline \Gamma'_{P'}\prec \underline \Gamma_P,\in \mathbb I$.
\end{definition}
\begin{remark}
    The revised (coherence under making an infinite edge finite) condition states that if $t_i$ belongs to an edge that is about to break, then perturbations do not depend on the position of $t_i$.

    Note that (coherence under forgetting labels) means that near the region where a pre-breaking $t_i$ approaches the boundary of an edge, the perturbation system is not allowed to depend on the position of the pre-breaking $t_i$.
\end{remark}

\subsection{From coherent perturbations to Floer cohomology}
\label{subsection:coherencetofloer}
We say that a stabilizing perturbation system $\mathcal P_{\mathbb I}$ is \emph{regular} if it is regular for all labeled types $\underline \Gamma_P\in \mathbb I$ with   $\ind(\underline \Gamma_P)\leq 1$.
Suppose we have a regular coherent perturbation system, then by (Coherence under forgetting labels) and \cref{remark:weaklyadapted},
\[ \mathcal M_{\mathcal P_{\mathbb I}}(X, L, D, h, \underline \Gamma_P)|_{\mathcal M^{P'<1/4}(\underline \Gamma_P)}\subset \mathcal M_{\mathcal P_{\mathbb I}}^{un}(X, L, D, \underline \Gamma_{P\setminus P'}).\]
Furthermore, when these spaces are regularly cut out, this is a diffeomorphism onto its image.
We say that for pre-breakings $P_1, P_2$ and curves $u_i\in  \mathcal M_{\mathcal P_{\mathbb I}}(X, L, D, h, \underline \Gamma_{P_i})|_{\mathcal M^{P<1/4}(\underline \Gamma)}$ that   $u_1\sim u_2$ if they have the same image in $\mathcal M_{\mathcal P_{\mathbb I}}^{un}(X, L, D,  \underline \Gamma)$. 
The coherence under forgetting labels property allows us to construct a moduli space of $\mathcal P_{\mathbb I}$-perturbed pseudoholomorphic disks which do not have boundary components arising from pre-breakings moving to the boundary of flow lines (\cref{subsub:newboundarytypes}:\cref{subsubitem:newboundarytypes}). In summary:
\begin{claim}
    Suppose that $\mathcal P_{\mathbb I}$ is a regular coherent perturbation system. For labeled type $\underline \Gamma$, we can define a  moduli space
    \[\mathcal M_{\mathcal P_{\mathbb I}}(X, L, D, \underline \Gamma):=\bigcup_{\text{labellings $P$ for $\underline \Gamma$}} \mathcal M_{\mathcal P_{\mathbb I}}(X, L, D,h ,\underline \Gamma_P) / \sim\]
    which is a smooth manifold.
\end{claim}
See \cref{fig:pseudoholomorphicmoduli} in conjunction with \cref{fig:labelledmodulie} for how $\mathcal M_{\mathcal P_{\mathbb I}}(X, L, D, \underline \Gamma)$ is assembled from charts consisting of pre-broken types.

\subsubsection{Extending perturbation systems}
When proving the existence of a full perturbation system it is desirable to construct a regular coherent perturbation system for some small collection of types $\mathbb I$, and then extend this regular coherent perturbation system to a full perturbation system.
\begin{definition} Let $\mathbb I$ be a collection of types, and $\mathcal P_{\mathbb I}$ be a perturbation system. We say that $\mathcal P_{ \mathbb J}$ extends $\mathcal P_{\mathbb I}$ if $\mathbb I\subset \mathbb J$ and for every $\underline \Gamma_P\in \mathbb I$
        \[\mathcal P_{\mathbb I}(\underline \Gamma_P)= \mathcal P_{\mathbb J}(\underline \Gamma_P).\]
\end{definition}
We will assume that \cite[Theorem 4.19]{charest2015floer} extends to pre-broken treed disks.
\begin{assumption}[Existence of Pearly Model for Labeled types ]
    Let $L\subset X$ be a Lagrangian submanifold, and $D\subset X$ be a weakly stabilizing divisor. Suppose we have already picked $\mathcal P_{\mathbb I}$  a regular coherent perturbation system for a downward closed set of combinatorial types $\mathbb I$. Then there exists a regular full coherent perturbation system $\mathcal P_{\mathbb T}$ extending $\mathcal P_\mathbb{I}$. Additionally, $\mathcal P_{\mathbb T}$ can be chosen so that whenever we have Gromov compactness for choices of perturbations and pseudoholomorphic treed disks and $\ind(\underline \Gamma)\leq 1$, the moduli spaces of a regularly coherent and weakly stabilized perturbation system have appropriate tubular neighborhoods providing a compactification. The orientations of these neighborhoods agree with the boundary stratification(following \cite[Theorem 4.19]{charest2015floer}):
    \begin{equation}
        \partial \overline {\mathcal M}_{\mathcal P_{\mathbb T}}(X, L, D, \underline \Gamma)=\bigcup_{\underline \Gamma^1\circ \underline \Gamma^2=\underline \Gamma} \overline{\mathcal M}_{\mathcal P_{\mathbb T}}(X, L, D, \underline \Gamma^1)\times \overline{\mathcal M}_{\mathcal P_{\mathbb T}}(X, L, D, \underline \Gamma^2)
    \end{equation}
    where $\underline \Gamma^2\circ \underline \Gamma^1$ is the concatenation of labeled trees where the outgoing label of $\underline \Gamma^1$ matches an incoming label for $\underline \Gamma^2$.
    \label{assum:compact}
\end{assumption}
The  differences between \cite[Theorem 4.19]{charest2015floer} and \cref{assum:compact} come from pre-breakings. With regards to coherence over the morphism of making an infinite edge finite (which has been modified), we note that the new coherence condition (\cref{def:coherence}:\cref{item:coherenceoverfinite}) forces the perturbation to be independent of the position of pre-breaking in a neighborhood of where breaking occurs. As a consequence, the argument for producing the boundary stratification is no different than in the domain (and label independent) setting. 

\subsubsection{From Coherence to $A_\infty$ relations}
Let $\mathcal P_{\mathbb T}$ be a full coherent regular perturbation system for $(X, L, D)$.
For $\beta\in H_2(X, L)$, let $ {\mathcal M}_{\mathcal P_{\mathbb T}}(X, L, D, \underline \Gamma,\beta)$ be the set of pseudoholomorphic treed disks where the sum of the homology classes of the disk components is $\beta$.  For fixed $\underline x$ a set of labels, and $\beta\in H_2(X, L)$, we denote by
\[
    \overline{\mathcal M}_{\mathcal P_{\mathbb T}}(X, L, D, \underline x,  \beta)=\bigcup_{\underline \Gamma \text{with labels $\underline x$}}\overline {\mathcal M}_{\mathcal P_{\mathbb T}}(X, L, D, \underline \Gamma,\beta)/\sim
\]
where these spaces are glued along their matching boundary components (corresponding to making an edge length non-zero).
Whenever these moduli spaces have expected dimension less than or equal to one, we have a boundary stratification
\begin{equation}
\partial \overline {\mathcal M}_{\mathcal P}(X, L, D, \underline y, \beta)=\bigcup_{\substack{\underline x_1\circ_k \underline x_2=\underline y\\\beta_1+\beta_2=\beta}} \mathcal M_{\mathcal P_{\mathbb T}}(X, L, D, \underline x_1, \beta_1)\times \mathcal M_{\mathcal P_{\mathbb I}}(X, L, D, \underline x_2, \beta_2)
\end{equation}
where $\underline x_1\circ_k \underline x_2=\underline y$ is the concatenation of the outgoing label of $\underline x_2$ with the $k$th incoming label of $\underline x_1$.
The count of treed disks in these moduli spaces allows us to construct the pearly model of $L$. 
\begin{definition}[Following Definition 4.28 \cite{charest2015floer}]
    Let $L$ be a Lagrangian submanifold as in \cref{assum:compact}, and let $\mathcal P_{\mathbb T}$ be a full coherent regular perturbation system. The \emph{pearly Floer cohomology} of $L$ is the $A_\infty$ algebra $\CF(L):=\Lambda\langle \Crit(h)\rangle$, whose $A_\infty$  structure coefficients are given by (\cite[Equation 4.34]{charest2015floer})
    \[\langle m^k(x_1, \ldots, x_k), x_0\rangle :=\sum_{\beta\in H_2(X, L)} (-1)^\heartsuit (\sigma(u)!)^{-1}T^{\omega(\beta)} \cdot \# \mathcal M_{\mathcal P}(X, L, D, \underline x, \beta)\]
    where $\underline x=\{x_0;x_1, \ldots x_k\}$, $\#$ is the count of $0$-dimensional components of the moduli space with orientation, $\sigma(u)$ denotes the number of interior marked points, and $\heartsuit=\sum_{i=1}^k i |x_i|$.
\end{definition}
\subsubsection{Remarks on relation to abstract perturbation techniques}
\label{rem:abstractperturbations}
    Given the large number of different approaches to regularizing Floer theory, we provide some context for where this framework fits in with respect to other constructions.
    Regularization techniques can be characterized by how ``uniform'' or ``dependent'' they are. 
    Generally, the uniform regularizations retain much of the underlying geometric data of the symplectic manifold and almost complex structure chosen, and so they can be useful for computations. However, this must be balanced against the additional flexibility that comes with using dependent regularizations, which is necessary for achieving regularity in some settings and can be advantageous for proving abstract properties of Floer cohomology.
    We give a list of different regularization techniques below, ordered from ``most uniform'' to ``most dependent''.
    \begin{itemize}
        \item Global perturbations of almost complex structures place us in the most geometric setting and therefore provide us with the best tools for explicit computations of moduli spaces of pseudoholomorphic treed disks. One can hope with the right choice of global perturbations that techniques like the open-mapping principle still hold, and thereby use geometric principles to characterize the moduli spaces contributing to the structure coefficients of the $A_\infty$ structure. Problematically, global perturbations are not flexible enough to give us regularization in examples we consider (in particular, the non-monotone setting).
        \item Domain-dependent perturbations give us all the flexibility we need to construct perturbations for defining the pearly Floer cochains in the non-monotone setting. However, it is difficult to explicitly construct a domain-dependent perturbation. 
        \item Domain and label dependent perturbations (what we use here) give us slightly more flexibility before. The upshot is that one may be able to use geometric perturbations for particular labeled combinatorial types contributing to particular structure constants in the Floer cohomology.
        \item Abstract perturbation techniques (from polyfolds and Kuranishi structures) give us the greatest flexibility. The gain is that it is possible to be extremely precise about what kinds of perturbation you want to take. However, the perturbations will generally not come from any kind of geometric source, making some portions of the theory more difficult to work with. 
    \end{itemize}
    It is a general expectation that one should be able to use regularizing perturbation datum from one of these ``uniform'' regularization techniques inside the framework of a ``dependent'' regularization technique.
    For instance, one could construct a ``domain-dependent'' perturbation by just defining it everywhere to be given by a global perturbation of almost complex structure. 

    The trick, which we now exploit, is to use uniform perturbations where you need to use geometric properties for computations, then later use dependent perturbations to achieve regularity and coherence more broadly.\footnote{In polyfold theory, Katrin Wehrheim explained this once as \emph{the Obamacare principle}, based on the idea that ``If you like the plan you have, you can keep it'' \cite{obamaspeech}. }
    For our applications, domain and label-dependent perturbations prove to be flexible enough that we can choose the perturbations we want where we want them.

    \subsection{Application of domain and label dependent perturbations: explicit computation of structure coefficients.}
    \label{subsec:geometricflowlines}
    We now give an application of using domain and label-dependent perturbations. 

    Consider the situation from \cref{subsub:domainandlabelsummary}. Recall that the hypothesis stated that there was a unique $J^0$=regular holomorphic disk $u_{ex}$ with one intersection with the stabilizing divisor.  We assumed that $u_{ex}$ intersected all upward and downward flow spaces of critical points transversely. 
    We also looked at two Morse critical points $x_1, x_0$ in degrees $1$ and $2$ respectively without a flow-line between them. From this we concluded that treed disks with combinatorial type in
    \[\mathbb I^{\leq 1}_{*; x_0}:=\left\{\underline \Gamma_P\;|\;\ind(\underline \Gamma_P)\leq 0, \parbox{6cm}{Input is empty or $x_1$, output is $x_0$, at most 1 interior marked point} \right\}\]
    are regularly cut for $\mathcal P^0$ the trivial perturbation determined by $J^0$ and $h$.
 We can additionally prove from the hypothesis that there are no broken flow trees with output on $x_0$ and one interior marked point. 
    
     \textbf{Case 1: $\ver(\underline \Gamma)=0$.}  Let $\underline \Gamma$ be the labeled combinatorial type of treed disks with 1 vertex labeled with the class of $u_{ex}$, and root labeled with $x_0$ , as in \cref{fig:m0breakings}. As the disk $u$ has the minimal number of stabilizing interior marked points, the only kinds of broken configurations which may occur are those where the semi-infinite edge breaks (as indicated by $\underline \Gamma'$ in the right-hand side of \cref{fig:m0breakings}). 
     As the disk $u_{ex}$ is transverse to all upwards and downward flow spaces and is regular, the configuration with output on $y$ is a regular configuration; this implies that $y$ has degree $2$. This means that the flow-line between $x_0$ and $y_0$ has expected dimension $-1$; since the function $h$ is Morse-Smale, there are no such flow lines. Therefore, there are no broken configurations of the type $\underline \Gamma'$. 
     
     \textbf{Case 2: $\ver(\underline \Gamma)=1$.}  In the case that we are looking at pearly trajectories between $x_0$ and $x_1$, the combinatorial types of breaking which can occur (due to the minimal interior marked point conditions) are listed in \cref{fig:breakings}.
     By the same reasoning as the case with one output, the first two types of breaking (of combinatorial type) $\underline\Gamma^1$ and $\underline\Gamma^2$ cannot occur. Without additional conditions, the combinatorial type $\underline\Gamma^3$ or $\underline\Gamma^4$ may occur (for instance, if there are non-regular Morse flow trees with inputs $x_1, y$ and output $x_0$). As we assumed that there are no Morse flow lines from $x_1$ to $x_0$, there are no unperturbed Morse flow trees that contain $x_1$ as input and $x_0$ as an output. Therefore the moduli space of unperturbed pearly flow trees of type $\underline\Gamma^3$ or $\underline\Gamma^4$ is empty.
     \begin{figure}
         \centering 
         \begin{subfigure}{\linewidth}
             \centering
         \scalebox{.75}{\begin{tikzpicture}

\begin{scope}[shift={(-9,-2.5)}]
\draw[fill=gray!20]  (2,2.5) ellipse (0.5 and 0.5);
\node at (2, 2.5) {$\times$};
\draw (2,2) -- (2,1);
\draw  (3,0.5) rectangle (1,4);

\node[right] at (2,1) {$x_0$};
\node[circle, fill=black, scale=.2] at (2,1) {};

\node at (1.5,3.5) {$\underline\Gamma$};

\end{scope}

\begin{scope}[shift={(-5.5,-1.5)}]
\draw[fill=gray!20]  (3.5,1.5) ellipse (0.5 and 0.5);
\draw (3.5,1) -- (3.5,0);
\node[right] at (3.5,0) {$x_0$};
\node[right] at (3.5,0.5) {$y$};
\node[circle, fill=black, scale=.2] at (3.5,0) {};
\node[circle, fill=black, scale=.2] at (3.5,0.5) {};
\draw  (2.5,3) rectangle (4.5,-0.5);
\node at (3.5,1.5) {$\times$};

\node at (3,2.5) {$\underline\Gamma'$};

\end{scope}

\draw [decorate, decoration={brace, amplitude=10pt,mirror,raise=4pt},yshift=0pt]
(-0.5,-2) -- (-0.5,1.5) node [black,midway,xshift=0.8cm,right,text width=5cm] {
This breaking cannot occur for index reasons.};
\draw[->] (-5.5,0) -- node[above, midway]{Breaking} (-3.5,0);
\end{tikzpicture} }
         \caption{Possible breaking of type with output on $x_0$.}
         \label{fig:m0breakings}
         \end{subfigure}
         \begin{subfigure}{\linewidth}
             \centering
         \scalebox{.75}{
\usetikzlibrary{matrix, arrows,  decorations.markings,  patterns,  plotmarks, decorations.pathreplacing}
\begin{tikzpicture}

\begin{scope}[shift={(-9,-2.5)}]
\draw[fill=gray!20]  (2,2.5) ellipse (0.5 and 0.5);
\node at (2, 2.5) {$\times$};
\draw (2,2) -- (2,1);
\draw (2,4) -- (2,3);
\draw  (3,0.5) rectangle (1,4.5);

\draw (2,2) -- (2,1.5);
\node[right] at (2,4) {$x_1$};
\node[right] at (2,1) {$x_0$};
\node[circle, fill=black, scale=.2] at (2,4) {};
\node[circle, fill=black, scale=.2] at (2,1) {};

\node at (1.5,4) {$\underline\Gamma$};

\end{scope}

\begin{scope}[shift={(-1.5,-1.5)}]
\draw (2,3) -- (2,2.5);
\draw[fill=gray!20]  (2,1) ellipse (0.5 and 0.5);
\node at (2, 1) {$\times$};
\draw (2,0.5) -- (2,0);
\draw (2,2.5) -- (2,1.5);
\node[right] at (2,3) {$x_1$};
\node[right] at (2,0) {$x_0$};
\node[circle, fill=black, scale=.2] at (2,3) {};
\node[circle, fill=black, scale=.2] at (2,0) {};
\node[right] at (2,2.5) {$y$};
\node[circle, fill=black, scale=.2] at (2,2.5) {};

\draw  (1,3.5) rectangle (3,-0.5);

\node at (1.5,3) {$\underline\Gamma^2$};
\end{scope}

\begin{scope}[shift={(-5.5,-1.5)}]
\draw[fill=gray!20]  (3.5,1.5) ellipse (0.5 and 0.5);
\node at (3.5,1.5) {$\times$};
\draw (3.5,1) -- (3.5,0);
\draw (3.5,3) -- (3.5,2.5)  (3.5,2.5) -- (3.5,2);
\node[right] at (3.5,3) {$x_1$};
\node[right] at (3.5,0) {$x_0$};
\node[right] at (3.5,2.5) {$y$};
\node[circle, fill=black, scale=.2] at (3.5,3) {};
\node[circle, fill=black, scale=.2] at (3.5,0) {};
\node[circle, fill=black, scale=.2] at (3.5,2.5) {};
\draw  (2.5,3.5) rectangle (4.5,-0.5);

\node at (3,3) {$\underline\Gamma^1$};

\end{scope}

\begin{scope}[shift={(-2,-1.5)}]
\draw[fill=gray!20]  (5,2) ellipse (0.5 and 0.5);
\node at (5,2) {$\times$};
\draw (5,1.5) -- (6,0.5) -- (6,0);
\draw (6,3) -- (6,0.5);
\node[left] at (6,3) {$x_1$};
\node[circle, fill=black, scale=.2] at (6,3) {};
\node[left] at (6,0) {$x_0$};
\node[circle, fill=black, scale=.2] at (6,0) {};
\node[right] at (5.5,1) {$y$};
\node[circle, fill=black, scale=.2] at (5.5,1) {};
\draw  (6.5,-0.5) rectangle (4,3.5);
\node at (4.5,3) {$\underline\Gamma^3$};

\end{scope}

\begin{scope}[]
\draw[fill=gray!20]  (6.5,0.5) ellipse (0.5 and 0.5);
\node at (6.5,0.5) {$\times$};
\draw (6.5,0) -- (5.5,-1) -- (5.5,-1.5);
\draw (5.5,1.5) -- (5.5,-1);
\node[right] at (5.5,1.5) {$x_1$};
\node[right] at (5.5,-1.5) {$x_0$};
\node[circle, fill=black, scale=.2] at (5.5,1.5) {};
\node[circle, fill=black, scale=.2] at (5.5,-1.5) {};

\node[right] at (6,-0.5) {$y$};
\node[circle, fill=black, scale=.2] at (6,-0.5) {};

\node[circle, fill=black, scale=.2] at (6,-0.5) {};
\draw  (5,2) rectangle (7.5,-2);

\node at (7,1.5) {$\underline\Gamma^4$};
\end{scope}

\draw[->] (-5.5,0) -- node[above, midway]{Breaking} (-3.5,0);

\draw [decorate, decoration={brace, amplitude=10pt,mirror,raise=4pt},yshift=0pt]
(2,-2.5) -- (7.5,-2.5) node [black,midway,below, yshift=-1em,text width=5cm] {
This breaking cannot occur for geometric reasons.};

\draw [decorate, decoration={brace, amplitude=10pt,mirror,raise=4pt},yshift=0pt]
(-3,-2.5) -- (1,-2.5) node [black,midway,below, yshift=-1em,text width=5cm] {
This breaking cannot occur for index reasons.};
\end{tikzpicture} }
         \caption{Possible breakings of a pearly flow-line between $x_0$ and $x_1$ which has a minimal energy disk.}
         \label{fig:breakings}
         \end{subfigure}
         \caption{Possible breakings of flow trees between $x_1$ and $x_0$ with a single marked point.}
     \end{figure}

     From this we conclude that $\mathcal P^0_{\mathbb I^{\leq 1}_{*; x_0}}$ is regular for broken types in $\mathbb I^{\leq 1}_{*; x_0}$ as well. 
    
    \begin{definition}
        Let $\mathbb I$ be a collection of types, and $\mathcal P^0_{\mathbb I}$ be a perturbation system defined for combinatorial types in $\mathbb I$. 
       We say that  $\mathcal P_{\mathbb J}$ weakly agrees with $\mathcal P^0_{\mathbb I}$  if for every $\underline \Gamma\in \mathbb I\cap \mathbb J$ with $\ind(\underline \Gamma_P)\leq 0$, the perturbation data agree in a small neighborhood of the labelled combinatorial types contained in these moduli spaces. That is, we require that
       \[\underline C_P(\mathcal M_{\mathcal P^0_{\mathbb I}}(X, L, D, h,\underline \Gamma_P))=\underline C_P(\mathcal M_{\mathcal P_{\mathbb J}}(X, L, D, h,\underline \Gamma_P))\]
       and there exists an open neighborhood $ U\supset \underline C_p(\mathcal M_{\mathcal P_{\mathbb I}}(X, L, D, h,\underline \Gamma_P))$ such that for all $\underline C_P\in U$ and $x\in \underline C_P$,
       \[\mathcal P_{\mathbb J}(\underline \Gamma_P)(\underline C_P, x)=\mathcal P^0_{\mathbb I}(\underline \Gamma_P)(\underline C_P, x).\]

       We say that $\mathcal P_{\mathbb J}$ weakly extends $\mathcal P^0_{\mathbb I}$ if they weakly agree and $\mathbb I\subset \mathbb J$.
    \end{definition}
    
    \begin{lemma}
        Let $\mathcal P^0_{ \mathbb I^{\leq1}_{*; x_0}}$ be the trivial perturbation on set of combinatorial types from the above discussion. There exists a full regular coherent perturbation system $\mathcal P$ which is a weak extension of $\mathcal P^0_{\mathbb I^{\leq 1}_{*; x_0}}$.
        \label{lemma:geometricflowlines}
    \end{lemma}
    
    The remainder of the section will construct $\mathcal P$. 
    We describe the construction of the perturbation datum in steps. Let $\mathbb I^0_{y; z}$ denote the combinatorial types of flow lines with no disks from $y$ to $z$.
    For any $\underline \Gamma_P\in \bigcup_{y, z\in \Crit(h)}\mathbb I^0_{y;z}$, take $\mathcal P(\underline \Gamma_P)$ to be the trivial perturbation datum. The trivial perturbation is coherent. Since $h$ is Morse, $\mathcal P(\underline \Gamma_P)$ is regular.
    The choices made for $\mathcal P$ so far agree weakly with $\mathcal P^0_{\mathbb I^{\leq 1}_{*, x_1}}$.
    
    We now pick perturbation datum $\mathcal P(\underline \Gamma_P)$ for $\underline \Gamma_P\in \bigcup_{z\in \Crit(h)}\mathbb I^1_{z}$, the types of index 0 with 1 internal marked point and 1 output. 
    \begin{claim}
        Perturbation datum $\mathcal P(\underline \Gamma_P)$ can be picked for pre-broken combinatorial types $\underline \Gamma_P\in \mathbb I^1_{z}$ in such a way that it is regular, weakly agrees with $\mathcal P^0_{\mathbb I^{\leq 1}_{*, x_1}}$ on the combinatorial types $\underline \Gamma\in \mathbb I^1_{x_0}$, and is coherent with previously made choices over $\bigcup_{y, z\in \Crit(h)}\mathbb I_{y;z}$.   
    \end{claim}
    \begin{proof}
        The definition of the perturbation depends on the number of pre-breaking points (i.e. critical values of $h$) on the flow-line between the boundary of the disk $u_{ex}$ and critical point $x_0$. We discuss the case where there are 0 or 1 pre-breaking points; the remaining cases are similar.
        \begin{itemize}
            \item Suppose that the single $\mathcal P^0$-pseudoholomorphic curve of type $\underline \Gamma_P\in \mathbb I^1_{x_0}$ from \cref{fig:m0breakings} has no pre-breaking. Then $\{\mathbb I^1_{x_0}\}$ is downward closed, and we can choose perturbation datum for this type freely. Then choose perturbation datum for the rest of $\bigcup_{z}\mathbb I_z^1$.
            \item Suppose that the single  $\mathcal P^0$-pseudoholomorphic curve of type $\underline \Gamma_P\in \mathbb I^1_{x_0}$ has a single pre-breaking, let $\underline C_P$ be the pre-broken disk domain.
            We assume that this pre-breaking has distance of at least 1/4 from the disk (so that the coherence condition for forgetting pre-breakings does not hold near this pre-broken combinatorial type). 
            Choose $\mathcal P(\underline \Gamma'_{P'})$ for all $\underline \Gamma'_{P'}\in \bigcup_{z\in \Crit(h)}\mathbb I^1_{z}$ with no pre-breakings. We can choose this perturbation to be as close to $\mathcal P^0$ as desired.
            Having chosen these $\mathcal P(\underline \Gamma'_{P'})$, we now define $\mathcal P(\underline \Gamma_P)$. The moduli space $\mathcal M(\underline \Gamma_P)$ is a 1-dimensional ray that measures the distance of the pre-broken label to the marked point on the disk $u_{ex}$.
            
            Let $U\subset \mathcal M(\underline \Gamma_P)$ be a small neighborhood of $\underline C_P$. 
            Choose $\mathcal P(\underline \Gamma_P)$ which interpolates between  $\mathcal P^0(\underline \Gamma_P)$ over the neighborhood $U$ and the previously chosen perturbations $\mathcal P_{\underline \Gamma'_{P'}}$ as the distance of the pre-breaking goes to infinity. 
            Since 
            \begin{itemize}
                \item $\mathcal P(\underline \Gamma'_{P'})$ was chosen close to $\mathcal P^0$, 
                \item the  moduli spaces of $\mathcal P^0$- curves with domain in the complement of $U$ is empty, and 
                \item emptiness is preserved under small changes of perturbation datum,
            \end{itemize}
            we have that the moduli of $\mathcal P(\underline \Gamma_P)$-curves with domain belonging to the complement of $U$ is empty. So $\mathcal P(\underline \Gamma_P)$ agrees with $\mathcal P^0(\underline \Gamma_P)$ in a neighborhood of the domains which are represented by pseudoholomorphic treed disks, which is to say that $\mathcal P(\underline \Gamma)$ weakly agrees with $\mathcal P^0(\underline \Gamma_P)$.
        \end{itemize}
        The proof is similar for examples where there are several pre-breakings on the domain $\underline C_P$ representing the single $\mathcal P^0$-pseudoholomorphic curve of type $\mathbb I^1_{x_0}$.
    \end{proof}
    We now handle the types of Morse flow trees. 
    Denote by $\mathbb I^0_{x,y;z}$ the moduli space of Morse flow trees with inputs $x, y$ and output $z$. We pick perturbation datum $\mathcal P$ for types in $\bigcup_{x,y,z}\mathbb I^0_{x,y;z}$. The only coherence conditions we need to worry about are coherence conditions from Morse flow lines; in particular regular and coherent $\mathcal P$ can be chosen to weakly extend $\mathcal P^0$ on $\bigcup_y \mathbb I_{x_1, y; x_0}$. Here, obtaining a weak agreement is easy: any sufficiently small choice of perturbation $\mathcal P$ will do the trick, as $\mathcal M_{\mathcal P^0}(X, L, D, \underline \Gamma)=\emptyset$ for $\underline \Gamma\in \bigcup_y \mathbb I_{x_1, y; x_0}$. The same argument holds for types $\underline \Gamma\in \bigcup_y \mathbb I_{y, x_1; x_0}$.

    We now extend $\mathcal P$ to $\mathbb I^1_{x_1;x_0}$, the set of treed disks with input on $x_1$, output on $x_0$, and 1 interior marked point. This contains three combinatorial types which have expected dimension 0, which are highlighted in \cref{fig:remainingtypes}. On the moduli spaces of types $\underline \Gamma^2, \underline \Gamma^3$, we choose perturbation datum which satisfies:
\begin{itemize}
\item  the coherence conditions for $\mathcal P$ as we go to broken configurations and
\item agrees with $\mathcal P^0$ as we go towards the ghost-bubbled configuration which forms a common boundary with $\underline \Gamma^1$ (which is also forced by coherence).
\end{itemize}
 We note that if previous choices for $\mathcal P$ are chosen sufficiently small, $\mathcal M_{\mathcal P}(X, L, D,\underline \Gamma^i)$ will remain empty for $i\in 2,3$, as the emptiness of moduli spaces is an open condition in the space of perturbations.
    On $\underline \Gamma^1$, we take $\mathcal P(\underline \Gamma^1)=\mathcal P^0(\underline \Gamma^1)$.
    \begin{figure}
		\centering
        \begin{tikzpicture}[scale=.7]

    \begin{scope}[shift={(-10.5,-2)}]
    \draw[fill=red!20]  (3,0.5) rectangle (1,4.5);
    \draw[fill=gray!20]  (2,2.5) ellipse (0.5 and 0.5);
    \node at (2,2.5) {$\times$};
    \draw (2,2) -- (2,1);
    \draw (2,4) -- (2,3);
    
    \draw (2,2) -- (2,1.5);
    \node[right] at (2,4) {$x_1$};
    \node[right] at (2,1) {$x_0$};
    \node[circle, fill=black, scale=.2] at (2,4) {};
    \node[circle, fill=black, scale=.2] at (2,1) {};
    
    \node at (2,0) {$\underline\Gamma^1$};

    \end{scope}

    \begin{scope}[shift={(-12,0.5)}]
    \draw[fill=gray!20]  (6,0) ellipse (0.5 and 0.5);
    \node at (6,0) {$\times$};
    \draw (5.5,0) -- (5.5,-1) -- (5.5,-1.5);
    \draw (5.5,1.5) -- (5.5,-1);
    \node[right] at (5.5,1.5) {$x_1$};
    \node[right] at (5.5,-1.5) {$x_0$};
    \node[circle, fill=black, scale=.2] at (5.5,1.5) {};
    \node[circle, fill=black, scale=.2] at (5.5,-1.5) {};

    \draw  (5,2) rectangle (7,-2);

    \end{scope}
    
    \begin{scope}[shift={(-19,-1)}]
    \draw[fill=red!20]  (6.5,-0.5) rectangle (4.5,3.5);
    \draw[fill=gray!20]  (5.2,2.05) ellipse (0.5 and 0.5);
    \node at (5.2,2.05) {$\times$};
    \draw (5.2,1.55) -- (6,0.5) -- (6,0);
    \draw (6,3) -- (6,0.5);
    \node[left] at (6,3) {$x_1$};
    \node[circle, fill=black, scale=.2] at (6,3) {};
    \node[left] at (6,0) {$x_0$};
    \node[circle, fill=black, scale=.2] at (6,0) {};

    \node at (5,-1) {$\underline\Gamma^2$};
    
    \end{scope}
    
    \begin{scope}[shift={(-21.5,-1)}]
    \draw[fill=gray!20]  (5.25,2) ellipse (0.5 and 0.5);
    \node at (5.25,2) {$\times$};
    \draw (5.25,1.5) -- (6,0.5) -- (6,0);
    \draw (6,3) -- (6,0.5);
    \node[left] at (6,3) {$x_1$};
    \node[circle, fill=black, scale=.2] at (6,3) {};
    \node[left] at (6,0) {$x_0$};
    \node[circle, fill=black, scale=.2] at (6,0) {};

    \draw  (6.5,-0.5) rectangle (4.5,3.5);

    \end{scope}
    \begin{scope}[shift={(-16.5,-1)}]
    \draw[fill=gray!20]  (5.5,1.5) ellipse (0.5 and 0.5);
    \node at (5.5,1.5) {$\times$};
    \draw (6,1.5) -- (6,1.5) -- (6,0);
    \draw (6,3) -- (6,1.5);
    \node[left] at (6,3) {$x_1$};
    \node[circle, fill=black, scale=.2] at (6,3) {};
    \node[left] at (6,0) {$x_0$};
    \node[circle, fill=black, scale=.2] at (6,0) {};

    \draw  (6.5,-0.5) rectangle (4.5,3.5);

    \end{scope}
    
    \begin{scope}[shift={(-7,0.5)}]
    \draw[fill=gray!20]  (6.25,0.5) ellipse (0.5 and 0.5);
    \node at (6.25,0.5) {$\times$};
    \draw (6.25,0) -- (5.5,-1) -- (5.5,-1.5);
    \draw (5.5,1.5) -- (5.5,-1);
    \node[right] at (5.5,1.5) {$x_1$};
    \node[right] at (5.5,-1.5) {$x_0$};
    \node[circle, fill=black, scale=.2] at (5.5,1.5) {};
    \node[circle, fill=black, scale=.2] at (5.5,-1.5) {};

    \draw  (5,2) rectangle (7,-2);

    \end{scope}
    
    \begin{scope}[shift={(-9.5,0.5)}]
    \draw[fill=red!20]  (5,2) rectangle (7,-2);
    \draw[fill=gray!20]  (6.25,0.5) ellipse (0.5 and 0.5);
    \node at (6.25,0.5) {$\times$};
    \draw (6.25,0) -- (5.5,-1) -- (5.5,-1.5);
    \draw (5.5,1.5) -- (5.5,-1);
    \node[right] at (5.5,1.5) {$x_1$};
    \node[right] at (5.5,-1.5) {$x_0$};
    \node[circle, fill=black, scale=.2] at (5.5,1.5) {};
    \node[circle, fill=black, scale=.2] at (5.5,-1.5) {};

    \node at (6.5,-2.5) {$\underline\Gamma^3$};
    \end{scope}

    \node[circle, fill=black, scale=.2] at (-15.8928,0) {};
    \node[circle, fill=black, scale=.2] at (-1.1786,0) {};
    \node[left] at (-15.8928,0) {$y$};
    \node[right] at (-1.1786,0) {$y$};
    \draw (-16,-2.5) -- (-1,-2.5);
    \node[circle, fill=black, scale=.2] at (-16,-2.5) {};
    \node[circle, fill=black, scale=.2] at (-11,-2.5) {};
    \node[circle, fill=black, scale=.2] at (-6,-2.5) {};
    \node[circle, fill=black, scale=.2] at (-1,-2.5) {};
    \node at (-8.5,4.5) {Perturbations already prescribed, with empty moduli spaces};
\draw[->] (-8.5,4) .. controls (-8.5,3.5) and (-16,3) .. (-16,2.5);
\draw[->] (-8.5,4) .. controls (-8.5,3.5) and (-11,3) .. (-11,2.5);
\draw[->] (-8.5,4) .. controls (-8.5,3.5) and (-6,3) .. (-6,2.5);
\draw[->] (-8.5,4) .. controls (-8.5,3.5) and (-1,3) .. (-1,2.5);
    \node at (-8.5,-4) {Equip these with a small perturbation of $\mathcal P^0$ so that moduli space remains empty};
\draw[->] (-8.5,-3.5) .. controls (-8.5,-3) and (-13.5,-2.5) .. (-13.5,-1.5);
\draw[->] (-8.5,-3.5) .. controls (-8.5,-3) and (-3.5,-2.5) .. (-3.5,-1.5);
\node at (-16.5,3) {$\mathcal P$};
\node at (-11.5,3) {$\mathcal P^0$};
\node at (-5.5,3) {$\mathcal P^0$};
\node at (-0.5,3) {$\mathcal P$};
\end{tikzpicture}         \caption{The three remaining tree types}
        \label{fig:remainingtypes}
    \end{figure}
    
    We have now defined regular and coherent $\mathcal P$ on a downward closed set of types containing $\mathbb I^{\leq 1}_{*; x_0}$. By construction $\mathcal P$ is a weak extension of $\mathcal P^0_{\mathbb I^{\leq 1}_{*; x_0}}$. By \cite[Theorem 4.19]{charest2017floer}, we can extend this to a regular and coherent perturbation on all types.  \section{Stabilizing divisors and pearly model for Lagrangian cobordisms}
	\label{app:cobordismstabilizingdivisor}
In this appendix, we drop the requirement that $\omega(\pi_2(X))<0$. For simplicity of exposition, we assume that $X$ is compact,  $H_1(L)$ is torsion free, and that $[\omega]\in H^2(X, \ZZ)$.
\subsection{Stabilizing Divisors: Background and Summary}
\label{subsec:cobordismstablebackground}
We use $J_\tau(X, \omega)$ to denote the space of $\omega$-tame almost complex structures. 
A symplectic divisor is a symplectic hypersurface $D\subset X$. If $[D]$ is Poincar\'e dual to $k\cdot [\omega]$, we say that the degree of $D$ is $k$. We say that $J$ is adapted to $D$ if $J(TD)=TD$.

    A \emph{weakly stabilizing divisor} \cite[Definition 3.8]{charest2015floer} for a Lagrangian $L\subset X$ is a symplectic divisor $D\subset X$ disjoint from $L$ and for which there exists a $J_D\in \mathcal J_\tau(X, \omega)$ adapted to $D$ so that every $J_D$ holomorphic disk or sphere intersects $D$ in at least one point. 

    A divisor is \emph{of sufficiently large degree} for an almost complex structure $J_D$ and a Lagrangian $L$ if for all $J_D$-holomorphic spheres $u_{S^2}$ and disks $u_{D^2}$ with boundary on $L$,
    \begin{align*}
        PD([D])([u_{S^2}]) \geq &2 c_1(X)([u_{S^2}])+\dim(X)+1\\
        PD([D])([u_{D^2}])\geq &1.
    \end{align*}

\begin{lemma*}[Lemma 3.9 \cite{charest2017floer}]
    Let $L$ be a Lagrangian submanifold, and $J$ be an almost complex structure. There exists a constant $k_m$ so that for every $0<\theta<1$, there exists a $k_\theta>0$ such that for all $k>k_\theta$ we can find a $\theta$-approximately $J$-holomorphic divisor $D$ of degree $k_mk$  which is of sufficiently large degree for $L$. 
\end{lemma*}
Note that the complex structure $J_D$ for which $D$ is weakly stabilizing will likely not be the structure $J$ we start with.
To use weakly stabilizing divisors for the purposes of constructing open Gromov-Witten invariants, one needs the divisor to transversely intersect $J_D$-holomorphic curves; additionally, one would like an open set worth of $J_D$'s to use as domain-dependent perturbations. 

\begin{definition}[Definition 4.24 \cite{charest2015floer}]
    For $E>0$, an almost complex structure $J_D\in \mathcal J_\tau (X)$ is $E$-stabilized by $D$ if and only if whenever $u: (\Sigma, \partial \Sigma)\to (X, L)$ is a non-constant $J_D$ holomorphic curve (where $\Sigma\in\{D^2, S^2\})$ of energy less than $E$ we have
    \begin{itemize}
        \item (Non-constant spheres) There are no spheres $u: \Sigma\to X$, whose images are contained in $D$ and
        \item (Sufficient intersections) Each sphere (resp. disk) has at least three (resp. one) intersection points with the divisor $D$. 
    \end{itemize}
\end{definition}

\begin{lemma*}[Lemma 4.25 \cite{charest2015floer}]
    For $\theta$ sufficiently small, suppose that $D$ has sufficiently large degree for every $J_D$ which is $\theta$-close to $J$. For each energy $E>0$, the set of $E$-stabilized tame almost complex structures which are $D$-adapted and $\theta$-close to $J$ is open and dense (in the set of tame almost complex structures which are $D$-adapted and $\theta$-close to $J$).
\end{lemma*}

\subsubsection{Comparison to previous work}
We adapt the following lemmas to the setting of Lagrangian cobordisms:
\begin{itemize}
    \item \cite[Lemma 3.9]{charest2017floer} and \cite[Lemma 8.11]{cieliebak2007symplectic},which show that there exist weakly-stabilizing divisors for pseudoholomorphic disks and spheres; and 
    \item \cite[Lemma 4.25]{charest2015floer}, itself an extension of \cite[Proposition 8.14, Corollary 8.20]{cieliebak2007symplectic} which shows that for a fixed energy bound we can find a dense set of stabilized almost complex structures.
\end{itemize}
 Let $K: L^+\rightsquigarrow L^-$ be a Lagrangian cobordism. Denote by $V\subset \CC$ the open set with compact closure so that $K|_{\pi_\CC^{-1}(\CC\setminus V)}$ is contains only the ends of the Lagrangian cobordism. 

\begin{lemma}[Weakly-Stabilizing Divisors for Lagrangian cobordisms]
Let $K: L^+\rightsquigarrow L^-$ be a Lagrangian cobordism.
Pick an almost complex structure $J_X\times \jmath$ for $X\times \CC$.
There exists a constant $k_m$ so that for every $0<\theta<1$, there exists a $k_\theta>0$ such that for all $k>k_\theta$ we can find a $\theta$-approximately $J$-holomorphic divisor $D$ of degree $k_mk$ which is of sufficiently large degree for  $K$. 
Furthermore, this divisor can be chosen so that $D|_{\pi_\CC^{-1}(\CC\setminus V)}=D_X\times (\CC\setminus V)$ for some divisor $D_X\subset X$.
\label{lem:weaklystabilizedcobordism}
\end{lemma}
For this divisor, we can find $J_D$ which are weakly stabilized by $D$  and belongs to
\[\mathcal J_{\tau, V}(X\times \CC, \omega+\omega_\CC):=\left\{J\in J_{\tau}(X\times \CC, \omega)\;\middle|\; \parbox{5cm}{ $\exists J_z\in J_{\tau}(X, \omega)$ with   $J=J_z\times \jmath $ outside $\pi_\CC^{-1}(V)$}\right\},\]
the set of almost complex structures on $X\times \CC$ which are split and constant outside of a given compact subset $V$.
There exists a restriction map $\res:\mathcal J_{\tau, V}(X\times \CC, \omega + \omega_\CC)\to \mathcal J_{\tau}(X, \omega)$ by restricting to a fiber $X\times \{z\}$ with $z\not\in V$.

\begin{lemma}[Density of $E$-stabilized almost complex structures]
    Given $K: L^+\rightsquigarrow L^-$, choose $\theta, D$ as above. Then for $E>0$ there exists an open and dense subset of almost complex structures 
    \[\mathcal J^*_V(X\times \CC, D, J, \theta, E)\subset \mathcal J_{\tau, V}(X\times \CC, D, J, \theta)\]
    of almost complex structures which are $E$-stabilized by $D$.
    \label{lem:estabilizedcobordism}
\end{lemma}

\subsection{Background: divisors in the complement of a Lagrangian.}
We first recall the construction of weakly-stabilizing divisors for Lagrangian submanifolds $L\subset X$ where $X$ is compact. Portions of this algorithm will come into play when we construct weakly stabilizing divisors for Lagrangian cobordisms $K\subset X\times \CC$.

The primary method (employed in \cite{donaldson1996sympsubmanifold,auroux1997asymptotically}) for constructing symplectic divisors is to present them as the zero sets of some line bundle. For $J\in J_{\tau}(X, \omega)$, let $E_X\to X$ be a Hermitian line bundle with connection $\nabla^E$ whose curvature is $i\omega$. The zero sets of transverse sections of $E^k_X$ will be Poincar\'{e} dual to $k[\omega]$.
When we have a $J$-holomorphic section $s: X\to E$ which is transverse to the zero section, the zero set $s^{-1}(0)$ is a symplectic divisor. 
The condition of being $J$-holomorphic can be weakened substantially while still preserving the symplecticity of the divisor. A common way to weaken this construction is to consider a sequence of sections $s_k: X\to E^k$ which are \emph{asymptotically holomorphic} and uniformly transverse to 0; then for sufficiently large $k$, the sections $s_k^{-1}(0)$ will give symplectic divisors of degree $k$. 

\cite[Theorem 3]{auroux1997asymptotically} provides a tool for constructing asymptotically holomorphic sections.
\begin{definition}
    Let $V\subset X$ be a subset. 
    We say that sections $s_k: X\to E^k$ \emph{fall off from $V$} if there exists constants $C, \lambda>0$ so that $|s_k(x)|< C\exp(-\lambda d(y, V)^2)$, where $d$ is the metric induced by $\omega, J$.
    
    A section is \emph{concentrated at $V$} if there exists a constant $c$ so that $|s_k(x)|>c$ for all $x\in V$, and the sections $s_k$ fall off from $V$.
\end{definition}
Sections which are concentrated along a subset $V$ can be perturbed to make them transverse to the zero section without losing transversality over $V$. This is because if a section is concentrated at $V$, it is ``highly transverse'' to the zero section (in the sense that it is very non-zero!).
This idea can be extended to sections which intersect the zero section.
\begin{definition}[Definition 17 of \cite{donaldson1996sympsubmanifold}]
    We say that sections  $s_k: X\to E^k$ are $\eta$-transverse to $0$ over $V$ if for all $x\in V$, whenever $|s_k(x)|<\eta$, the covariant derivative $\nabla s_k(x): T_xX\to E_x^k$ is surjective with bound $|\nabla s_k(x)|>k^{1/2}\cdot \eta$.
    The metric we use here is the metric induced by $\omega$ and $J$. 

    Sections $s_{k}: X\to E^k$ have a \emph{neighborhood of $\eta$-transversality} over $V'$, if $s_k$ is $\eta$-transverse over all $x$ with $d(x, V')<4k^{-1/3}$. We say that a section is \emph{transversely extendible over $U$} if $U$ has compact closure, and $s_{k}$ has a neighborhood of $\eta$-transversality over $\partial U$. 

    Given $U$ a set, let $U^{=}:=\{x\in U \;|\; d(x, \partial U)>4k^{-1/3}$.
\end{definition}
From the definition, it is immediate that if $s_k: X\to E^k$ is concentrated at $V$, then it is transversely extendible over $X=U\setminus V$. The notion of transversely extendible is based on the following theorem:

\begin{theorem}[Adapted from Theorem 3\cite{auroux1997asymptotically}]
    Fix $\epsilon>0$. Given $s_k: X\to E$ asymptotically holomorphic sections and any open subset $U\subset X$ with compact closure, there exists asymptotically holomorphic sections $\tilde s_{k}: X\to E^k$ and $\tilde \eta>0$ so that 
    \begin{itemize}
        \item The section $\tilde s_{k}$ are asymptotically holomorphic ,
        \item The sections are $\tilde \eta$-transverse over $U^=$,
        \item $|\tilde s_{k}-s_k|<\epsilon$ and $|\nabla \tilde s_{k}- \nabla s_k|<k^{1/2}\epsilon$, and
        \item $s_k(x)=\tilde s_k(x)$ on the complement of $U$. 
    \end{itemize}
    \label{thm:donaldsondivisor}
\end{theorem}
An immediate corollary employed in \cite{auroux2001symplectic} is that whenever $X$ is compact, $s_{k, V}: X\to E^k$ asymptotically holomorphic sections concentrated on $V$, then there exists $\eta$-transverse holomorphic sections which are non-vanishing on $V$. A small modification allows us to replace ``concentrated'' with ``transversely extendible in the complement''.
\begin{corollary}
    Suppose that $s_{k, out}: X\to E^k$ is asymptotically holomorphic, transverse to zero over $X\setminus U$, and transversely extendible over $U$. Additionally assume that $s_{k, out}^{-1}(0)|_{X\setminus U}$ is a symplectic divisor over $X\setminus U$ for $k\gg 0$.
    There exists $\tilde \eta>0$ and asymptotically holomorphic sections $\tilde s_{k}: X\to \CC$, agreeing with $s_k$ over $X\setminus U$ and  $\tilde{\tilde {\eta}}$ transverse to zero over $U$. In particular, $\tilde s_k^{-1}(0)$ is a symplectic divisor for $k \gg 0$.
    \label{cor:extensionofsections}
\end{corollary}
\begin{proof}
    By definition of transverse extendibility, $s_{k,out}$ has a neighborhood of $\eta$-transversality along $\partial(U)$. This implies that  $s_{k,out}$ is $\eta$-transverse in $U\setminus U^=$.

    Pick $\epsilon$ small enough so that $\epsilon<\eta/2$.  By \cref{thm:donaldsondivisor}, we can construct sections $\tilde s_k$ with $\epsilon$-specified which are $\tilde \eta$-transverse over all points in $U^=\subset U$. As $s_{k,out}$ is $\eta$-transverse in $U\setminus U^=$, and $|\tilde s_{k}-s_k|<\epsilon$ and $|\nabla \tilde s_{k}- \nabla s_k|<k^{1/2}\epsilon$, we obtain that  $\tilde s_k$ is at least $\eta-\epsilon$ transverse over $U\setminus U^=$. Therefore, $\tilde s_k$ are $\tilde{\tilde{\eta}}=\min(\tilde \eta, \eta/2)$ transverse over $U$. 

    It follows that for $k\gg 0$, $\tilde s^{-1}_k(0)|_{U}$ is a symplectic hypersurface. Since $\tilde s^{-1}(0)|_{X\setminus U}=s^{-1}(0)|_{X\setminus U}$ (which is symplectic for $k\gg 0$ by assumption) we obtain that for $k\gg 0$ $\tilde s^{-1}(0)$ is a symplectic hypersurface.
\end{proof}

We now return to the problem of finding $D$ disjoint from $L$: this amounts to constructing asymptotically $J$-holomorphic and $\eta$-transverse sections $s_k$ such that $s_k(x)\neq 0$ for all $x\in L$. 

\cite[Lemma 3.9]{charest2015floer} observes that these sections must satisfy an additional requirement if we would like $D_k=s_k^{-1}(0)$ to be weakly stabilizing. Suppose that $D_k$ is given as the zero set of sections $s_k: X\to E^k$, which are non-vanishing on $L$.
The sections determine trivializations $\tau_k$ of $E^k\to L$; the connection 1-form written in this trivialization determines a class $\alpha_{\tau_k}\in H^1(L, \RR)$. We can compute the intersection number of $D_k$ with a class of disk $u:(D^2, \partial D^2)\to (X, L)$ by:
\begin{equation}
    [u]\cdot [D_k]= k\int_{D^2} u^*\omega - \int_{[\partial D^2]} u^* \alpha_{\tau_k}.\label{eq:boundedconnectionforms}
\end{equation}
The only way we can hope for this to be positive is if the connection 1-forms $\alpha_{\tau_k}$ are bounded; we call such a selection of trivializations \emph{bounded}. \cite[Lemma 3.9]{charest2015floer} proves that for fixed $L$, boundedness of $\alpha_{\tau_k}$ is a sufficient condition for the divisors $D_k$ to be of sufficiently large degree for large enough $k$.  

The bounded curvature form requirement following \cref{eq:boundedconnectionforms} can be accommodated in the construction of \cite{auroux2001symplectic}, which we now recall. The construction of $s_{k, L}$ starts by finding a constant $C_L$ and picking trivializations $\tau_k: L\to E^k$ with 
\begin{align*}
    |\tau_k(x)|=1 && |\nabla \tau_k(x)|_g< C_L
\end{align*}
These trivializations can additionally be chosen so that their connection 1-forms $\alpha_{\tau_k}$ are bounded.
Associated to $p\in L$ a point, \cite{auroux2001symplectic} constructs asymptotically holomorphic and uniformly bounded sections 
    \[s_{k, p, L}(x):=\frac{\tau_k(p)}{|s_{k, p}(p)|}s_{k, p}(x)\]
where $s_{k, p}(x)$ is asymptotically holomorphic section concentrated at a point $p$.

Consider a \emph{finite} set of points $P(k)\subset K$ with the property that the radius $\frac{1}{\sqrt{k}}$-balls centered at $x\in P(k)$ cover $K\cap \pi_\CC^{-1}(V)$, and the radius $\frac{1}{3\sqrt{k}}$-balls around the $x\in P(k)$ are disjoint from another and $K\cap \pi_\CC^{-1}(\CC\setminus V)$. By construction , the arguments of $s_{k, p, L}(x)$ do not differ by much (\cref{eq:arguments}). Therefore
\[s_{k, L}:=\sum_{p\in P(k)}s_{k, p, L}(x)\]
is an asymptotically holomorphic section concentrated along $L$. To obtain a section which is $\eta$-transverse everywhere (giving us symplectic divisors) we apply \cref{cor:extensionofsections} to perturb by a small amount in the complement of $L$.

\subsection{Construction of Weakly Stabilizing Divisors for Lagrangian Cobordisms}
We adapt the constructions above for Lagrangian cobordisms $K\subset X\times \CC$. Again, we assume that $X$ is compact,  $H_1(K)$ is torsion free, and that $[\omega]\in H^2(X, \ZZ)$.
\subsubsection{Toy case: cobordism with empty ends}
We start with a toy case: let $K\subset X\times \CC$ be a compact Lagrangian submanifold, and suppose that $X$ is compact. We construct a weakly stabilizing divisor for $K$.

Take $V:=\{z\;|\; |z|< R\}\subset \CC$ a compact subset so that the ends of $K$ are disjoint from $V$. Additionally, choose $J=J_X\times \jmath$ a split almost complex structure.
We now sketch how to construct a symplectic divisor in the complement of $K$. 
The approach is similar to the one before: consider the bundle  $E=(E_X \boxtimes E_\CC)$, with connection $\nabla^{E_X}\boxtimes \nabla^{E_\CC}$ whose curvature form is $\omega_X+\omega_\CC$. 
We take trivializations $\tau_k: K\to E^k$ which satisfy the connection 1-form bound following \cref{eq:boundedconnectionforms}. By application of \cite{auroux2001symplectic} we can construct $\eta$-transverse asymptotically $J$-holomorphic section $s_{k, K}: X\times \CC\to E$ which is concentrated along $K$. 

We cannot from here immediately proceed using the construction of \cite{auroux2001symplectic}, as we need to perturb over a non-compact set to achieve transversality. However, we can find a perturbation $s_{k,\text{out}}: X\times \CC\to E^k$ which is transverse (but not $\eta$-transverse!) and asymptotically holomorphic outside of the region $V$. Consider the section $s_{k, R}:\CC\to E_\CC$ modeled after the section which is asymptotically holomorphic and concentrated on $S^1_R$ from \cite[Pg 746]{auroux2001symplectic}; this can be explicitly written in coordinates as 
\[ s_{k,R}(z)=\rho_k \exp\left(\frac{ k (R-|z|)^2}{2}\right).\]
where $\rho_k: \CC\to \RR$ is a bump function equal to 1 outside a small neighborhood of the origin. 
The $s_{k, R}(z)$ define asymptotically holomorphic sections which are non-vanishing outside of $|z|>R$. Furthermore, $s_{k, R}(z)$ is concentrated along the boundary of $S^1_R$.

Now pick $s_{k, X}: X\to E$ a section which is asymptotically holomorphic and $\eta$-transverse. 
This gives us the section  $s_{k,\text{out}}:=s_{k,X}\boxtimes s_{k, R}: X\times \CC\to E^k$. 
For $(x, z)\in X\times \CC$ with $||z|-R|<k^{-1/3}$, we have 
\begin{align*}
    |s_{k, \text{out}}|+|\nabla^E s_{k, \text{out}}| =& |s_{k,X}|\cdot |s_{k, R}|+ |\nabla^{E_X}s_{k,X}|\cdot |s_{k, R}|+ |s_{k,X}|\cdot |\nabla^{E_\CC}s_{k, R}|\\
    >& \left(|s_{k, X}|+|\nabla^{E_X} s_{k, X}|\right)|s_{k, R}|>\frac{\eta}{3}.
\end{align*}
So, $s_{k, \text{out}}$ is $\frac{\eta}{3}$-transverse in a neighborhood of $|z|=R$.

This means  $s_{k, \text{out}}+ s_{k, K}$ is $\frac{\eta}{3}$ transverse in a neighborhood of the boundary of $\partial(X\setminus(K\cup \pi_\CC^{-1}(\CC\setminus V))$ --- which means that it is transversely extendible over the complement of $K\cup \pi_\CC^{-1}(\CC\setminus V)$. Since the zero locus of $s_{k,\text{out}}+ s_{k, K}$ restricted to $K\cup \pi_\CC^{-1}(\CC\setminus V)$ is symplectic for $k\gg 0$, by \cref{cor:extensionofsections} we obtain a section $\tilde s_k: X\times \CC\to E^k$ whose zero sets $s_{k}^{-1}(0)$ are symplectic for $k \gg 0$ and agree with the zero set of $s_{k,\text{out}}+ s_{k, K}$ over $K\cup \pi_\CC^{-1}(\CC\setminus V)$. That portion is 
\[s_{k}^{-1}(0)|_{\pi_\CC^{-1}(\CC\setminus V)}= s_{k,\text{out}}^{-1}(0)|_{\pi_\CC^{-1}(\CC\setminus V)}= s_{k, X}^{-1}(0)\times (\CC\setminus V).\]

\subsubsection{Full case: Construction of Donaldson divisors for Lagrangian Cobordisms}
\label{lem:hypersurfaceexistence}
    We first note that there is a flat trivialization of $\tau_{k, \RR}: \RR\to E^k_\CC$ for the Lagrangian $\RR\subset \CC$. 
    Pick an almost complex structure $J_X\times \jmath$ for $X\times \CC$.
    As in the construction from \cite{auroux2001symplectic}, we take trivializations $\tau_{k,K}: K\to  E^k$; we ask that they satisfy the following additional properties:
    \begin{itemize}
        \item  the sections satisfy the bounds from \cite[Lemma 2]{auroux2001symplectic}.
        \item  There exist choices of trivializations $\tau^\pm_k: L^\pm\to E_X^k$ so that $\tau_k$ splits as $\tau^\pm_k\times \tau_\RR$ when restricted to the ends of the cobordism $L^+\times [t^+, \infty), L^-\times (-\infty, t^-]$.
        \item  Additionally, the trivializations $\tau_+, \tau_-$ chosen above satisfy the property that at each $x\in L^+\cap L^-$, we have $\tau_+(x)=\tau_-(x)$. 
    \end{itemize}
    The last condition can be achieved with the argument of \cite[Lemma 3.11 (b)]{charest2015floer}.

    The divisor we construct will be the zero set of a transverse $\theta$-approximately holomorphic section $s_k: X\times \CC\to E^k$ which is non-vanishing on $K$.
    The restriction of $s_k|_K: K\to E^k$ provides a trivialization homotopic the sections $\tau_{k, K}$. Provided that these have bounded connection forms (following the discussion of  \cref{eq:boundedconnectionforms}) the resulting divisors we construct will become weakly stabilizing when $k\gg 0$ \cite[Lemma 3.9]{charest2017floer}. 

    Using the construction of \cite{auroux2001symplectic} there exist constants $C, \eta>0$ and asymptotically $J_X$-holomorphic $\eta$-transverse sections $s_{k, X}: X\to E^k$ which have the property that $s_{k, X}(x)>C$ for all $x\in L^+\cup L^-$; the section arises as a perturbation of a section concentrated on $L^+\cup L^-$. For any given $\theta>0$, we can choose the perturbation sufficiently small so that for large $k$ and any $x\in L^+\cup L^-$ the argument of the section $s_{k, X}$ is $\theta$ approximately for the argument of $\tau_k^\pm$ (\cite[p. 746]{auroux2001symplectic}) i.e.
    \begin{equation}\sup_{x\in L^\pm}\left|\arg\left(\frac{s_{k, X}(x)}{\tau_k^\pm(x)}\right)\right|<\theta
        \label{eq:arguments}
    \end{equation}
    
    Take $\theta<\frac{\pi}{8}$ in the above construction, and define the section $s_{k, X\times \CC, \text{out}}:=s_{k, X}\boxtimes s_{k,R}: X\times \CC\to E^k$. The sections asymptotically $J_X\times \jmath$ holomorphic, transverse to the zero section outside a non-compact set, and non-vanishing over the ends of the Lagrangian cobordism.
    Additionally, $s_{k, X\times \CC, \text{out}}$ has the property that for sufficiently large $k$, the argument of this section approximates (up to $\theta$) the argument of $\tau_{k, K}$ over $X\times (\CC\setminus V)$.

    It remains to modify this section so that it is non-vanishing over the entirety of $K$ and add further perturbations to obtain $\tilde \eta$-transversality over $\pi_\CC^{-1}(V)$. Consider locally concentrated perturbations whose arguments are approximately determined by $\tau_k$ 
    \[s_{k, p, K}(x,z):=\frac{\tau_k(p)}{|s_{k, p}(p)|}s_{k, p}(x,z).\]
    For each $k$, consider a subset of points $P(k)\subset K$ whose radius $\frac{1}{\sqrt{k}}$ balls cover $K\cap \pi^{-1}_\CC(V)$ and are distance at least $\frac{1}{3\sqrt{k}}$ from each other. \cite{auroux2001symplectic} shows that 
    \[s_{k, K , \text{int}}(x):= \frac{\eta}{6}\sum_{p\in P(k)} s_{k, p, K}(x)\]
    is concentrated at $K\setminus \pi^{-1}(U)$. Furthermore, for any $\theta$ and $k$ sufficiently large,  $\sup_{x \in K\setminus \pi_\CC^{-1}(V)}\left|\arg\left(\frac{s_{k, K, \text{int}}(x)}{\tau_k(x)}\right)\right|<\theta$. 
    Therefore, for sufficiently large $k$, the sections $s_{k, K, \text{int}}(x)$ and $s_{k,X\times \CC, \text{out} }(x)$ have  arguments agreeing up to error of $2\theta = \pi/4$ over all $x\in K\cap \pi_\CC^{-1}(V)$; therefore the sum $|s_{k, K, \text{int}}+ s_{k, X\times \CC, \text{out}}|$ is at least $\eta/6$ over $K\cap \pi^{-1}(V)$. 
    
    In summary: we have constructed a section $s_k:=s_{k, K, \text{int}}+ s_{k, X\times \CC, \text{out}}$ which is $\eta$-transverse over a neighborhood of $\partial(X\times V\setminus K)$, and $ s_k^{-1}(0)$ is a symplectic divisor over a neighborhood of $K\cup \pi^{-1}_\CC(\CC\setminus V)$ for $k\gg 0$. 
    From here we apply \cref{cor:extensionofsections} to construct asymptotically holomorphic sections  $\tilde s_k: X\times \CC\to E^k$ whose zero set is symplectic for $k \gg 0$; as in the toy case, 
    \[\tilde s_k^{-1}(0)|_{X\times (\CC\setminus V)}= s_k^{-1}(0)|_{X\times (\CC\setminus V)}=s_{k, X}^{-1}(0)\times (\CC\setminus V).\]

    It follows from \cite[Lemma 3.9]{charest2017floer} and \cite[Lemma 8.11]{cieliebak2007symplectic}that for $k$ sufficiently large, $\tilde s_k^{-1}(0)$ is a divisor of sufficiently large degree for $K$. This completes the proof of \cref{lem:weaklystabilizedcobordism}.
\subsubsection{Weak stability}

We now show that for $k$ sufficiently large, the Donaldson divisors $D$ constructed as the zero sets of $\tilde s_k: X\times \CC \to E^k$ described in \cref{lem:hypersurfaceexistence} are weakly stabilizing. The argument follows \cite[Section 4.5]{charest2017floer}. 

Given $J\in\mathcal J_{\tau, V}(X\times \CC, \omega+\omega_\CC)$ we use \cref{lem:hypersurfaceexistence} to construct  $D\subset X\times \CC$ which splits in the complement of $\pi^{-1}_\CC(V)$, is $\theta$-approximately holomorphic and has sufficiently large degree $k$. This restricts to fibers $X\times \{z\}, z\not\in K$ to give us divisors $D_X$ which are $\theta$-approximately holomorphic with respect to the almost complex structure $J_X$.

Denote by 
\begin{align*}
    \mathcal J_{\tau}(X, D_X, J_X, \theta):=&\{ J_{X,D_X}\in\mathcal J_{\tau}(X, \omega)\;|\; \|J_X-J_{X, D_X}\|<\theta, J_{X, D_X}(TD_X)=TD_X.\}\\
    \mathcal J_{\tau, V}(X\times \CC, D, J, \theta):=&\{ J_D\in\mathcal J_{\tau, V}(X\times \CC, \omega+\omega_\CC)\;|\; \|J-J_D\|<\theta, J_D(TD)=TD.\}
\end{align*}
to be the almost complex structures which are $\theta$-close and adapted to the divisor $D$. There is a restriction map  $\res: \mathcal J_{\tau, V}(X\times \CC, D, J, \theta)\to  \mathcal J_{\tau}(X, D_X, J_X, \theta)$. The argument of \cite[Lemma 3.9]{charest2017floer} carries over here, so for $\theta>0$ there exists $d_0(\theta)$ so that if $D$ has degree $d_0(\theta)$ then $D$ is sufficiently large (\cite[Definition 4.24]{charest2017floer}) for all almost complex structures which are $\theta$-close to $J$.

\begin{proof}[Proof of \cref{lem:estabilizedcobordism}]
    As in  \cite[Lemma 4.9]{charest2017floer}/\cite[Lemma 4.25]{charest2015floer}, the proof is similar to \cite[Proposition 8.14]{cieliebak2007symplectic}. The argument in \cite[Lemma 4.9]{charest2017floer} shows first shows that  $\mathcal J^*_V(X\times \CC, D, J, \theta, E)$ is open.

    Consider a convergent sequence $J^i$ in the complement of $\mathcal J_{\tau, V}(X\times \CC, D, J, \theta,E)$ where there are $J^i$-holomorphic spheres $u^i$ which are contained in $D$. Then at least one of the two hold:
    \begin{itemize}
        \item There is an infinite subsequence of $J^i$-holomorphic $u^i$ with $\Im(\pi_\CC\circ u^i)\subset V$. Since the images of these spheres are contained in a compact set, we may apply Gromov compactness to show that a limiting subsequence converges to $u$ with $\Im(u)\subset D$. Therefore, the limit of the $J^i$  is in the complement of $\mathcal J^*_V(X\times \CC, D, J, \theta, E)$
        \item There is an infinite subsequence of $J^i$-holomorphic $u^i$ with $\Im(\pi_\CC\circ u^i)\not\subset V$. By open mapping principle, the $\pi_\CC\circ u^i$ are constant, and so we obtain a subsequence of $J^i_X$-holomorphic maps with image $\pi_X\circ u^i$ in $D_X$. The same Gromov compactness argument shows that the limit of the $J^i_X$ is in the complement of $\mathcal J^*(X, D_X, J_X, \theta, E)$. Therefore the limit of the $J^i$ is in the complement of $\mathcal J^*_V(X\times \CC, D, J, \theta, E)$.
    \end{itemize}
    The next step is to show that $\mathcal J^*_V(X\times \CC, D, J, \theta, E)$ is dense. This is done by showing that $\mathcal J^*_V(X\times \CC, D, J, \theta, E)$
    contains $\mathcal J^{reg}_V(X\times \CC, J, \theta, E)$, the set of almost complex structures such that all simple holomorphic curves up to energy $E$ are regular. The argument that these regularizing complex structures are $E$-stabilized is identical to the proof given in \cite[Lemma 4.9]{charest2017floer} which follows the ideas of \cite[Proposition 8.11]{cieliebak2007symplectic}.
    The portion which differs between our setting and the one considered in \cite{charest2017floer} is to prove that the regularizing almost complex structures $\mathcal J^{reg}_V(X\times \CC, J, \theta, E)$ are comeager in $\mathcal \mathcal J_{\tau, V}(X\times \CC, D, J, \theta)$. 

   The standard proof, which we use here, is to show that the universal Cauchy-Riemann operator 
    \begin{align*}
        \bar \partial ^{X\times \CC}:\mathcal B(K)\times \mathcal \mathcal J_{\tau, V}(X\times \CC, D, J, \theta) \to \mathcal E_{X\times \CC}\\
        (u, J)\mapsto \bar \partial^{X\times \CC}_J(u)
    \end{align*}
    is transverse to $0$. Here, $\mathcal B(K)$ is the Banach manifold of disks with boundary on $K$, and $\mathcal E_{X\times \CC}$ is the Banach bundle whose fiber at $u\in \mathcal B$ is sections of $\Omega^{0, 1}(u^*T(X\times \CC))$.
    Given a pair $(u, J)$ we break into two cases: if $\Im(u)$ is disjoint from $\pi^{-1}_\CC(V)$, or if $u\subset \pi^{-1}_\CC(V)$. 
    \begin{enumerate}
        \item In the former case ($\Im(\pi\circ u)\cap V=\emptyset$) the open mapping principle implies that $\bar \partial^{X\times \CC}_J(u)=0$ only when $\pi_\CC(u)$ is constant. From this setting, we can apply \cite[Lemma 4.25]{charest2015floer} to show that $\bar \partial^X: \mathcal B(L^\pm)\times \mathcal J_{\tau}(X, D_X, J_X, \theta)\to \mathcal E_X$ is a submersion . It follows that the map $\bar \partial^{X\times \CC}$ is a submersion as well.
        \item In the latter case ($\Im(\pi\circ u)\cap V\neq \emptyset$),  we fall into the setting described by \cite[Definition 5.5]{cieliebak2007symplectic}. Then \cite[Lemma 5.6]{cieliebak2007symplectic} states the universal Cauchy-Riemann operator is transverse to 0 when restricted to variations of $J$ fixed over the complement of $U=\pi^{-1}(V)$ \footnote{In \cite{cieliebak2007symplectic} there is different notation: their set $V$ is our set $U$.} as long as the image of $u$ is not contained within $U$. 
    \end{enumerate}
    Thus, the universal Cauchy-Riemann operator is transverse to 0. By Sard-Smale the set of regular $J$, $\mathcal J^{reg}_V(X\times \CC, J, \theta, E)$ is comeager.
\end{proof}
\subsection{Compactness }
The above constructions allow us to define regular moduli spaces of $\mathcal P$-pseudoholomorphic treed disks with boundary on a Lagrangian cobordism $K$. Let $X$ be a rational symplectic manifold, $K\subset X\times \CC$ a Lagrangian cobordism, and $D\subset X\times \CC$ chosen so that \cref{lem:estabilizedcobordism} holds. 
We say that perturbation datum is \emph{stabilizing and cobordism admissible} if it is chosen from the neighborhoods $\mathcal J^*_V(X\times \CC, D, J, \theta, E)$. The perturbation system is called \emph{admissible} if it is coherent, regular, stabilizing, and cobordism admissible. 
 We now need the cobordism analog of \cite[Theorem 4.27]{charest2017floer}.
\begin{prop}
    Let $\mathcal P_{\mathbb I}$ be an admissible perturbation system.
    For any $\underline \Gamma$ of $\ind(\underline \Gamma)\leq 1$, the moduli space $\mathcal M_{\mathcal P_{\mathbb I}}(X\times \CC,K,D, \underline \Gamma)$ is compact, whose boundary components are given by Morse flow-line breaking in $K$.
    \label{prop:compactness}
\end{prop}
\begin{proof}
    The only modification needed from \cite[Theorem 4.27]{charest2017floer} is to address the use of Gromov compactness in the setting of $X\times \CC$. 
    We give a brief recap of the argument used by \cite{biran2008lagrangian}. Consider a domain-dependent pseudoholomorphic map $u: \Sigma\to X$. Since we choose domain-dependent almost complex structures from $\mathcal J_{\tau, V}(X\times \CC, \omega+\omega_\CC)$ the map $\pi_\CC \circ u$ is holomorphic for points whose image lies outside of $V$. As a result, we may apply the maximum principle to show that holomorphic disks with boundary on a Lagrangian cobordism $K$  either
    \begin{itemize}
        \item have image contained in $X\times V$ or;
        \item live in a fiber of the projection so that $\pi_\CC(u)$ is constant.
    \end{itemize}
    By choosing a Morse function whose gradient flow points outwards along the ends of the Lagrangian cobordism, we can show that the image of a holomorphic treed disk is contained within $X \times V$  (see \cref{lemma:openmapping} for a full argument). We can therefore apply Gromov compactness.
\end{proof}

By \cite[Theorem 4.19]{charest2015floer} there exists admissible perturbation datum. 

\begin{corollary}
    Let $X$ be a compact rational symplectic manifold. Let $K: L^+\rightsquigarrow L^-$ be a spin and graded Lagrangian cobordism, $h: K\to \RR$ an admissible Morse function. There exists a comeager set of perturbation datum so that $\CF(K, h, \mathcal P)$ defines a filtered $A_\infty$ algebra. 
    \label{cor:pearlymodelexistence}
\end{corollary}

\subsection{Pearly model for Lagrangian Cobordisms}

The main result of this section is the matching of moduli spaces of pseudoholomorphic disks for $L^+$ with those of $K$. 
Given a pre-broken type $\underline \Gamma_P$ for $L^+\subset X$, denote by $(i_*^+)\underline \Gamma_P$ the type for $K\subset X$ whose underlying tree has combinatorial type $\Gamma$ and whose pre-breakings and labels are determined by the identification of $i_*^+\Crit(h^+)\subset \Crit(h)$.
Denote by $\mathbb I^+_+:=i^+_{*}\mathbb T_{L^+}$ the set of pre-broken combinatorial types of $(K, h)$ whose labels come from pushforward of labels on $(L, h^+)$. 
Let $\mathbb I_+$ denote the trees which have outgoing edge labeled by a critical point in $i_*^+\Crit(h^+)\subset \Crit(h)$.
\begin{theorem}[Compatibility of Pearly Model of Cobordism]
    Let $K:L^-\rightsquigarrow L^+$ be a Lagrangian cobordism, and $h: K\to \RR$ be  a compatible Morse function. 
    There exists admissible perturbations systems 
    $\mathcal P$ for $K$ and 
    $(i^+)^*\mathcal P$ for $L^+$   
    so that for any labeled combinatorial type $\underline \Gamma_P\in\mathbb I_+$.
    \begin{align*}
         \mathcal M_{\mathcal P}(X\times \CC,K,D, \underline \Gamma_P)=\left\{ \begin{array}{cc} \mathcal M_{(i^+)^*\mathcal P} (X,L^+,D_X, \underline \Gamma'_P) & \text{ if $\mathbb I^+_+\ni \underline \Gamma_P=i^+_*\underline \Gamma'_P$ }\\ 
            \emptyset & \text{if $\underline \Gamma_P\in \mathbb I_+\setminus \mathbb I^+_+$.}
         \end{array}
         \right.
    \end{align*}
    \label{assum:pearlycompatibility}
\end{theorem}
\begin{proof}
    The admissibility of the pullback perturbation $(i^+)^*\mathcal P$ is proven in \cref{cor:pullbackisadmissible}. The characterization of the moduli spaces are \cref{eq:matchingifmatching,clm:emptyotherwise}.
\end{proof}
An algebraic corollary of this statement is the pearly model equivalent of \cref{prop:morseprojection}

\begin{corollary}
    Let $K: L^+\rightsquigarrow L^-$ be a Lagrangian cobordism. Then the projections 
   \[
       \setlength\mathsurround{0pt}\begin{tikzcd}
           \; & \CF(K, h) \arrow{dl}{\beta^{-}} \arrow{dr}{\beta^+} \\
           \CF(L^-, h^-)& & \CF(L^+, h^+)
       \end{tikzcd}\setlength\mathsurround{.8pt}
   \]
   are filtered $A_\infty$-homomorphisms.
   \label{cor:projectionsareainfinity}
\end{corollary}
\begin{proof}
   The conditions on the moduli spaces imposed by \cref{assum:pearlycompatibility} show whenever $\{x_i\}$ is a sequence of Floer cochains with at least one $x_i$ corresponding to a critical point in $\Crit(h)\setminus i_*^+\Crit(h^+)$, that 
   \[m^k(x_1\tensor \cdots\tensor x_k)\in \Lambda\langle \Crit(h)\setminus i_*^+\Crit(h^+)\rangle \subset \CF(K, h).\]
   This proves that $\Lambda\langle \Crit(h)\setminus i_*^+\Crit(h^+)\rangle $ is an $A_\infty$ ideal, and that the projection $\CF(K, h)/\langle \Crit(h)\setminus i_*^+\Crit(h^+)\rangle$ is an $A_\infty$ homomorphism. 
   It additionally follows from \cref{assum:pearlycompatibility} that the $A_\infty$ product on 
   $\CF(K, h)/\langle \Crit(h)\setminus i_*^+\Crit(h^+)\rangle$ matches the $A_\infty$ structure on $\CF(L^+, h^+)$.

   The same argument holds on the negative end as well.
\end{proof}

\subsubsection{Proof of Compatibility of Moduli Spaces}

\begin{lemma}
    Let  $\mathcal P$ be an admissible perturbation datum for $K$.
    Consider a combinatorial type $\underline \Gamma\in \mathbb I^+_+$ and  $u\in \mathcal M_{ \mathcal P_{\mathbb I_+^+}}(X\times \CC, K, D, \underline \Gamma)$.
    We claim that $\pi_\CC(u)=t^+$.
    \label{lemma:openmapping}
\end{lemma}
\begin{proof} Let $\gamma_e: (s^-_e, s^+_e)\to K$ be the Morse flow lines of $u$, and let $u_v: (D^2_v, \partial D^2_v)\to (X\times \CC, K)$ be the $ \mathcal P_{}$-regular holomorphic disks of $u$, where $v, e$ denote vertices and edges of $\underline \Gamma$ respectively. We make two observations:
\begin{itemize}
\item If $\pi_\CC\circ \gamma(s_+)=t^+$, then $\pi_\CC\circ \gamma(s^-_e)=t^+$ because $\grad h^+$ points away from $t^+$.
\item As the composition  $u\circ \pi_\CC:D^2\to  X\times \CC\to \CC$ is holomorphic outside of $V$ , if there exists a point $z_0\in D^2$ so that $\pi_\CC\circ u_v(z_0)=t^+$, then $\pi_\CC\circ u_v(z)=t^+$ by the maximum principle.
\end{itemize}
At the root of the tree, we have $\pi_\CC(\gamma_{e_0}(s^+_e))=\pi_\CC(x_0)=t^+$. Recursively working upwards through the tree from the root and applying the observations shows that $\pi_\CC(u)=t^+$. 
\end{proof}    
From this, we can draw two conclusions:

Firstly, given $\underline \Gamma_P$ for $L^+\subset X$, and admissible perturbation $\mathcal P(i^+_*\underline \Gamma_P)$ for the corresponding type on $K$, we obtain a domain dependent perturbation of almost complex structure $(i^+)^*\mathcal P(\underline \Gamma_P)$ by using the restriction map $\res: \mathcal J_{\tau, V}(X\times \CC, D, J, \theta)\to  \mathcal J_{\tau}(X, D_X, J_X, \theta)$. 
Every $ \mathcal P$- pseudoholomorphic disk $u$ of type $i^+_*\underline \Gamma_P$ lies completely inside $X\times \{t^+\}$, and (by compatibility of our two perturbation systems) gives rise to a $ (i^+)^*\mathcal P_(\underline \Gamma_P)$-pseudoholomorphic disk $\pi_X(u)$ with boundary on $L^+\subset X$. We can therefore state that (as a set) 
\[\mathcal M_{ \mathcal P}( X\times \CC,K, D, i^+_*\underline \Gamma_P))\subset \mathcal M_{ (i^+)^* \mathcal P}(X, L, D_X, \underline \Gamma_P).\] 
Because the gradient flow of an admissible Morse function points outwards at $t^+$, and  $\mathcal P$ is regular and split in a neighborhood of $X\times \{t^+\}$, $\pi_X(u)$ is a regular disk for $(i^+)^*\mathcal P$.

Similarly, the lift of any  $u\in \mathcal M_{(i^+)^* \mathcal P }(X, L, D_X, \underline \Gamma_P) $ to $u\times\{t^+\}$ is a solution for the $\mathcal P$-perturbed $\bar \partial$ equation.
This shows that 
\begin{equation}
    \mathcal M_{ \mathcal P}( X\times \CC,K, D, i^+_*\underline \Gamma_P)= \mathcal M_{ (i^+)^* \mathcal P}(X, L, D_X, \underline \Gamma_P).
    \label{eq:matchingifmatching}
\end{equation}
Since $\mathcal P$ is regular, $u\times \{t^+\}$ this is a regularly cut-out treed disk, so $\pi_X(u\times \{t^+\})=u$ is regular. We conclude that $(i^+)^*\mathcal P$ is regular perturbation system.

\begin{corollary}
    Let $\mathcal P_{\mathbb T_K}$ be an admissible perturbation system for $K$. Then $(i^+)^*\mathcal P_{\mathbb T_L}$ is admissible perturbation system for $L^+$. 
    \label{cor:pullbackisadmissible}
\end{corollary}

Secondly: let $\mathbb I_+\setminus \mathbb I^+_+$ denote the set of labeled trees with outgoing label in $i_*^+\Crit(h^+)$, but at least one incoming label in $\Crit(h)\setminus i_*^+\Crit(h^+)$.
\begin{claim}
    If $\underline \Gamma_P\in \mathbb I_+\setminus \mathbb I^+_+$, then $\mathcal M(X\times \CC, K, D, h, \underline \Gamma_P)=\emptyset$. 
    \label{clm:emptyotherwise}
\end{claim}
\begin{proof}[Proof of Claim]
    For all $\underline \Gamma_P\in \mathbb I_+\setminus \mathbb I^+_+$, 
     $\underline \Gamma_P$ has an input label which does not belong to $i_*^+\Crit(h^+)$. Suppose for contradiction that we have a pseudoholomorphic treed disk $u$ of type $\underline \Gamma_P$.  Since the output label of $\underline \Gamma_P$ does still belong to $i_*^+\Crit(h^+)$ we may apply \cref{lemma:openmapping} and conclude $\pi_\CC(u)=t^+$. 
    However, this implies that all the labels of $\underline \Gamma_P$ belong to $i_*^+\Crit(h^+)$, a contradiction. Therefore $\mathcal M_{\mathcal P}(K, X\times \CC, D, \underline \Gamma_P)=\emptyset$.
\end{proof}

\printbibliography

@misc{obamaspeech,
  title={ The President's Weekly Address },
  author={ Barack Obama },
  note={ June 6, 2009 }}

@article{cieliebak2007symplectic,
  title={Symplectic hypersurfaces and transversality in Gromov-Witten theory},
  author={Cieliebak, Kai and Mohnke, Klaus and others},
  journal={Journal of Symplectic Geometry},
  volume={5},
  number={3},
  pages={281--356},
  year={2007},
  publisher={International Press of Boston}
}

@article{biran2013fukayacategories,
  title={Lagrangian cobordism and {F}ukaya categories},
  author={Biran, Paul and Cornea, Octav},
  journal={Geometric and Functional Analysis},
  volume={24},
  number={6},
  pages={1731--1830},
  year={2014},
  publisher={Springer}
}

@article{cornea2019lagrangian,
  title={Lagrangian cobordism and metric invariants},
  author={Cornea, Octav and Shelukhin, Egor},
  journal={Journal of Differential Geometry},
  volume={112},
  number={1},
  pages={1--45},
  year={2019},
  publisher={Lehigh University}
}

@article{biran2013lagrangian,
  title={Lagrangian cobordism. {I}},
  author={Biran, Paul and Cornea, Octav},
  journal={Journal of the American Mathematical Society},
  volume={26},
  number={2},
  pages={295--340},
  year={2013}
}

@article{auroux2009special,
  title={Special {L}agrangian fibrations, wall-crossing, and mirror symmetry},
  author={Auroux, Denis},
  journal={Surveys in Differential Geometry},
  volume={13},
  number={1},
  pages={1--48},
  year={2008},
  publisher={International Press of Boston}
}

@article{fukaya1997zero,
  title={Zero-loop open strings in the cotangent bundle and {M}orse homotopy},
  author={Fukaya, Kenji and Oh, Yong-Geun},
  journal={Asian Journal of Mathematics},
  volume={1},
  number={1},
  pages={96--180},
  year={1997},
  publisher={International Press of Boston}
}

@article{strominger1996mirror,
  title={Mirror symmetry is {T}-duality},
  author={Strominger, Andrew and Yau, Shing-Tung and Zaslow, Eric},
  journal={Nuclear Physics B},
  volume={479},
  number={1-2},
  pages={243--259},
  year={1996},
  publisher={Elsevier}
}

@article{markl2006transferring,
  title={Transferring $A_\infty $(strongly homotopy associative) structures},
  author={Markl, Martin},
  journal={Proceedings of the 25th Winter School" Geometry and Physics"},
  pages={139--151},
  year={2006},
  publisher={Circolo Matematico di Palermo}
}

@article{palmer2019invariance,
  title={Invariance of immersed {F}loer cohomology under {L}agrangian surgery},
  author={Palmer, Joseph and Woodward, Chris},
  journal={arXiv preprint arXiv:1903.01943},
  year={2019}
}

@article{charest2015floer,
	title={Floer theory and flips},
	author={Charest, Fran{\c{c}}ois and Woodward, Chris T},
	journal={Memoirs of the American Mathematical Society },
	year={To Appear}
}

@article{donaldson1996sympsubmanifold,
author = {S. K. Donaldson},
title = {{Symplectic submanifolds and almost-complex geometry}},
volume = {44},
journal = {Journal of Differential Geometry},
number = {4},
publisher = {Lehigh University},
pages = {666 -- 705},
year = {1996},
doi = {10.4310/jdg/1214459407},
URL = {https://doi.org/10.4310/jdg/1214459407}
}

@article{auroux2001symplectic,
  title={Symplectic hypersurfaces in the complement of an isotropic submanifold},
  author={Auroux, Denis and Gayet, Damien and Mohsen, Jean-Paul},
  journal={Mathematische Annalen},
  volume={321},
  number={4},
  pages={739--754},
  year={2001},
  publisher={Springer}
}

@article{auroux1997asymptotically,
  title={Asymptotically holomorphic families of symplectic submanifolds},
  author={Auroux, Denis},
  journal={Geometric \& Functional Analysis GAFA},
  volume={7},
  number={6},
  pages={971--995},
  year={1997},
  publisher={Springer}
}

@article{charest2017floer,
  title={Floer trajectories and stabilizing divisors},
  author={Charest, Fran{\c{c}}ois and Woodward, Chris},
  journal={Journal of Fixed Point Theory and Applications},
  volume={19},
  number={2},
  pages={1165--1236},
  year={2017},
  publisher={Springer}
}

@misc{hicks2021lagrangian,
      title={Lagrangian cobordisms and Lagrangian surgery}, 
      author={Jeff Hicks},
      year={2021},
      arxiv={2102.10197},
      archivePrefix={arXiv},
      primaryClass={math.SG}
}

@ARTICLE {oh1995RiemannHilbert,
	author  = "Yong-Geun Oh",
	title   = "{Riemann-Hilbert} problem and application to the perturbation theory of analytic discs",
	journal = "Kyungpook Math Journal",
	year    = "1995",
	volume  = "35",
	pages   = "39-76"
}

@article{fukaya2010cyclic,
	title={Cyclic symmetry and adic convergence in {Lagrangian Floer} theory},
	author={Fukaya, Kenji},
	journal={Kyoto Journal of Mathematics},
	volume={50},
	number={3},
	pages={521--590},
	year={2010},
	publisher={Kyoto University}
}

@book{kontsevich2001homological,
  title={Homological mirror symmetry and torus fibrations},
  author={Kontsevich, Maxim and Soibelman, Yan},
  year={2001},
  publisher={World Scientific}
}

@article{lipreliminary,
	title={{$A_\infty$} Structures from {M}orse Trees with Pseudoholomorphic Disks},
	author={Li, Jiayong and Wehrheim, Katrin},
  journal={Preliminary Draft},
  year={2014}
}

@article{usher2016persistent,
  title={Persistent homology and Floer--Novikov theory},
  author={Usher, Michael and Zhang, Jun},
  journal={Geometry \& Topology},
  volume={20},
  number={6},
  pages={3333--3430},
  year={2016},
  publisher={Mathematical Sciences Publishers}
}

@article{haug2015lagrangian,
  title={Lagrangian antisurgery},
  author={Haug, Luis},
  journal={Mathematical Research Letters},
  volume={27},
  number={5},
  pages={1423--1464},
  year={2020},
  publisher={International Press of Boston}
}

@article{auroux2007mirror,
  title={Mirror symmetry and {T}-duality in the complement of an anticanonical divisor},
  author={Auroux, Denis},
  journal={Journal of G{\"o}kova Geometry Topology},
  volume={1},
  pages={51--91},
  year={2007}
}

@article{cho2004holomorphic,
  title={Holomorphic discs, spin structures, and Floer cohomology of the Clifford torus},
  author={Cho, Cheol-Hyun},
  journal={International mathematics research notices},
  volume={2004},
  number={35},
  pages={1803--1843},
  year={2004},
  publisher={Hindawi Publishing Corporation}
}

@inproceedings{kontsevich1994homological,
  title={Homological algebra of mirror symmetry},
  author={Kontsevich, Maxim},
  booktitle={Proceedings of the international congress of mathematicians},
  pages={120--139},
  year={1995},
  organization={Springer}
}

@article{pascaleff2017wall,
  title={The wall-crossing formula and Lagrangian mutations},
  author={Pascaleff, James and Tonkonog, Dmitry},
  journal={Advances in Mathematics},
  volume={361},
  pages={106850},
  year={2020},
  publisher={Elsevier}
}

@book{fukaya2010lagrangian,
  title={Lagrangian intersection {F}loer theory: anomaly and obstruction, Part {I}},
  author={Fukaya, Kenji and Oh, Yong-Geun and Ohta, Hiroshi and Ono, Kaoru},
  volume={41},
  year={2010},
  publisher={American Mathematical Soc.}
}

@article{arnol1980lagrange,
  title={Lagrange and {L}egendre cobordisms. {I}},
  author={Arnol'd, Vladimir Igorevich},
  journal={Functional Analysis and Its Applications},
  volume={14},
  number={3},
  pages={167--177},
  year={1980},
  publisher={Springer}
}

@misc{biran2008lagrangian,
      title={Lagrangian Quantum Homology}, 
      author={Paul Biran and Octav Cornea},
      year={2008},
      arxiv={0808.3989},
      archivePrefix={arXiv},
      primaryClass={math.SG}
}

@misc{rizell2018refined,
      title={Refined disk potentials for immersed Lagrangian surfaces}, 
      author={Georgios Dimitroglou Rizell and Tobias Ekholm and Dmitry Tonkonog},
      year={2020},
      arxiv={1806.03722},
      archivePrefix={arXiv},
      primaryClass={math.SG}
}

@article{kadeishvili1980homology,
  title={On the homology theory of fibre spaces},
  author={Kadeishvili, Tornike V},
  journal={Russian Mathematical Surveys},
  volume={35},
  number={3},
  pages={231},
  year={1980},
  publisher={IOP Publishing}
  }

\Addresses
\end{document}